\documentclass[11pt,bull-l]{amsart}
\usepackage[top = 0.9in, bottom = 0.9in, left =1in, right = 1in]{geometry}
\usepackage{geometry}
\usepackage{array}
\usepackage[utf8]{inputenc}
\usepackage{hyperref}
\hypersetup{
  colorlinks   = true, %Colours links instead of ugly boxes
  urlcolor     = red, %Colour for external hyperlinks
  linkcolor    = blue, %Colour of internal links
  citecolor   = blue %Colour of citations, could be ``red''
}

\usepackage{listings}
\usepackage{multimedia} % to embed movies in the PDF file
\usepackage{graphicx}
\usepackage{comment}
\usepackage[english]{babel}
\usepackage{amsfonts,amscd,amssymb,amsmath,mathrsfs}
\usepackage{wrapfig}
\usepackage{multirow}
\usepackage{verbatim,cite}
\usepackage{float}
\usepackage{cancel}
\usepackage{caption}
\usepackage{subcaption}
\usepackage{mathdots}
\usepackage{bbm}
\usepackage{tikz}
\usepackage{xcolor}
\usepackage[algoruled]{algorithm2e}
\usepackage{paper-defs}

\renewcommand{\vec}{\mathbf}
\newcommand{\scrS}{\mathcal S}
\newcommand{\scrT}{\mathcal T}
\newcommand{\Julia}{\texttt{Julia}}

\title{The Akhiezer iteration and inverse-free solvers for Sylvester matrix equations}
\author{Cade Ballew}
\address{University of Washington, Seattle, WA}
\email{ballew@uw.edu}

\author{Thomas Trogdon}
\address{University of Washington, Seattle, WA}
\email{trogdon@uw.edu}

\author{Heather Wilber}
\address{University of Washington, Seattle, WA}
\email{hdw27@uw.edu}
\date{}

\begin{document}

\begin{abstract}
Two inverse-free iterative methods are developed for solving Sylvester matrix equations when the spectra of the coefficient matrices are on, or near, known disjoint subintervals of the real axis. Both methods use the recently-introduced Akhiezer iteration: one to address an equivalent problem of approximating the matrix sign function applied to a block matrix and the other to directly approximate the inverse of the Sylvester operator.  In each case this results in provable and computable geometric rates of convergence. When the right-hand side matrix is low rank, both methods require only low-rank matrix-matrix products. Relative to existing approaches, the methods presented here can be more efficient and require less storage when the coefficient matrices are dense or otherwise costly to invert. Applications include solving partial differential equations and computing Fr\'echet derivatives.
\end{abstract}
\maketitle
	
% possible examples: 
% convection-diff with low order accuracy
%
% pseudospectral shattering to improve conditioning?
	
\section{Introduction} \label{sect:intro}
In numerical linear algebra, a fundamental problem is computing the solution of the Sylvester matrix equation, which is given by
\begin{equation} 
\label{eq:Genmat}
\bX \bA - \bB \bX = \bC, \quad \bX, \bC \in \mathbb{C}^{m \times n}, \quad \bA \in \mathbb{C}^{n \times n}, \quad \bB \in \mathbb{C}^{m \times m},
\end{equation} 
with $\bX$ as the unknown. It appears in applications related to model reduction and stability analysis for large-scale dynamical systems~\cite{antoulas2005approximation, gugercin2003modified, penzl1999cyclic}, eigenvalue assignment for vibrating structures~\cite{brahma2009optimization}, noise reduction in image processing~\cite{calvetti1996application}, various tasks in control systems and signal processing~\cite{benner2004solving,laub1985numerical, gajic2008lyapunov, van1991structured}, and the solving of partial differential equations (PDEs)~\cite{boulle2020computing,ADIPoisson, massei2019fast,slomka2018stokes}. We consider the case where the \textit{coefficient matrices} $\bA$ and $\bB$ are diagonalizable with spectra each contained in disjoint regions on or near the real axis.  

Iterative methods for solving these equations have been a central focus of research for several decades (see~\cite{simoncini2016computational} and the references therein), and this has co-evolved alongside the broader development of rational Krylov methods~\cite{berljafa2017rkfit, beckermann2021low, ruhe1984rational} and related iterative solvers that fundamentally involve approximations via rational functions. In settings where $\bC$ is low rank and $\bA$ and $\bB$ are banded, sparse, or otherwise structured so that shifted inverts are inexpensive, rational-based iterative methods can be extraordinarily effective. However, when $\bA$ and $\bB$ are dense or otherwise costly to invert, such methods will become prohibitively expensive, offering no advantage over the $\OO(m^3 + n^3)$ direct solver of Bartels and Stewart~\cite{BartelsStewart}, which is itself too costly in the large-scale setting. A natural alternative one might look for is an iterative method based on polynomial approximation. This avoids the cost of inversion, instead requiring low-rank matrix-matrix products when $\bC$ is low rank. Existing methods based on such ideas (e.g., polynomial Krylov or block polynomial Krylov methods) incur a cost per iteration that grows cubically with respect to the iteration number due to the orthogonalization step~\cite{simoncini2007new}, often suffer from storage constraints, as they require storing all previous iterations, and their convergence rates are often too slow to capture the singular value decay of the true solution \cite{KressnerMemory08,penzl1999cyclic}. Such methods also require residual-monitoring, while the methods that we introduce here do not.

With the limitations of competing methods in mind, we consider a successful method for solving~\eqref{eq:Genmat} when $\bC$ is low rank and $\bA$ and $\bB$ are dense to be one that requires $\oo(n^3+m^3)$ arithmetic operations and $\oo(nm)$ stored entries. Direct and rational-based solvers typically fail at the former, while standard polynomial-based methods typically fail at the latter.

Here, we introduce two new polynomial-based iterative methods for solving Sylvester matrix equations that achieve both criteria. Unlike other polynomial-based methods, these methods rely on the use of so-called \textit{Akhiezer polynomials} \cite[Chapter 10]{akhiezer}, a generalization of the Chebyshev polynomials that are orthogonal with respect to a weighted inner product over domains consisting of disjoint intervals on the real axis. Because these polynomials have good convergence properties on cut domains\footnote{As with the Chebyshev polynomials, Akhiezer polynomial series converge at the same rate as the best polynomial approximation \cite{WalshBook}; however, the minimax property of the Chebyshev polynomials is only retained under certain conditions on the domain endpoints and the degree of the polynomial \cite[Chapter 3]{FischerBook}.}, we are able to obtain rates of convergence in our algorithms that are geometric with the degree of the polynomial (and number of iterations). The result is a collection of inverse-free iterative methods with provable, and often explicit, geometric rates of convergence (see Figure~\ref{fig:scal_approx}). When $\bC$ is low rank, as is often the case in practice, the methods developed here only require low-rank matrix-matrix products and some low-dimensional QR and singular value decompositions. They do not require storage and orthogonalization against a growing collection of basis vectors. Since we have error bounds, monitoring of a residual or other convergence indicators is also not required. When $\bX$ is numerically of low rank, these methods can be used to construct an approximate solution in low-rank form. 

A particularly attractive feature of the methods that we introduce here is that their convergence rates are easily determined a priori. That is, denoting the true solution to \eqref{eq:Genmat} by $\bX_*$, the $k$th iterate $\bX_k$ satisfies
\begin{equation}\label{eq:gen_conv_rate}
\|\bX_k-\bX_*\|\leq C\varrho^{-k},
\end{equation}
in an appropriate norm for some constants $C>0$ and $0<\varrho^{-1}<1$. Let $\sigma(\vec M)$ denote the set of eigenvalues of a linear operator $\vec M$, i.e., the spectrum. If $\sigma(\bA)\subset[\beta,1]$ and $\sigma(\bB)\subset[-1,-\beta]$ for $\beta\in(0,1)$, Method 1 (developed in Section~\ref{sect:alg}) has convergence rate
\begin{equation}\label{eq:conv_method1}
\varrho^{-1}(\beta) =\sqrt{\frac{1-\beta}{1+\beta}},
\end{equation}
while Method 2 (developed in Section~\ref{sect:Method2}) has convergence rate
\begin{equation}\label{eq:conv_method2}
\varrho^{-1}(\beta)=\frac{1-\sqrt{\beta}}{1+\sqrt{\beta}}.
\end{equation}
In \cite[Equation 3.1]{PolyKrylovError}, Simoncini and Druskin give an error bound of the form \eqref{eq:gen_conv_rate} for a polynomial Krylov method (e.g., \cite[Algorithm 5]{simoncini2016computational}) in the Lyapunov case\footnote{See \cite{KressnerT10} for more general error bounds. Our methods should also be able to provide error bounds for polynomial Krylov methods, and this will be explored in future work.} $\bB=-\bA^T$, $\bC=\bU\bU^T$ with
\begin{equation}\label{eq:conv_krylov}
\varrho^{-1}(\beta)=\frac{\sqrt{\beta+1}-\sqrt{2\beta}}{\sqrt{\beta+1}+\sqrt{2\beta}}.
\end{equation}
% \begin{remark}
% ; however, the convergence rate bounds for Methods 1 and 2 that we derive below apply to general Sylvester equations~\eqref{eq:Genmat}
% \end{remark}
In Figure~\ref{fig:conv_rates}, we plot these convergence rate bounds, as a function of $\beta$, as well as the storage required for a test problem. We observe that both of our methods exhibit slower convergence than the polynomial Krylov method, although Method 2 is quite competitive. As we demonstrate below, Method 2 can be comparable to a polynomial Krylov method in runtime, as the iterations of such a method require orthogonalization and are therefore more expensive. On the other hand, run for enough iterations, both of our methods outperform a generic polynomial Krylov method (\cite[Algorithm 5]{simoncini2016computational}) in terms of the maximum required storage. Despite having the slowest convergence rate, Method 1 tends to perform the best in this sense.
\begin{figure}
    \centering
    \begin{subfigure}{0.495\linewidth}
		\centering
		\includegraphics[width=\linewidth]{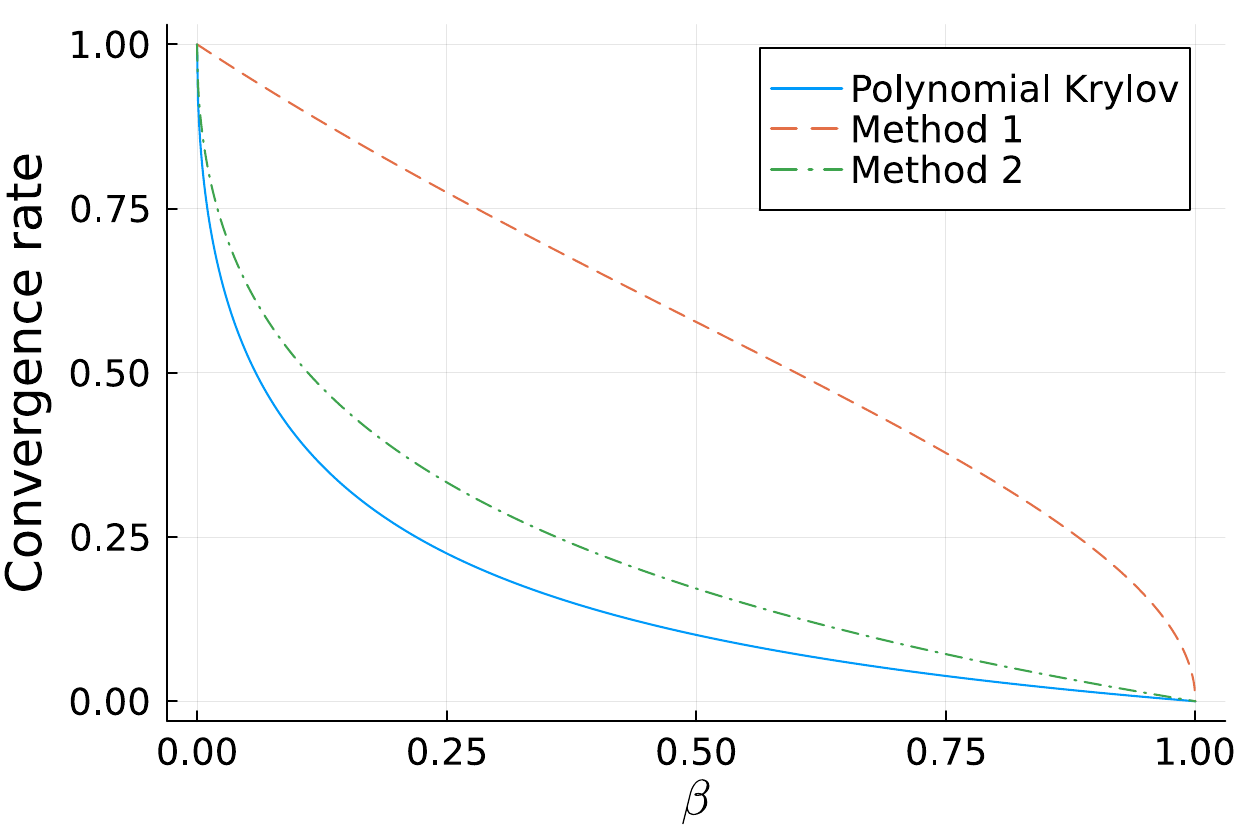}
	\end{subfigure}
	\begin{subfigure}{0.495\linewidth}
		\centering
		\includegraphics[width=\linewidth]{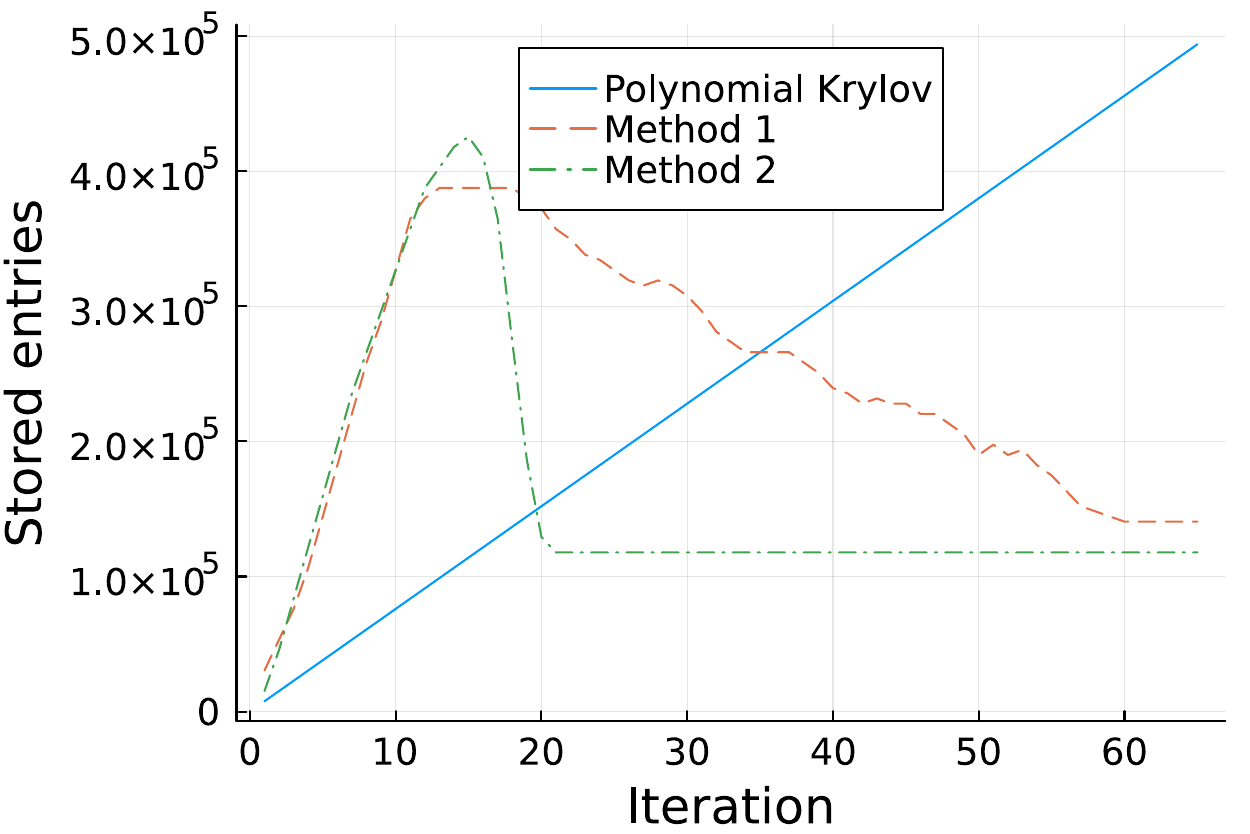}
	\end{subfigure}
    \caption{Left: Convergence rate bounds \eqref{eq:conv_krylov},\eqref{eq:conv_method1},\eqref{eq:conv_method2} for solving~\eqref{eq:Genmat} with a polynomial Krylov method, Method 1, and Method 2, respectively, when $\bB=-\bA^T$, $\bC=\bU\bU^T$, and $\sigma(\bA)\subset[\beta,1]$, plotted as a function of $\beta$.
    Right: Number of stored matrix entries required at each iteration of each method applied to the problem of Figure~\ref{fig:unweight_rank} described in Section~\ref{sect:low-rank}, with $\bA,\bB$ dense.}
    \label{fig:conv_rates}
\end{figure}

\subsection{An approach based on matrix functions}
The two primary classes of iterative methods developed for solving~\eqref{eq:Genmat} are those based on Krylov subspace projections, such as the block polynomial Krylov method, extended Krylov method~\cite{simoncini2007new} and the rational Krylov subspace method~\cite{beckermann2011error, druskin2011analysis}, and those based around directly using rational approximations to construct low-rank solutions, including the cyclic Smith method~\cite{penzl1999cyclic, sabino2007solution, smith1968matrix}, the low-rank alternating direction implicit (LR-ADI) method~\cite{benner2009adi, li2002low}, and their variants. A distinct idea, first introduced by Roberts in~\cite{roberts1980linear} in the context of solving matrix Riccati equations, is based instead on approximating the matrix sign function~\cite{benner1999solving, denman1976matrix, higham2008functions}. This is also the core idea animating Method 1 below. We define the sign function somewhat unconventionally as follows: 

\begin{definition} \label{def:signfun}
Let $U_\bA$ and $U_\bB$ denote two disjoint sets in $\mathbb{C}$. Then, the function ${\rm sign}(z) = {\rm sign}(z;U_\bA, U_\bB)$ is defined on $U_\bA \cup U_\bB$ as 
\begin{equation*}
{\rm sign}(z) = \begin{cases}
		\phantom{-}1,& z \in U_\bA,\\
		-1,& z\in U_\bB.
	\end{cases}
	\end{equation*}
\end{definition}
When the context is clear, we simplify the notation and refer to ${\rm sign}(z)$ without explicitly listing the sets it is defined on. If $\sigma(\vec A) \subset U_\bA$ and $\sigma(\vec B) \subset U_\bB$, then using the definitions for matrix functions we develop below,
 \begin{equation} \label{eq:signblock}
	\sign \underbrace{\begin{pmatrix}\bA&\bzero \\ \bC&\bB\end{pmatrix}}_{\bH}=\begin{pmatrix}\bI_n&\bzero \\ \bX&\bI_m\end{pmatrix}\begin{pmatrix}\bI_n&\bzero \\ \bzero&-\bI_m\end{pmatrix}\begin{pmatrix}\bI_n&\bzero \\ -\bX&\bI_m\end{pmatrix}=\begin{pmatrix}\bI_n&\bzero \\ 2\bX&-\bI_m\end{pmatrix}.
	\end{equation}
Using~\eqref{eq:signblock}, the solution $\bX$ of \eqref{eq:Genmat} can be obtained by computing the matrix $\frac{1}{2}\sign(\bH)$ and isolating the lower left block \cite[Section 2.4]{higham2008functions}. To get an inverse-free iterative method for computing $\bX$ from this, we construct an approximation 
\begin{align}\label{eq:sign}
 \sign(z) \approx \sum_{j = 1}^m \alpha_j p_j(z),
\end{align}
where $\left(p_j(z)\right)_{j=0}^\infty$ are the Akhiezer polynomials, and are therefore orthogonal with respect to a special inner product. An efficient numerical method for evaluating these polynomials (and computing the coefficients $\alpha_j$ in the expansion) was recently introduced in~\cite{RHISM,AkhIter}. In principle, one can simply evaluate \eqref{eq:sign} at $z = \bH$ and extract the appropriate block to create an approximate solution to~\eqref{eq:Genmat}. Evaluating each $p_j(\bH)$ directly is expensive since the dimensions of $\bH$ are double those of the original problem, and we seek to avoid this. Using the block triangular structure and the three-term recurrence satisfied by the Akhiezer polynomials, we develop a recursion that describes the evolution of the lower left block as $j$ grows, so that in most cases\footnote{The cost of the compression step dominates if the cost of matrix-vector products is $\oo(m+n)$.}, the resulting method only incurs a per-iteration computational cost on the same order as the cost of matrix-vector products with $\bA$ and $\bB$. Under our convergence heuristics, $\OO\left(\log(m+n)\right)$ iterations are required to satisfy an error tolerance, and therefore the overall cost of the low-rank variation of the method when $\bA$ and $\bB$ are dense is $\OO(m^2\log m+n^2\log n)$. 
 % Furthermore, as we demonstrate in Figure~\ref{fig:timings}, our method often has a shorter runtime than competing methods, even for small problems.

The second method that we introduce functions largely in the same manner. Rather than passing to a sign function, we directly approximate the inverse of the Sylvester operator $S_{\bA,\bB}\bY=\bY\bA-\bB\bY$ via
\begin{equation*}
\frac{1}{z}\approx\sum_{j=1}^m\hat\alpha_jp_j(z),
\end{equation*}
where $\left(p_j(z)\right)_{j=0}^\infty$ are again the Akhiezer polynomials. 
% We do so by passing to the vectorized form of the Sylvester operator $\bA^T\otimes\bI_m-\bI_n\otimes\bB$, but because the Akhiezer iteration requires only matrix-vector multiplications, we circumvent the need to form this $mn\times mn$ matrix. 
Again, the resulting method incurs a per-iteration computational cost on the same order as the cost of matrix-vector products with $\bA$ and $\bB$ in most cases, requiring $\OO\left(\log(m+n)\right)$ iterations to satisfy an error tolerance with an overall cost of $\OO(m^2\log m+n^2\log n)$  when $\bA$ and $\bB$ are dense. Both methods that we introduce allow for compression at each step. Empirically, this, in combination with the three-term Akhiezer polynomial recurrence, ensures that $\oo(nm)$ entries need to be stored when $\bC$ is low rank. These methods can also be applied when the spectra of $\bA$ and $\bB$ lie in or near an arbitrary number of real interval \`a la \cite{AkhIter}.

Inverse-free iterative solvers can be advantageous in settings where the coefficient matrices are dense, where fast matrix-vector product routines for these matrices are available, and in settings where coefficient matrices are not directly accessible and only their action can be observed. We view the Akheizer-based method presented here as an analogue to the Chebyshev iteration for solving linear systems. The Chebyshev iteration can be advantageous over the conjugate gradient method or GMRES since it only requires short recurrences and avoids orthogonalization steps, and the Akheizer-based approach is advantageous over polynomial Krylov methods in the same way.
Our hope is that the availability of a new inverse-free method will prompt further developments, including investigations into other use-cases and the analysis and development of analogous polynomial families for more general domains. As we show in Section~\ref{sect:other_apps}, Method 1 also supplies a new approach for related tasks, including the evaluation of the Fr\'echet derivative of matrix functions and the solving of other matrix equations. There is also a natural connection to the general evaluation of functions of matrices, which is thoroughly discussed in~\cite{AkhIter}. 
% add some specific examples

Figure~\ref{fig:scal_approx} compares the errors in approximating the sign function via Akhiezer polynomials with the approximation error achieved via optimal rational approximation~\cite[Chapter~9]{akhiezer} (see also~\cite{nakatsukasa2016computing, trefethen2024computation}) on various subsets of $[-1, 1]$. While we cannot expect to achieve the optimal rational convergence rate, the convergence rate is fast enough to yield a competitive algorithm in regimes where inversion is costly. Comparisons with other inverse-free methods are shown in Section~\ref{sect:applications}.
%In Figure~\ref{fig:scal_approx}, we plot errors in approximating the sign function on various subsets of $[-1,1]$. We observe that depending on how close the two intervals are, the Akhiezer polynomial approximation may perform decently relative to the optimal rational approximation. 
\begin{figure}
	\centering
	\begin{subfigure}{0.495\linewidth}
		\centering
		\includegraphics[width=\linewidth]{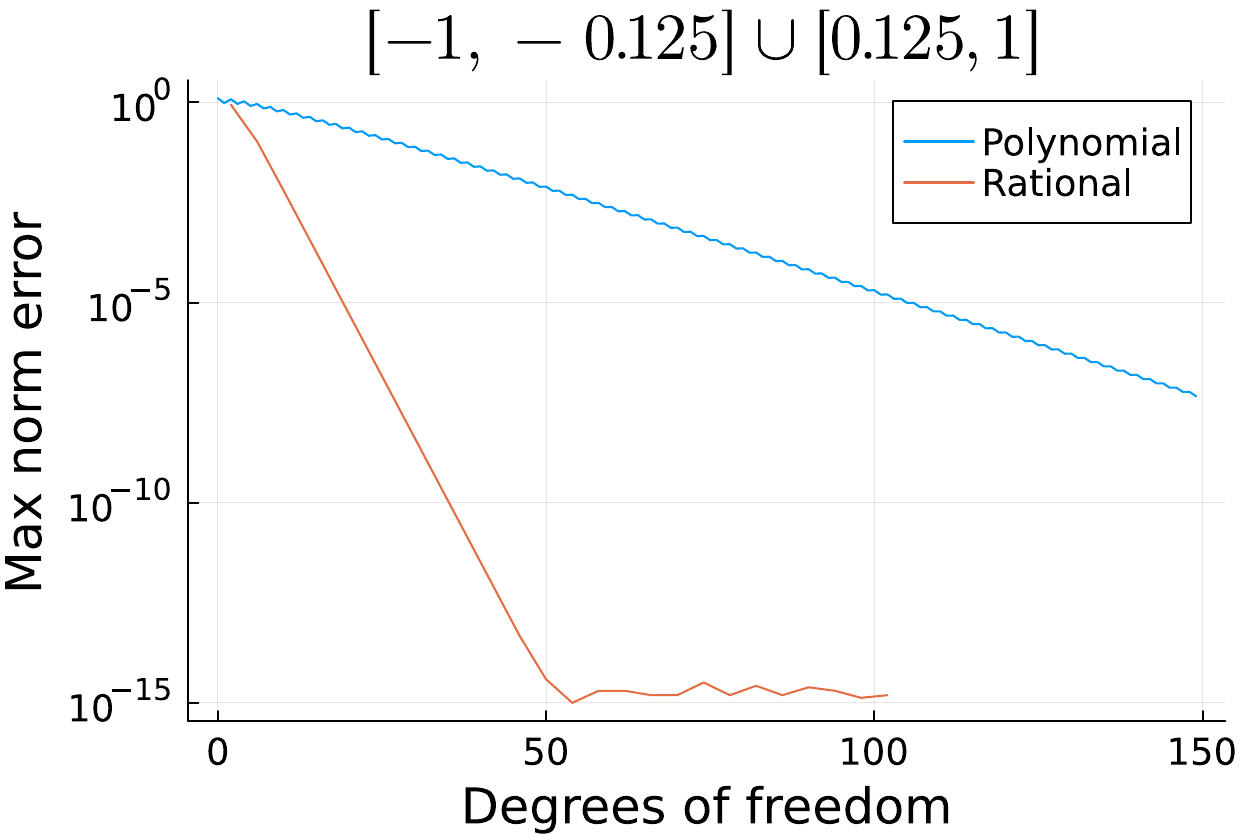}
	\end{subfigure}
	\begin{subfigure}{0.495\linewidth}
		\centering
		\includegraphics[width=\linewidth]{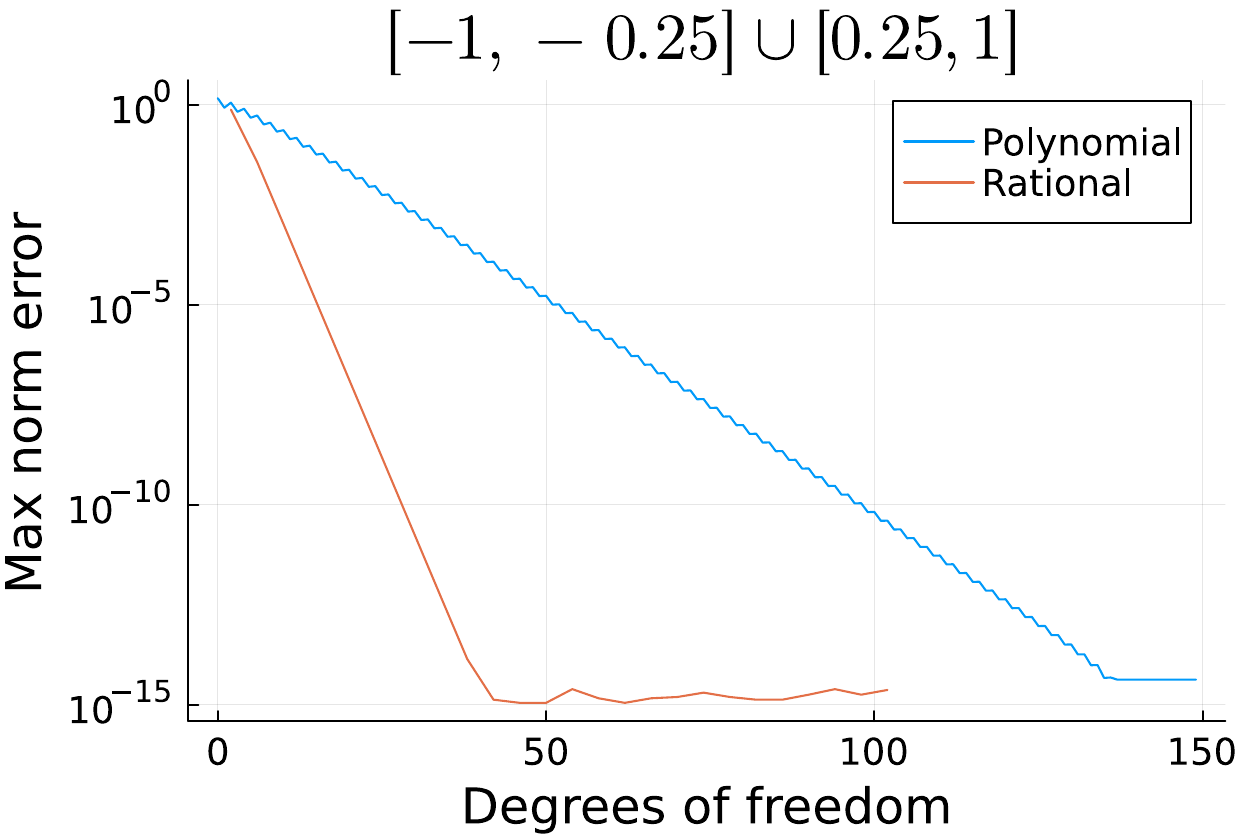}
	\end{subfigure}
    \begin{subfigure}{0.495\linewidth}
		\centering
		\includegraphics[width=\linewidth]{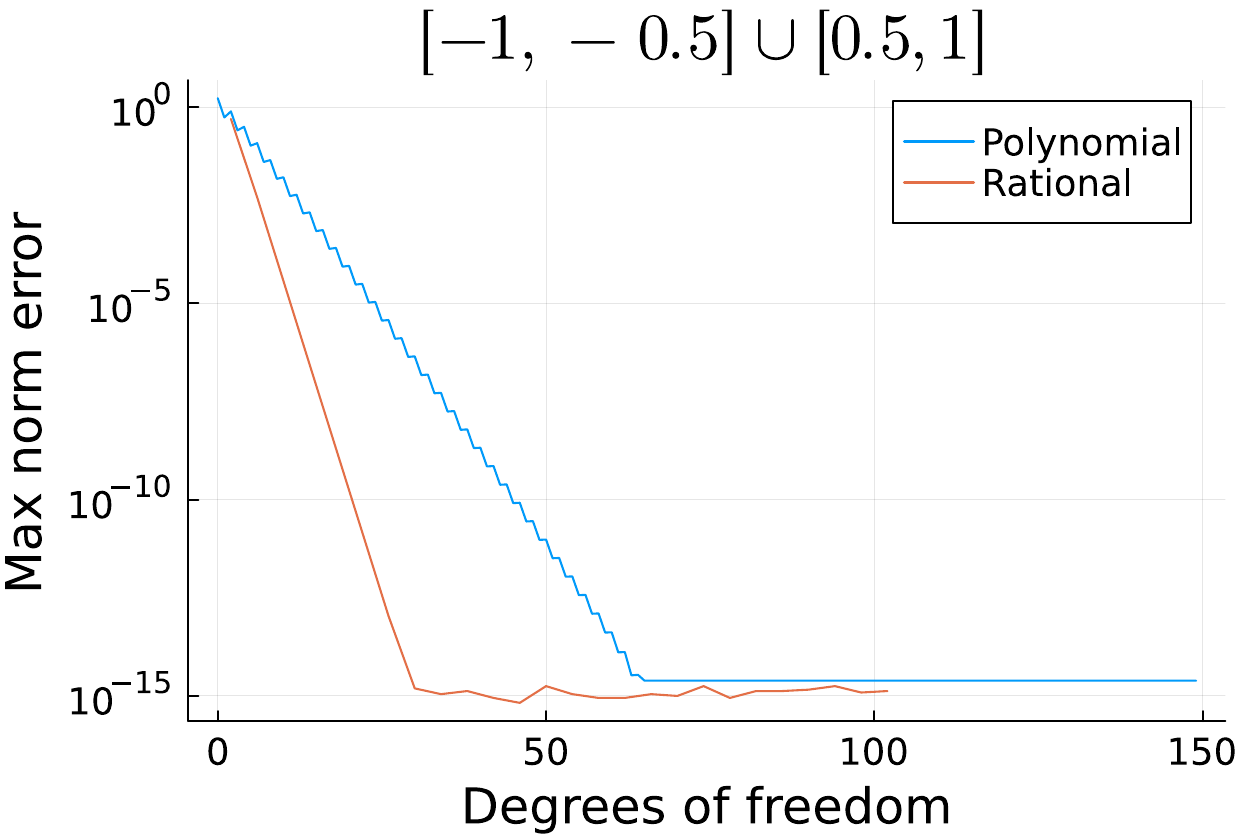}
	\end{subfigure}
    \begin{subfigure}{0.495\linewidth}
		\centering
		\includegraphics[width=\linewidth]{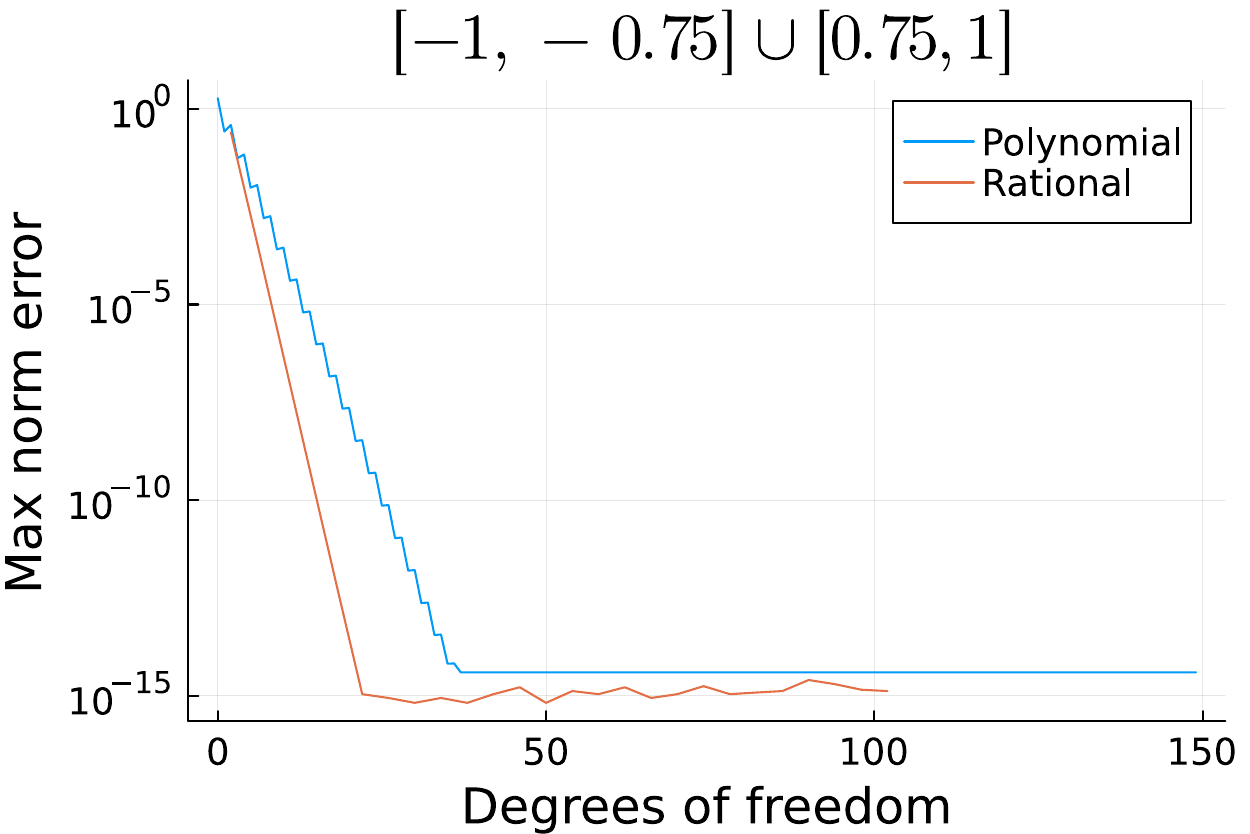}
	\end{subfigure}
\caption{Error plots of Akhiezer polynomial series and optimal Zolotarev rational approximations to the sign function on different cut domains.}
\label{fig:scal_approx}
\end{figure}
Furthermore, this convergence rate can be determined a priori using only the intervals constituting the cut domain. It is governed by the level curves of the function $\ex^{\re\mathfrak g(z)}$, where $\mathfrak g(z)$ is the classical Green's function with pole at infinity defined in Section~\ref{sect:param} \cite{WalshBook,FischerBook}. The convergence rate of an approximation to a function $f$ is governed by the largest value of $c$ such that $f$ can be analytically extended to the level set $\re\mathfrak g(z) < c$ \cite{WalshBook}.

For the sign function, this maximal level set will correspond to a value $c^*$ such that the level curves $\re\mathfrak g(z) = c^*$ around each interval intersect. Thus, as $c$ crosses $c^*$, we see a bifurcation of a connected level curve to a disconnected level curve, or vice versa. In Figure~\ref{fig:bernstein}, we demonstrate this. We describe an inexpensive method to compute this intersection point, and thus the rate of convergence of Method 1, in Section~\ref{sect:param}. These level curves are higher-genus analogs to Bernstein ellipses, as they govern the convergence rate of Akhiezer polynomial series on cut domains in the same manner that Bernstein ellipses govern the convergence rate of Chebyshev polynomial series on a single interval. Plots of these level curves for more general domains can be found in~\cite{EmbreeT99}.
\begin{figure}
	\centering
	\begin{subfigure}{0.495\linewidth}
		\centering
		\includegraphics[width=\linewidth]{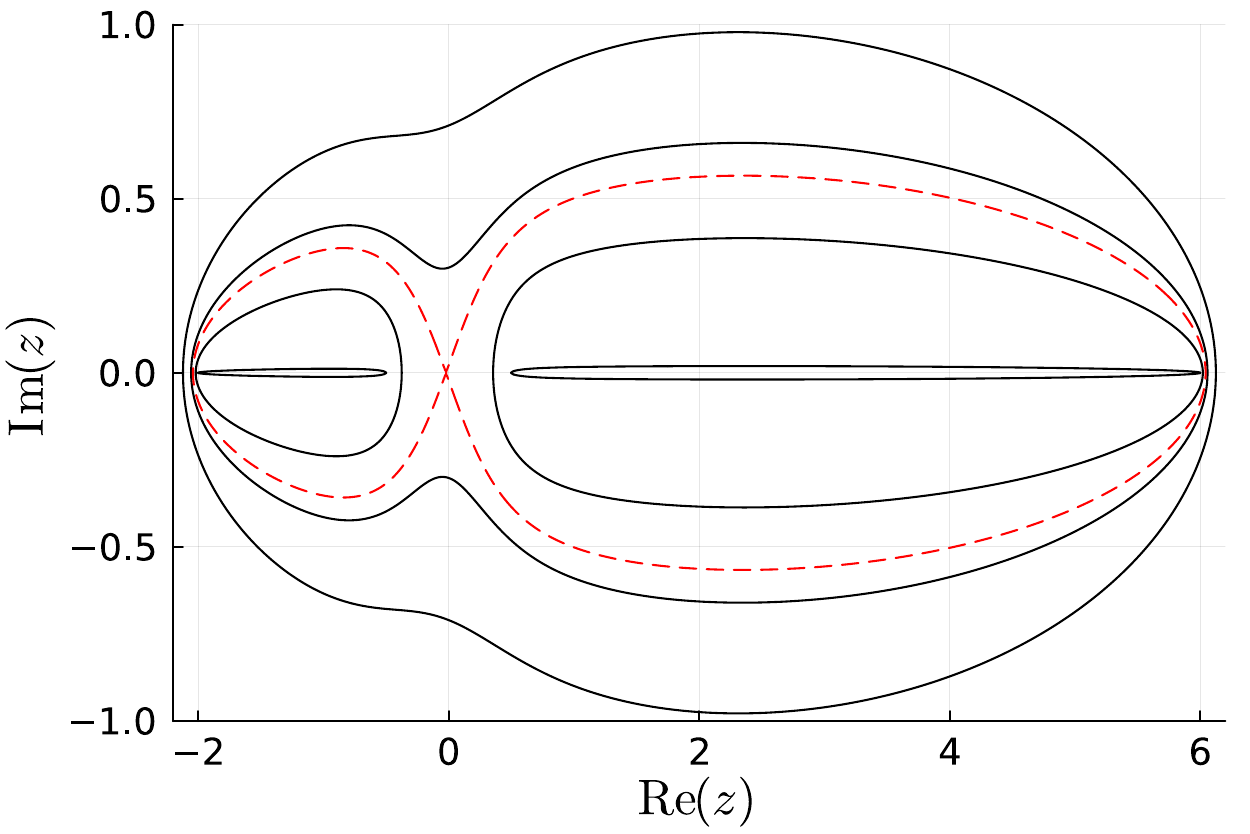}
	\end{subfigure}
	\begin{subfigure}{0.495\linewidth}
		\centering
		\includegraphics[width=\linewidth]{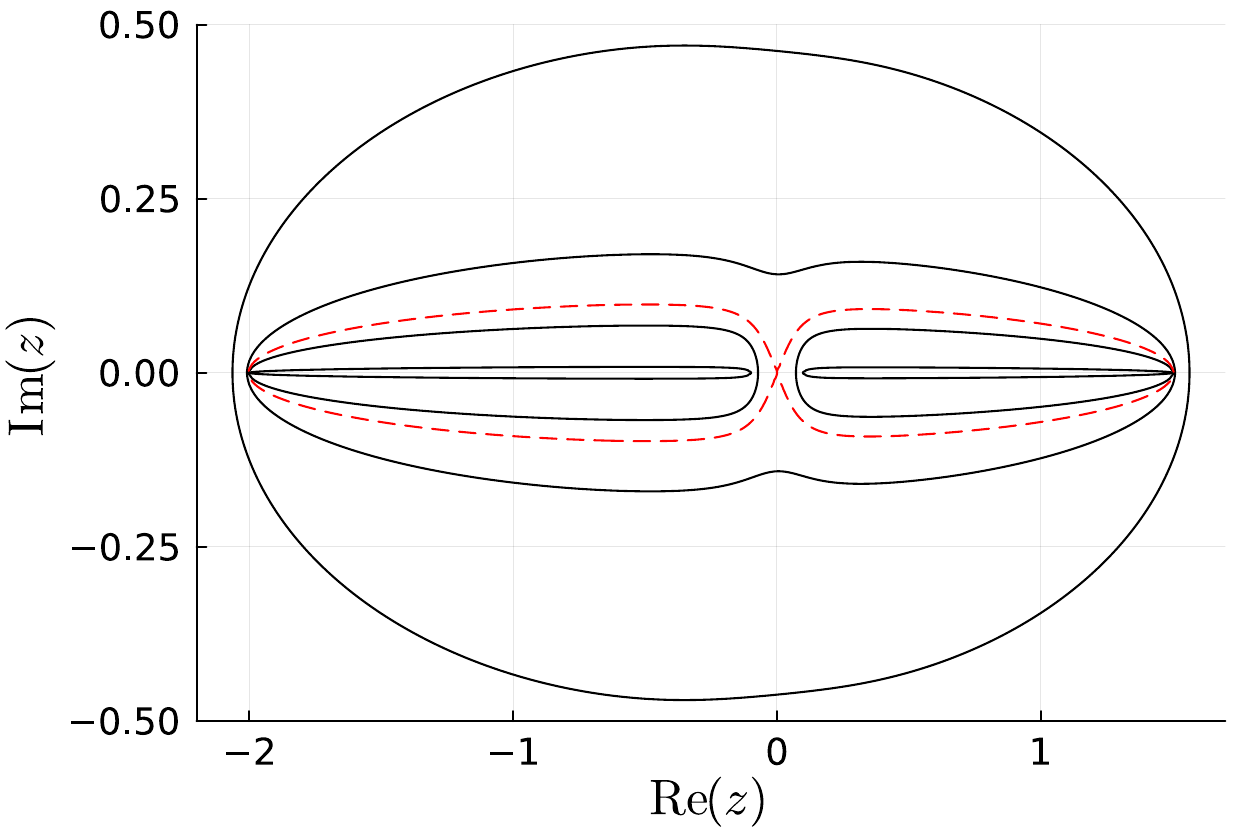}
	\end{subfigure}
\caption{Generalized Bernstein ellipses for the cut domains $[-2,-0.1]\cup[0.1,1.5]$ (left) and $[-2,-0.5]\cup[0.5,6]$ (right). These are level curves of $\ex^{\re\mathfrak g(z)}$ from Lemma~\ref{lem:bernstein}. The red dashed line is the level curve where the curves around each interval first intersect, $\re\mathfrak g(z) = c^*$, and this governs the convergence rate of the Akhiezer polynomial approximation to the sign function on these domains.}
\label{fig:bernstein}
\end{figure}

\subsection{Related work} 
As stated above, many polynomial-based methods suffer from growing per-iteration cost, heavy storage requirements, and necessary residual computation which, in conjunction with the slower convergence rate of such methods relative to rational-based methods, make large-scale problems intractable. Considerable effort has focused on improving such methods on all three fronts \cite{PalittaS18,KressnerLMP21,KressnerMemory08,palitta2024sketched,CasulliHK25}. In particular, the method of \cite{palitta2024sketched} both greatly reduces the per-iteration cost and storage requirements of polynomial-based Krylov subspace methods while accelerating convergence. Despite the fact that the convergence of polynomial-based Krylov methods is slow compared to that of rational Krylov methods, the savings induced by solving the smaller sketched problems can make up the difference because of the per-iteration savings. We suspect that the use of the Akhiezer polynomials in concert with a sketched Krylov subspace method might lead to faster runtimes for such schemes, but we leave this for future work.

Owing to the fact that ${\rm sign}(\bH) = \bH(\bH^{2})^{-1/2}$, iterative methods for computing matrix square roots can be used to evaluate the matrix sign function. While many of these algorithms make use of inversion, the Newton--Schulz iteration is an exception. More precisely, the Newton--Schulz iteration to compute the sign function is simple: 
\begin{equation*}
\bH_{k+1}=\frac{1}{2}\bH_k\left(3\bI-\bH_k^2\right).
\end{equation*}
This iteration converges quadratically, but it requires that $\|\bI-\bH^2\|<1$ in a subordinate matrix norm \cite[Section 5.3]{higham2008functions}. By contrast\footnote{Newton--Schulz can be modified to converge globally, but this requires $\bH$ to be Hermitian \cite{chen2014stable}.}, the methods presented here can be applied whenever the eigenvalues of $\bH$ are close to known intervals on the real axis. We discuss other connections with algorithms for the matrix square root in Section~\ref{sect:sqrt} and refer to~\cite{higham2008functions} for a thorough survey.

% A current limitation of our approach is that it is only applicable when the spectral sets of the coefficient matrices are close to the real axis and the matrices are not highly nonnnormal; substantial work is needed to generalize these polynomials to other domains. In contrast, well-posedness for the sketched Krylov subspace method is less restrictive, requiring the so-called ``sketched" fields of values of the coefficient matrices to be nonintersecting.   

\subsection{Organization} The rest of this paper is organized as follows. In Section~\ref{sect:notation}, we establish notation for orthogonal polynomials and matrix functions and describe the Akhiezer iteration of \cite{AkhIter}. In Section~\ref{sect:alg}, we present Method 1, extending the Akhiezer iteration to an inverse-free iterative method for solving Sylvester matrix equations via the matrix sign function and analyze its complexity and convergence properties. In Section~\ref{sect:Method2}, we present Method 2, using the Akhiezer iteration to directly approximate the inverse of the Sylvester operator and again analyze its complexity and convergence properties. In Section~\ref{sect:applications}, we apply the methods presented here to solve integral equations and two-dimensional PDEs. In Section~\ref{sect:other_apps}, we discuss other related applications. Code used to generate the plots in this paper can be found at \cite{code_repo}.

\section{Orthogonal polynomials and the Akhiezer iteration}\label{sect:notation}
\subsection{Orthogonal polynomials and Cauchy integrals}
	For our purposes, a weight function $w$ is a nonnegative function supported on a finite union of disjoint intervals $\Sigma$, $\Sigma\subset\real$, that is continuous and positive on the interior of $\Sigma$ such that $\int_\Sigma w(x)\df x=1$. Consider a sequence of univariate monic polynomials $\left(\pi_j(x)\right)_{j=0}^\infty$ such that $\pi_j$ has degree $j$ for all $j\in\mathbb{N}$. These polynomials are said to be orthogonal with respect to a weight function $w$ if
	\[
	\langle\pi_j,\pi_k\rangle_{L^2_w(\Sigma)}=h_j\delta_{jk},
	\]
	where $h_j>0$, $\delta_{jk}$ is the Kronecker delta, and
	\begin{align}\label{eq:inner}
		\langle g,h\rangle_{L^2_w(\Sigma)}=\int_\Sigma g(x)\overline{h(x)}w(x)\df x, \quad \|g\|_{L^2_w(\Sigma)}=\sqrt{\langle g,g\rangle_{L^2_w(\Sigma)}}.
	\end{align}
	The orthonormal polynomials $\left(p_j(x)\right)_{j=0}^\infty$ are defined by
	\[
	p_j(x)=\frac{1}{\sqrt{h_j}}\pi_j(x),
	\]
	for all $j\in\mathbb{N}$. The orthonormal polynomials satisfy the symmetric three-term recurrence
	\begin{equation}\label{eq:recurr}
		\begin{aligned}
			&xp_0(x)=a_0p_0(x)+b_0p_1(x),\\
			&xp_k(x)=b_{k-1}p_{k-1}(x)+a_kp_k(x)+b_kp_{k+1}(x),\quad k\geq1,
		\end{aligned}
	\end{equation}
	where $b_k>0$ for all $k$. A general reference is \cite{Szego1939}.
	
	Given a contour $\Gamma\subset\mathbb{C}$ and a function $f:\Gamma\to\mathbb{C}$, the Cauchy transform $\mathcal{C}_\Gamma$ is an operator that maps $f$ to its Cauchy integral, i.e.,
	\[
	\mathcal{C}_\Gamma f(z)=\frac{1}{2\pi \im}\int_\Gamma\frac{f(s)}{s-z}\df s,\quad z\in\compl\setminus\Gamma.
	\]
	% Cauchy integrals of orthogonal polynomials will also be of use in this work. For orthonormal polynomials $p_k(x)$ corresponding to the weight function $w$ on $\Sigma\subset\real$, define 
	% \[
	% C_k(z)=\mathcal{C}_\Sigma[p_kw](z)=\frac{1}{2\pi \im}\int_\Sigma\frac{p_k(s)w(s)}{s-z}\df s,
	% \] 
	% for $z\in\mathbb{C}\setminus \Sigma$.
    The weighted Cauchy transforms $\mathcal C_{\Sigma} [ p_k w](z)$ will be important in the developments below.
\subsection{Functions of matrices} The function of a matrix $f(\bM)$ can be defined in several equivalent ways. For our purposes, the following two definitions will suffice:
	\begin{definition}\label{def:fdisc}
	Suppose that $f$ is a scalar-valued 
    % \textcolor{red}{Should we make the lambdas sigmas instead for consistency?} 
    function defined on the spectrum of a diagonalizable matrix $\bM\in\compl^{n\times n}$, where $\bM$ is diagonalized as $\bM=\bV\bLambda\bV^{-1}$, $\bLambda = {\rm diag}(\lambda_1,\ldots,\lambda_n)$. Then, 
    \begin{equation*}
        f(\bM):=\bV\mathrm{diag}\left(f(\lambda_1),\ldots,f(\lambda_n)\right)\bV^{-1}.
    \end{equation*}
	\end{definition}
    With suitable assumptions on $f$, this definition can be extended to nondiagonalizable matrices via the Jordan normal form, but we will restrict most of our analysis to diagonalizable matrices.  When $f$ is appropriately analytic, we have the following equivalent definition \cite[Theorem 1.12]{higham2008functions}:
	\begin{definition} \label{def:fan} Suppose that $\Gamma$ is a counterclockwise-oriented simple curve that encloses the spectrum of $\bM\in\compl^{n\times n}$ and that $f$ is analytic in a region containing $\Gamma$ and its interior. Then,
		\begin{equation*}
			f(\bM):=\frac{1}{2\pi\im}\int_\Gamma f(z)(z\bI-\bM)^{-1}\df z.
		\end{equation*}
	\end{definition}
	
	\subsection{The Akhiezer iteration}
	Introduced in \cite{AkhIter}, the Akhiezer iteration\footnote{The method takes the name of Naum Akhiezer because, as discussed in Section~\ref{sec:AData}, the classes of orthogonal polynomials that make the computation of the input data to the algorithm particularly efficient are often called \emph{Akhiezer polynomials} \cite{akhiezer,Chen2008}.} uses orthogonal polynomial series expansions to compute matrix functions. Given a finite union of disjoint intervals $\Sigma\subset\real$ and a function $f$ that is analytic in a region containing $\Sigma$, let $p_0,p_1,\ldots$ denote the orthonormal polynomials with respect to $w$. Then, for $x\in\Sigma$, a $p_j$-series expansion for $f$ is given by
	\begin{equation*}
	f(x)=\sum_{j=0}^\infty\alpha_jp_j(x),\quad \alpha_j=\langle f,p_j\rangle_{L^2_w(\Sigma)}.
	\end{equation*}
	For a matrix $\bM$ with eigenvalues on or near $\Sigma$, this extends to an iterative method for computing $f(\bM)$ by truncating the series:
	\begin{equation}\label{eq:series_trunc}
		f(\bM)=\sum_{j=0}^\infty\alpha_jp_j(\bM)\approx\sum_{j=0}^k\alpha_jp_j(\bM)=:\bF_{k+1}.
	\end{equation}
	Of course, this requires methods to compute the coefficients $\{\alpha_j\}_{j=0}^\infty$. The matrix polynomials $\left(p_j(\bA)\right)_{j=0}^\infty$ can be generated by \eqref{eq:recurr},
	\begin{equation}\label{eq:mat_polys}
		\begin{aligned}
			&p_0(\bM)=\bI,\\
			&p_1(\bM)=\frac{1}{b_0}(\bM p_0(\bM)-a_0p_0(\bM)),\\
			&p_{k}(\bM)=\frac{1}{b_{k-1}}(\bM p_{k-1}(\bM)-a_{k-1}p_{k-1}(\bM)-b_{k-2}p_{k-2}(\bM)),\quad k\geq2.
		\end{aligned}
	\end{equation}
	The coefficients $\alpha_j$ cannot be computed analytically, in general, so let $\Gamma, f$ be as in Definition~\ref{def:fan}. Then,
	\begin{equation*}
	\alpha_j=\langle f,p_j\rangle_{L^2_w(\Sigma)}=\int_{\Sigma}f(x)p_j(x)w(x)\df x=\int_{\Sigma}\left(\frac{1}{2\pi\im}\int_{\Gamma}\frac{f(z)}{z-x}\df z\right)p_j(x)w(x)\df x.
	\end{equation*}
	After parameterizing $\Gamma$ and applying a quadrature rule, such as a trapezoid rule \cite{trapezoid}, resulting in nodes $\{z_j\}_{j=1}^m$ and weights $\{w_j\}_{j=1}^m$,
	\begin{equation*}
	\alpha_j\approx\int_{\Sigma}\frac{1}{2\pi\im}\sum_{k=1}^m\frac{f(z_k)w_k}{z_k-x}p_j(x)w(x)\df x=-\sum_{k=1}^mf(z_k) w_k \cC_\Sigma[p_jw](z_k).
	\end{equation*}
    We see that in order to approximate the coefficients accurately, it suffices to be able to evaluate the weighted Cauchy integrals pointwise.
	Therefore, to apply the approximation \eqref{eq:series_trunc},  the recurrence coefficients and the pointwise evaluation of the Cauchy integrals of the desired orthonormal polynomials are all that is required.  We refer to these input coefficients as the \emph{Akhiezer data}.

With the Akhiezer data in hand, a precomputation, the approximation \eqref{eq:series_trunc} can be implemented as the Akhiezer iteration as in Algorithm~\ref{alg:akh_func}. See \cite[Section 4]{AkhIter} for the convergence analysis of the Akhiezer iteration, including that case where eigenvalues of $\bM$ lie off $\Sigma$.
\begin{algorithm}
		\caption{Akhiezer iteration for matrix function approximation}\label{alg:akh_func}
		\textbf{Input: }{$f$, $\mathbf{M}$, and the Akhiezer data (functions to compute recurrence coefficients $a_k,b_k$ and $p_k$-series coefficients $\alpha_k$).}\\
		Set $\mathbf F_0=\bzero$.\\
		\For{k=0,1,\ldots}{
			\uIf{k=0}{
				Set $\mathbf P_0=\mathbf I$.
			}
			\uElseIf{k=1}{
				Set $\mathbf P_{1}=\frac{1}{b_0}(\mathbf M\mathbf P_0-a_0\mathbf P_0)$.
			}
			\Else{
				Set $\mathbf P_{k}=\frac{1}{b_{k-1}}(\mathbf M\mathbf P_{k-1} - a_{k-1}\mathbf P_{k-1} -  b_{k-2}\mathbf P_{k-2}).$
			}
			Set $\mathbf F_{k+1}=\mathbf F_k + \alpha_k \mathbf P_k$.\\
			\If{\emph{converged}}{\Return{$\mathbf F_{k+1}$}.}
		}
	\end{algorithm}
    \begin{remark}
    For most functions $f$, we choose $\Gamma$ to be a collection of circles around each component of $\Sigma$, applying the trapezoid rule.  Other contours, such as ellipses or hyperbolae, could provide better approximations in certain situations. We present an alternative approach specific to the sign function in Section~\ref{sect:Poisson}.
    \end{remark}

    % and utilize a trapezoid rule, but one could, in principle, cater these to the function $f$ and matrix $\bM$, if, for example, $f$ had a singularity near one of the intervals. Specifically, one may wish to instead build $\Gamma$ out of ellipses to optimize the convergence rate when $f$ has a small region of analyticity.

\subsection{Computing the Akhiezer data}\label{sec:AData}

If $\Sigma$ is only a single interval, scaled and shifted Chebyshev polynomials and their simple recurrence are applicable. Formulae for the Cauchy integrals of Chebyshev polynomials can be found in \cite[Section 4.1]{RHISM}. The method that results is closely related to the well-known Chebyshev iteration \cite{Manteuffel1977} in the case of $f(z) = z^{-1}$.

In \cite{AkhIter} and \cite{RHISM}, the authors present several methods for generating recurrence coefficients and the evaluation of Cauchy integrals. In particular, when $\Sigma=[\beta_1,\gamma_1]\cup[\beta_2,\gamma_2]$, $\gamma_1 < \beta_2$, explicit formulae for the recurrence coefficients
% \footnote{Empirically, in the symmetric case $\alpha=-\beta$, the recurrence coefficient formulae in \cite[Section 3.1]{AkhIter} appear to reduce to $a_n=(-1)^n\beta$ ($n\geq0$), $b_0=\sqrt{\frac{1-\beta^2}{2}}$, $b_n=\frac{\sqrt{1-\beta^2}}{2}$ ($n\geq1$), but we do not yet have a proof of this.} 
and Cauchy integrals of the so-called \emph{Akhiezer polynomials} are found in terms of Jacobi elliptic and Jacobi theta functions\footnote{The recurrence coefficients are also known to satisfy a recurrence relation of their own \cite[Theorem 3.3.7]{FischerBook}.} \cite[Section 3.1]{AkhIter}.  The two-interval Akhiezer polynomials are orthonormal polynomials with respect to the weight function \cite[Chapter 10]{akhiezer}
	\begin{equation}\label{eq:akweight}
		w(x)=\frac{1}{\pi}\mathbbm{1}_\Sigma(x)\frac{\sqrt{x-\gamma_1}}{\sqrt{\gamma_2-x}\sqrt{x-\beta_1}\sqrt{x-\beta_2}}.
	\end{equation}
\begin{remark}
    In the symmetric case $\Sigma=[-1,-\beta]\cup[\beta,1]$ ($\beta>0$), the Akhiezer polynomial recurrence coefficients can be shown to take a particularly simple form:
    	\begin{equation*}
		\begin{aligned}
			a_n&=(-1)^n\beta, \quad n\geq0,\\
			b_0&=\sqrt{\frac{1-\beta^2}{2}}, \quad b_n=\frac{\sqrt{1-\beta^2}}{2},\quad n\geq1.
		\end{aligned}
	\end{equation*}
    This follows from a result of Perherstorfer \cite{Peherstorfer1990}, but can be found directly, for example, by using continued fractions to show that \eqref{eq:akweight} is the density of the spectral measure for the Jacobi matrix of recurrence coefficients.
    
    If $\gamma_1+\beta_2=\beta_1+\gamma_2$, the Akhiezer polynomial recurrence coefficients for $\Sigma=[\beta_1,\gamma_1]\cup[\beta_2,\gamma_2]$, $\gamma_1 < \beta_2$, can be obtained by shifting and scaling these formulae.
\end{remark}
	 If $\Sigma=\bigcup_{j=1}^{g+1}[\beta_j,\gamma_j]$ consists of more than two intervals, we refer the reader to the Riemann--Hilbert-based numerical method introduced in \cite{RHISM} and simplified in \cite[Appendix A]{AkhIter}. This method requires $\OO(1)$ time to compute any given recurrence coefficient and $\OO(1)$ time to evaluate any weighted Cauchy integral at a point. The method is particularly efficient for weight functions of the form
	\begin{equation*}
		w(x)\propto\mathbbm{1}_{\Sigma}(x)\frac{\sqrt{\gamma_{g+1}-x}\prod_{j=1}^{g+1}\sqrt{x-\beta_j}}{\prod_{j=1}^{g}\sqrt{x-\gamma_j}}.
	\end{equation*}
	Akhiezer's formulae and the Riemann--Hilbert-based numerical method are implemented in the \Julia\ package \texttt{RecurrenceCoefficients.jl} \cite{RecurrenceCoefficients.jl}.
	
	While these are our methods of choice, alternative methods do exist. An optimized $\OO(N^2)$ algorithm for computing $N$ pairs of recurrence coefficients for general orthogonal polynomials is given as RKPW in \cite{Gragg1984} and \texttt{lanczos.m} in \cite{gautschi}. More precisely, one uses a discretization of the inner product \eqref{eq:inner} as input to RKPW. In a different approach, \cite{Wheeler1984}, Wheeler constructs a weight function on two intervals such that the recurrence coefficients of the resulting orthogonal polynomials are 2-periodic and easily computable. Additionally, the annular polynomials of \cite{annuliOPs} in 1D yield orthogonal polynomials on symmetric intervals. In all of these cases, one may still need to compute Cauchy integrals of the orthonormal polynomials if the coefficients $\alpha_j$ cannot be computed through other means. See \cite[Chapter 7]{olver_slevinsky_townsend_2020} for a discussion of this computation.

    For $\Sigma=\bigcup_{j=1}^{g+1}[\beta_j,\gamma_j]$, a more general class of weight functions to which the forthcoming analysis applies is those of the form
    \begin{equation}\label{eq:gen_weight}
    w(x)=\sum_{j=1}^{g+1}\mathbbm{1}_{[\beta_j,\gamma_j]}(x)h_j(x)\left(\sqrt{x-\beta_j}\right)^{c_j}\left(\sqrt{\gamma_j-x}\right)^{d_j},\quad c_j,d_j\in\{-1,1\},
    \end{equation}
    where each $h_j$ is positive on $[\beta_j,\gamma_j]$ and has an analytic extension to a neighborhood of $[\beta_j,\gamma_j]$. The numerical method of \cite{RHISM} applies to all weight functions of this class.
	
	\section{Method 1: the matrix sign function}\label{sect:alg}
	We now turn to solving the Sylvester matrix equation~\eqref{eq:Genmat}. We assume that the spectrum of $\bA$ is contained in $U_\bA$ and that the spectrum of $\bB$ is contained in $U_\bB$. Suppose $\Xi = \alpha + \ex^{\im \theta} \mathbb R$ is chosen such that the halfplanes $H_\bA$ and $H_\bB$ separated by $\Xi$ contain $U_\bA$ and $U_\bB$, respectively. Then, define
    \begin{equation}\label{eq:sign_func}
	\sign:\compl\setminus\Xi\to\compl,\quad\sign(z)=\begin{cases}
		\phantom{-}1,&z\in H_\bA,\\
		-1,&z\in H_\bB.
	\end{cases}
	\end{equation}
     The choice of $\Xi$ will affect the convergence rate estimates that appear later in this section.
	
	The solution of \eqref{eq:Genmat} can be obtained as a subblock of the matrix sign function of the block matrix $\bH$ in~\eqref{eq:signblock}.
	This follows from the factorizations
	\[
	\begin{pmatrix}\bA&\bzero \\ \bC+\bB\bX&\bB\end{pmatrix}=\begin{pmatrix}\bA&\bzero \\ \bC&\bB\end{pmatrix}\begin{pmatrix}\bI_n&\bzero \\ \bX&\bI_m\end{pmatrix},\quad \begin{pmatrix}\bA&\bzero \\ \bX\bA&\bB\end{pmatrix}=\begin{pmatrix}\bI_n&\bzero \\ \bX&\bI_m\end{pmatrix}\begin{pmatrix}\bA&\bzero \\ \bzero&\bB\end{pmatrix},
	\]
	so \eqref{eq:Genmat} is equivalent to the equation
	\begin{equation}\label{eq:block_sylv_eqn}
	\bH:=\begin{pmatrix}\bA&\bzero \\ \bC&\bB\end{pmatrix}=\begin{pmatrix}\bI_n&\bzero \\ \bX&\bI_m\end{pmatrix}\begin{pmatrix}\bA&\bzero \\ \bzero&\bB\end{pmatrix}\begin{pmatrix}\bI_n&\bzero \\ -\bX&\bI_m\end{pmatrix}.
	\end{equation}
	When $\sigma(\bA) \subset [\beta_2,\gamma_2] \subset \mathbb R$, $\sigma(\bB) \subset [\beta_1,\gamma_1] \subset \mathbb R$,  and $\gamma_1 < \beta_2$, the Akhiezer iteration for $\Sigma=[\beta_1,\gamma_1]\cup[\beta_2,\gamma_2]$, as presented in Algorithm~\ref{alg:akh_func}, can be directly applied to compute $\sign(\bH)$, and thus solve the Sylvester equation \eqref{eq:Genmat}. Note that despite its discontinuous nature, the sign function can always be defined so that it is analytic in a (disconnected) region containing $\Sigma$.  We discuss this application and efficiency improvements in the following subsections.

    This construction directly yields the following solution formula for~\eqref{eq:Genmat}, where the convergence of the infinite sum follows from Lemma~\ref{lem:poly_asymp} below:
    \begin{theorem}\label{thm:sol_formula}
        Define the sign function as in~\eqref{eq:sign_func} and let $\sigma(\bA),\sigma(\bB)$ be contained in the interior of $\Sigma=\bigcup_{j=1}^{g+1}[\beta_j,\gamma_j]\subset\real$ where $\Xi$ does not intersect $\Sigma$. Let $w$ be a weight function of the form~\eqref{eq:gen_weight} and denote the corresponding orthonormal polynomials by $\left(p_j(x)\right)_{j=0}^\infty$. Then, the solution of the Sylvester equation~\eqref{eq:Genmat} is given by
        \begin{equation*}
            \bX=\sum_{j=0}^\infty\alpha_j p_j(\bH)_{n+1:n+m,1:n},\quad \bH=\begin{pmatrix}\bA&\bzero \\ \bC&\bB\end{pmatrix},\quad \alpha_j=\int_\Sigma\sign(x)p_j(x)w(x)\df x.
        \end{equation*}
    \end{theorem}
    \begin{remark}\label{rem:eigs}
        The assumption that $\sigma(\bA),\sigma(\bB)\subset\Sigma$ in Theorem~\ref{thm:sol_formula} is unnecessary. In the notation of Appendix~\ref{ap:eig}, this assumption may be replaced with the requirement that $\nu(z_*;\bH)<0$, where $z_*$ is any point on the level curve $\Gamma_\varrho = \partial B_\varrho$ \eqref{eq:Bvarrho}, interior to which the sign function is analytic. See Lemma~\ref{lem:bernstein}. 
    \end{remark}
    
	\subsection{Isolating the lower left block}
	Since only the lower left block of $\sign(\bH)$ is required to solve the Sylvester equation \eqref{eq:Genmat}, we can reduce the necessary arithmetic operations using the following lemma.
	\begin{lemma}\label{lem:block_recurr}
		Given a sequence of orthonormal polynomials $\left(p_j(z)\right)_{j=0}^\infty$ with three-term recurrence \eqref{eq:recurr}, the lower left block of $p_j(\bH)$ where
		\begin{equation*}
		\bH = \begin{pmatrix}\bA&\bzero \\ \bC&\bB\end{pmatrix},
		\end{equation*}
		is given by 
		\begin{equation*}
		\bC p_j(\bA)+\bG_j,
		\end{equation*}
		where $\bG_j$ satisfies
		\begin{equation}\label{eq:block_recurr}
			\begin{aligned}
				&\bG_0=-\bC,\\
				&\bG_1=\frac{1}{b_0}(\bG_0\bA+(a_0+1)\bC),\\
				&\bG_j=\frac{1}{b_{j-1}}\left(\bG_{j-1}\bA+p_{j-1}(\bB)\bC-a_{j-1}\bG_{j-1}-b_{j-2}\bG_{j-2}\right),\quad j\geq2.
			\end{aligned}
		\end{equation}
	\end{lemma}
	\begin{proof}
	For $j=0$, the lower left block of $p_0(\bH) = \vec I$ is zero.
%	\begin{equation*}
%	p_0(\bH)=\begin{pmatrix}\bI&\bzero \\ \bzero&\bI\end{pmatrix}.
%	\end{equation*}
	Since $p_0(\bA)=\bI$, $\bC p_0(\bA)+\bG_0=\bzero$ when $\bG_0=-\bC$.
	When $j=1$, by \eqref{eq:mat_polys}, the lower left block of $p_1(\bH)$ is given by $\frac{1}{b_0}\bC$. Applying \eqref{eq:mat_polys} to $p_1(\bA)$,
	\begin{align*}
		\frac{1}{b_0}\bC&=\frac{1}{b_0}\bC(-a_0\bI+(a_0+1)\bI)=\frac{1}{b_0}\bC(b_0p_1(\bA)-\bA+(a_0+1)\bI)\\
        &=\bC p_1(\bA)+\frac{1}{b_0}(\bG_0\bA+(a_0+1)\bC).
	\end{align*}
	Using induction, let $j\geq2$ and assume that the lemma holds for $p_{j-1}(\bH)$ and $p_{j-2}(\bH)$. Then,
	\begin{equation*}
		p_{j-1}(\bH)=\begin{pmatrix}
			p_{j-1}(\bA)&\bzero\\\bC p_{j-1}(\bA)+\bG_{j-1}&p_{j-1}(\bB)
		\end{pmatrix},\quad p_{j-2}(\bH)=\begin{pmatrix}
			p_{j-2}(\bA)&\bzero\\\bC p_{j-2}(\bA)+\bG_{j-2}&p_{j-2}(\bB)
		\end{pmatrix}.
	\end{equation*}
	Applying \eqref{eq:mat_polys} and noting that matrices commute with polynomials of themselves, the lower left block of $p_{j}(\bH)$ is given by
	\begin{equation*}
		\begin{aligned}
			&\frac{1}{b_{j-1}}\left(\bC p_{j-1}(\bA)\bA+\bG_{j-1}\bA+p_{j-1}(\bB)\bC-a_{j-1}\bC p_{j-1}(\bA)-a_{j-1}\bG_{j-1}-b_{j-2}\bC p_{j-2}(\bA)-b_{j-2}\bG_{j-2}\right)\\&=
			\bC p_j(\bA)+\frac{1}{b_{j-1}}\left(\bG_{j-1}\bA+p_{j-1}(\bB)\bC-a_{j-1}\bG_{j-1}-b_{j-2}\bG_{j-2}\right),
		\end{aligned}
	\end{equation*}
	which completes the proof.
	\end{proof}
	Lemma~\ref{lem:block_recurr} allows the Akhiezer iteration applied to Sylvester equations to be implemented such that full block $(n+m)\times(n+m)$ matrix-matrix multiplications are no longer needed; however, the matrix polynomials $p_j(\bA)$ and $p_j(\bB)$ are both still required to compute the desired lower left block. Thus, recurrences for these matrix polynomials and the approximate solution can be run in parallel, effectively decoupling the large matrix-matrix product into five smaller ones. A Sylvester equation solver based on this decoupling can be implemented as in Algorithm \ref{alg:sylv_decoupled}.
	\begin{algorithm}
		\caption{Decoupled $\bX\bA-\bB\bX=\bC$ solver via the Akhiezer iteration}\label{alg:sylv_decoupled}
		\textbf{Input: }{Matrices $\bA,\bB,\bC$ and the Akhiezer data (functions to compute recurrence coefficients $a_k,b_k$ and $p_k$-series coefficients $\alpha_k$ for the sign function).}\\
		Set $\bX_0=\bzero$.\\
		\For{k=0,1,\ldots}{
			\uIf{k=0}{
				Set $p_0(\mathbf{A})=\mathbf I$.\\
				Set $p_0(\mathbf{B})=\bI$.\\
				Set $\mathbf G_0=-\bC$.
			}
			\uElseIf{k=1}{
				Set $p_1(\mathbf{A})=\frac{1}{b_0}(p_0(\mathbf{A})\mathbf A-a_0p_0(\mathbf A))$.\\
				Set $p_1(\mathbf{B})=\frac{1}{b_0}(\mathbf Bp_0(\mathbf{B})-a_0p_0(\mathbf B))$.\\
				Set $\mathbf{G}_1=\frac{1}{b_0}(\bG_0\bA+(a_0+1)\bC)$.
			}
			\Else{
				Set $p_k(\mathbf{A})=\frac{1}{b_{k-1}}(p_{k-1}(\mathbf{A})\mathbf A-a_{k-1}p_{k-1}(\mathbf{A})-b_{k-2}p_{k-2}(\mathbf{A}))$.\\
				Set $p_k(\mathbf{B})=\frac{1}{b_{k-1}}(\mathbf Bp_{k-1}(\mathbf{B})-a_{k-1}p_{k-1}(\mathbf{B})-b_{k-2}p_{k-2}(\mathbf{B}))$.\\
				Set $\bG_k=\frac{1}{b_{k-1}}\left(\bG_{k-1}\bA+p_{k-1}(\bB)\bC-a_{k-1}\bG_{k-1}-b_{k-2}\bG_{k-2}\right)$.
			}
			Set $\mathbf X_{k+1}=\mathbf X_k + \frac{\alpha_k}{2}\left(\bC p_k(\bA)+\bG_k\right)$.\\
			\If{\emph{converged}}{\Return{$\bX_{k+1}$}.}
		}
	\end{algorithm}
	\subsection{Compression and low-rank structure}\label{sect:low-rank}
	In the case where the data matrix $\bC$ in \eqref{eq:Genmat} is low rank, Algorithm~\ref{alg:sylv_decoupled} is suboptimal in that it requires full matrix-matrix products. Moreover, it is known that when the rank of $\bC$ is bounded above by a small number $r$ and $\bA$ and $\bB$ are normal\footnote{Generalizations for the nonnormal case are also known~\cite{beckermann2019bounds}.} with well-separated spectral sets, then $\bX$ is well-approximated by a low-rank matrix~\cite{beckermann2019bounds}. We say that such a matrix is \textit{numerically} of low rank. Specifically, we say that $\bX$ is numerically of rank $k$ with respect to the tolerance $0< \epsilon < 1$ if the $(k+1)$st  singular value of $\bX$ is bounded above by $\|\bX\|_F \epsilon$, where $\|\diamond\|_F$ denotes the standard Frobenius norm. Throughout this paper, we take $\epsilon=10^{-14}$.  Rather than constructing $\bX_k$ outright, it is preferable to construct low-rank factors $\bW_k \bZ_k = \bX_k$. 
    
    Here, we suppose we are given a low-rank factorization of $\bC$ in~\eqref{eq:Genmat}, and so consider the equation
 \begin{equation}\label{eq:lr_sylv_eqn}
	\bX\bA-\bB\bX=\bU\bV,
	\end{equation}
    where $\bU\in\compl^{m\times r}$ and $\bV\in\compl^{r\times n}$ are given with $r$ much smaller than $m$ and $n$. 
    To incorporate the low-rank structure into our iteration, we note that the recurrence \eqref{eq:mat_polys} can be left- and right-multiplied, respectively, giving
    \begin{equation*}
		\begin{aligned}
			&\bV p_0(\bA)=\bV,\\
			&\bV p_1(\bA)=\frac{1}{b_0}(\bV p_0(\bA)\bA-a_0\bV p_0(\bA)),\\
			&\bV p_{k}(\bA)=\frac{1}{b_{k-1}}(\bV p_{k-1}(\bA)\bA-a_{k-1}\bV p_{k-1}(\bA)-b_{k-2}\bV p_{k-2}(\bA)),\quad k\geq2,
		\end{aligned}
\end{equation*}
and
\begin{equation*}
		\begin{aligned}
			&p_0(\bB)\bU=\bU,\\
			&p_1(\bB)\bU=\frac{1}{b_0}(\bB p_0(\bB)\bU-a_0p_0(\bB)\bU),\\
			&p_{k}(\bB)\bU=\frac{1}{b_{k-1}}(\bB p_{k-1}(\bB)\bU-a_{k-1}p_{k-1}(\bB)\bU-b_{k-2}p_{k-2}(\bB)\bU),\quad k\geq2.
		\end{aligned}
\end{equation*}
By computing with only these recurrence formulae, we circumvent the need to compute the full diagonal blocks of $p_j(\bH)$, greatly reducing the size of the needed matrix-matrix products when $r$ is small. Letting $\bG_j=\bJ_j\bK_j$, we write the recurrence \eqref{eq:block_recurr} in block low-rank form as
\begin{equation*}
\begin{aligned}
        &\bJ_0\bK_0=\bG_0=-\bU\bV,\\
        &\bJ_1\bK_1=\bG_1=\frac{1}{b_0}\begin{pmatrix}
					\mathbf{J}_{0} &(a_0+1)\mathbf U
				\end{pmatrix}\begin{pmatrix}
					\mathbf{K}_{0}\mathbf{A} \\ \mathbf{V}
				\end{pmatrix},\\
        &\bJ_j\bK_j=\bG_j=\frac{1}{b_{j-1}}\begin{pmatrix}
					\mathbf{J}_{j-1} &p_{j-1}(\mathbf{B})\mathbf{U} &-a_{j-1}\mathbf{J}_{j-1} &-b_{j-2}\mathbf{J}_{j-2}
				\end{pmatrix}\begin{pmatrix}
					\mathbf{K}_{j-1}\mathbf{A} \\ \mathbf{V} \\ \mathbf{K}_{j-1} \\ \mathbf{K}_{j-2}
				\end{pmatrix},\quad j\geq2.
    \end{aligned}
\end{equation*}
Then, letting $\bX_k=\bW_k\bZ_k$, the update step from Algorithm~\ref{alg:sylv_decoupled} can be written as
\begin{equation*}
\bW_{k+1}\bZ_{k+1}=\bX_{k+1}=\begin{pmatrix}
				\mathbf W_{k}& \frac{\alpha_k}{2}\mathbf{U} & \frac{\alpha_k}{2}\mathbf{J}_k
			\end{pmatrix}\begin{pmatrix}
				\mathbf Z_{k}\\ \mathbf{V}p_k(\mathbf{A}) \\ \mathbf{K}_k
			\end{pmatrix}.
\end{equation*}
These relations allow Algorithm~\ref{alg:sylv_decoupled} to be decoupled such that two matrices, $\bW_k$ and $\bZ_k$, are updated separately and the approximate solution is the product of the two. However, their direct use is hampered by the fact that the number of rows or columns grows rapidly with $k$. To address this, we compress these matrices to their numerical rank as needed. A simple compression method based on the QR/LQ and singular value decompositions (SVD) is given in Algorithm~\ref{alg:compress}.
\begin{algorithm}
\caption{Simple compression of ($\bJ, \bK$)}\label{alg:compress}
$\texttt{COMPRESS}(\bJ,\bK,\epsilon)$\\
\textbf{Input: }{Matrices $\bJ\in\compl^{m\times n}$, $\bK\in\compl^{n\times \ell}$ and tolerance $\epsilon$.}\\
Run QR on $\bJ$ to get $\bQ_\bJ\bR=\bJ$.\\
Run LQ on $\bK$ to get $\bL\bQ_\bK=\bK$.\\
Run SVD on $\bR\bL$ to get $\bU\bSigma\bV=\bR\bL$.\\
Truncate $\bSigma$ so that all remaining singular values satisfy $\sigma_j^2/\sum_{j}\sigma_j^2<\epsilon$, leaving $\Tilde\bSigma\in\real^{k\times k}$. \\
Truncate rows or columns of $\bU$ and $\bV$ as appropriate, leaving $\Tilde\bU\in\compl^{m\times k}$ and $\Tilde\bV\in\compl^{k\times \ell}$.\\
Return $\Tilde\bJ=\bQ_\bJ\Tilde\bU\Tilde\bSigma^{1/2}$, $\Tilde\bK=\Tilde\bSigma^{1/2}\Tilde\bV\bQ_\bK$.
\end{algorithm}

Given a method that compresses the matrices $\bJ_k\bK_k$ and $\bW_k\bZ_k$ to match their numerical rank, the above recurrence relations allow Algorithm~\ref{alg:sylv_decoupled} to be rewritten so that only low-rank matrix-matrix products and compression steps are necessary to compute a solution. An implementation of this is given as Algorithm~\ref{alg:sylv_low_rank}.
    
\begin{algorithm}
		\caption{Low-rank $\bX\bA-\bB\bX=\bU\bV$ solver using the Akhiezer iteration for the sign function}\label{alg:sylv_low_rank}
		\textbf{Input: }{Matrices $\bA,\bB,\bU,\bV$, the Akhiezer data (functions to compute recurrence coefficients $a_k,b_k$ and $p_k$-series coefficients $\alpha_k$ for the sign function), computed decay rate $\varrho$ and constant $c$ as in Section~\ref{sect:param}, and compression tolerance $\epsilon$.}\\
		Initialize $\mathbf W_0$ and $\mathbf Z_0$ to be empty.\\
		\For{k=0,1,\ldots}{
			\uIf{k=0}{
				Set $\mathbf{V}p_0(\mathbf{A})=\mathbf V$.\\
				Set $p_0(\mathbf{B})\mathbf{U}=\mathbf{U}$.\\
				Set $\mathbf J_0=-\mathbf U$.\\
				Set $\mathbf K_0=\mathbf V$.\\
			}
			\uElseIf{k=1}{
				Set $\mathbf{V}p_1(\mathbf{A})=\frac{1}{b_0}(\mathbf{V}p_0(\mathbf{A})\mathbf A-a_0\bV p_0(\mathbf A))$.\\
				Set $p_1(\mathbf{B})\mathbf{U}=\frac{1}{b_0}(\mathbf Bp_0(\mathbf{B})\mathbf{U}-a_0p_0(\mathbf B)\bU)$.\\
				Set $\mathbf{J}_1=\frac{1}{b_0}\begin{pmatrix}
					\mathbf{J}_{0} &(a_0+1)\mathbf U
				\end{pmatrix}$.\\
				Set $\mathbf{K}_1=\begin{pmatrix}
					\mathbf{K}_{0}\mathbf{A} \\ \mathbf{V}
				\end{pmatrix}$.
			}
			\Else{
				Set $\mathbf{V}p_k(\mathbf{A})=\frac{1}{b_{k-1}}(\mathbf{V}p_{k-1}(\mathbf{A})\mathbf A-a_{k-1}\mathbf{V}p_{k-1}(\mathbf{A})-b_{k-2}\mathbf{V}p_{k-2}(\mathbf{A}))$.\\
				Set $p_k(\mathbf{B})\mathbf{U}=\frac{1}{b_{k-1}}(\mathbf Bp_{k-1}(\mathbf{B})\mathbf{U}-a_{k-1}p_{k-1}(\mathbf{B})\mathbf{U}-b_{k-2}p_{k-2}(\mathbf{B})\mathbf{U})$.\\
				Set $\mathbf{J}_k=\frac{1}{b_{k-1}}\begin{pmatrix}
					\mathbf{J}_{k-1} &p_{k-1}(\mathbf{B})\mathbf{U} &-a_{k-1}\mathbf{J}_{k-1} &-b_{k-2}\mathbf{J}_{k-2}
				\end{pmatrix}$.\\
				Set $\mathbf K_k=\begin{pmatrix}
					\mathbf{K}_{k-1}\mathbf{A} \\ \mathbf{V} \\ \mathbf{K}_{k-1} \\ \mathbf{K}_{k-2}
				\end{pmatrix}$.
			}
                Set $\bJ_{k},\bK_{k}=\texttt{COMPRESS}\left(\bJ_{k},\bK_{k},\frac{\varrho^k\epsilon}{c}\right)$.\\
			Set $\mathbf W_{k+1}= \begin{pmatrix}
				\mathbf W_{k}& \frac{\alpha_k}{2}\mathbf{U} & \frac{\alpha_k}{2}\mathbf{J}_k
			\end{pmatrix}$.\\
			Set $\mathbf Z_{k+1}=\begin{pmatrix}
				\mathbf Z_{k}\\ \mathbf{V}p_k(\mathbf{A}) \\ \mathbf{K}_k
			\end{pmatrix}$.\\
                Set $\bW_{k+1},\bZ_{k+1}=\texttt{COMPRESS}\left(\bW_{k+1},\bZ_{k+1},\epsilon\right)$.\\
			\If{\emph{converged}}{\Return{$\mathbf W_{k+1},\mathbf Z_{k+1}$}.}
		}
	\end{algorithm}

This algorithm chooses compression tolerances in a manner that depends on the iteration step which we now motivate. Since $\bX$ will be numerically of low rank, we expect the approximate solution $\bW_k\bZ_k$ of Algorithm~\ref{alg:sylv_low_rank} to be numerically of low rank for all $k$. However, we have no reason to expect the lower left block of $p_j(\bH)$ to be numerically of low rank for large $j$, meaning that as the iteration count increases, a low-rank object is being constructed out of increasingly higher-rank objects. In Figure \ref{fig:unweight_rank}, we plot the numerical ranks of both the component pieces $\bJ_k\bK_k$ and of the approximate solution $\bW_k\bZ_k$ for an example problem where $\bA\in\mathbb{R}^{1000\times 1000}$, $\sigma(\vec A) \subset [2,3]$, $\bB\in\mathbb{R}^{900\times 900}$,  $\sigma(\vec B) \subset [-1.8,-0.5]$, $\bU\in\real^{900\times2}$, and $\bV\in\real^{2\times1000}$.
\begin{figure}
	\centering
	\begin{subfigure}{0.495\linewidth}
		\centering
		\includegraphics[width=\linewidth]{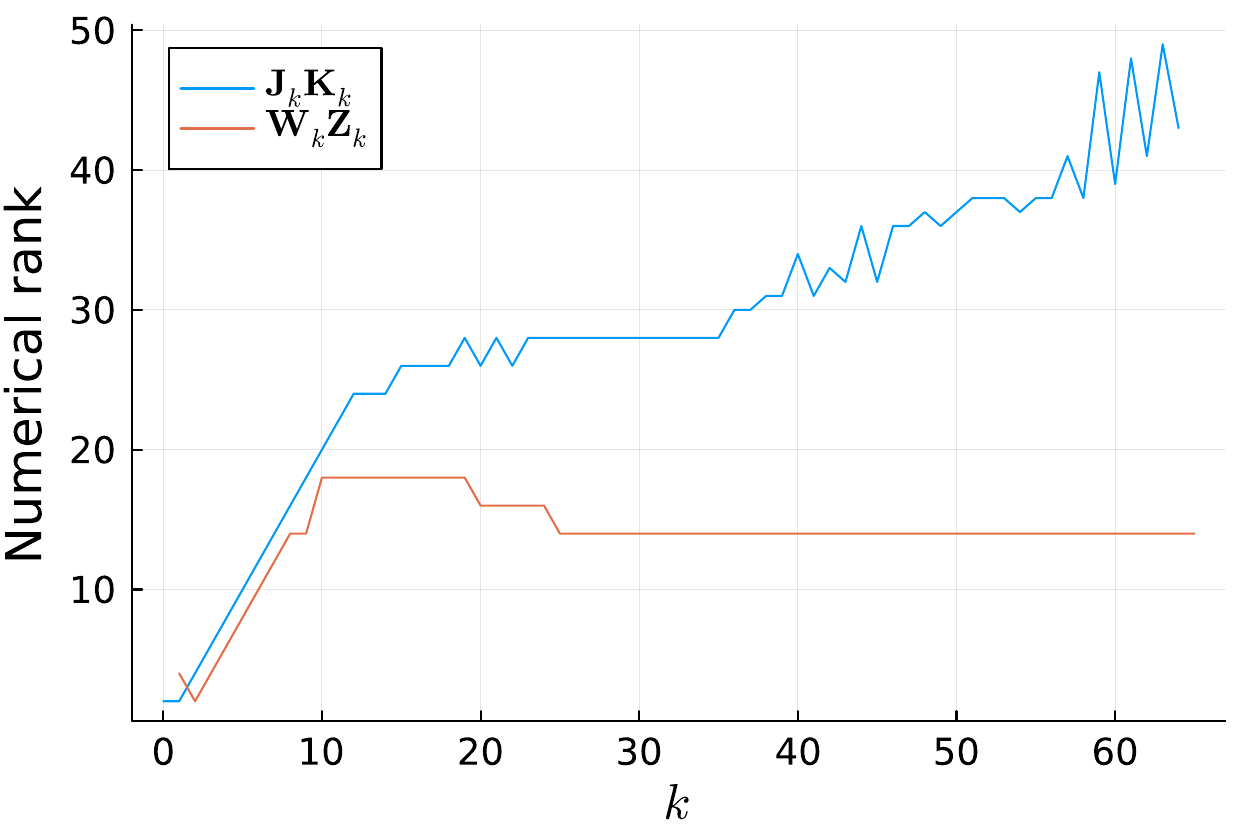}
	\end{subfigure}
	\begin{subfigure}{0.495\linewidth}
		\centering
		\includegraphics[width=\linewidth]{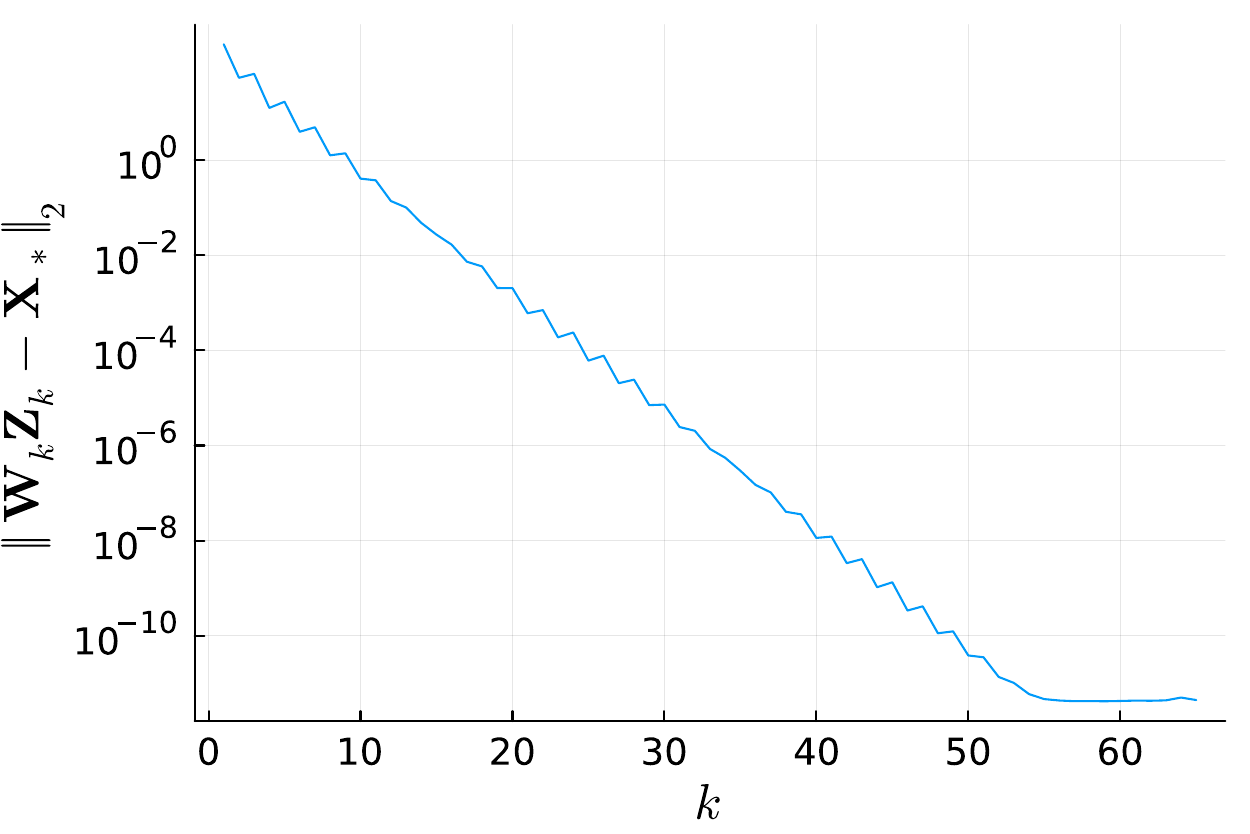}
	\end{subfigure}
\caption{Numerical rank of the iterates of Algorithm~\ref{alg:sylv_low_rank} with unweighted compression (left)  and error of the approximation at each iteration against the true solution $\bX_*$ of \eqref{eq:lr_sylv_eqn}, computed via $\texttt{sylvester}$, the built-in \Julia\ implementation of the Bartels--Stewart algorithm \cite{BartelsStewart} (right). Here, $\bA\in\mathbb{R}^{1000\times 1000}$, $\bB\in\mathbb{R}^{900\times 900}$, $\bU\bV$ is rank 2, and the numerical solution appears to have numerical rank 14. }
\label{fig:unweight_rank}
\end{figure}
The numerical rank of $\bJ_k\bK_k$ appears to grow, while the numerical rank of $\bW_k\bZ_k$ converges to that of the true solution. The reason for this is that $\bJ_k\bK_k$ is always scaled by a $p_j$-series coefficient $\alpha_k$ when $\bW_k\bZ_k$ is updated, but these coefficients decay exponentially as $k$ increases \cite[Lemma 4.10]{AkhIter}. Due to this decay and the fact that $\bJ_k\bK_k$ is used to generate $\bJ_{k+1}\bK_{k+1}$, one may consider enlarging the tolerance (as in Algorithm~\ref{alg:compress}) in the compression of $\bJ_k\bK_k$ by the decay rate of $|\alpha_k|$. See Section~\ref{sect:param} for a full discussion of this. In Figure~\ref{fig:weight_rank}, we plot the numerical rank of $\bJ_k\bK_k$ and $\bW_k\bZ_k$ when the compression of $\bJ_k\bK_k$ is weighted by the decay rate. We observe that the rank of $\bJ_k\bK_k$ now decays as the rank of $\bW_k\bZ_k$ converges to that of the true solution, but the convergence rate of the iteration is not noticeably affected. 
\begin{figure}
	\centering
	\begin{subfigure}{0.495\linewidth}
		\centering
		\includegraphics[width=\linewidth]{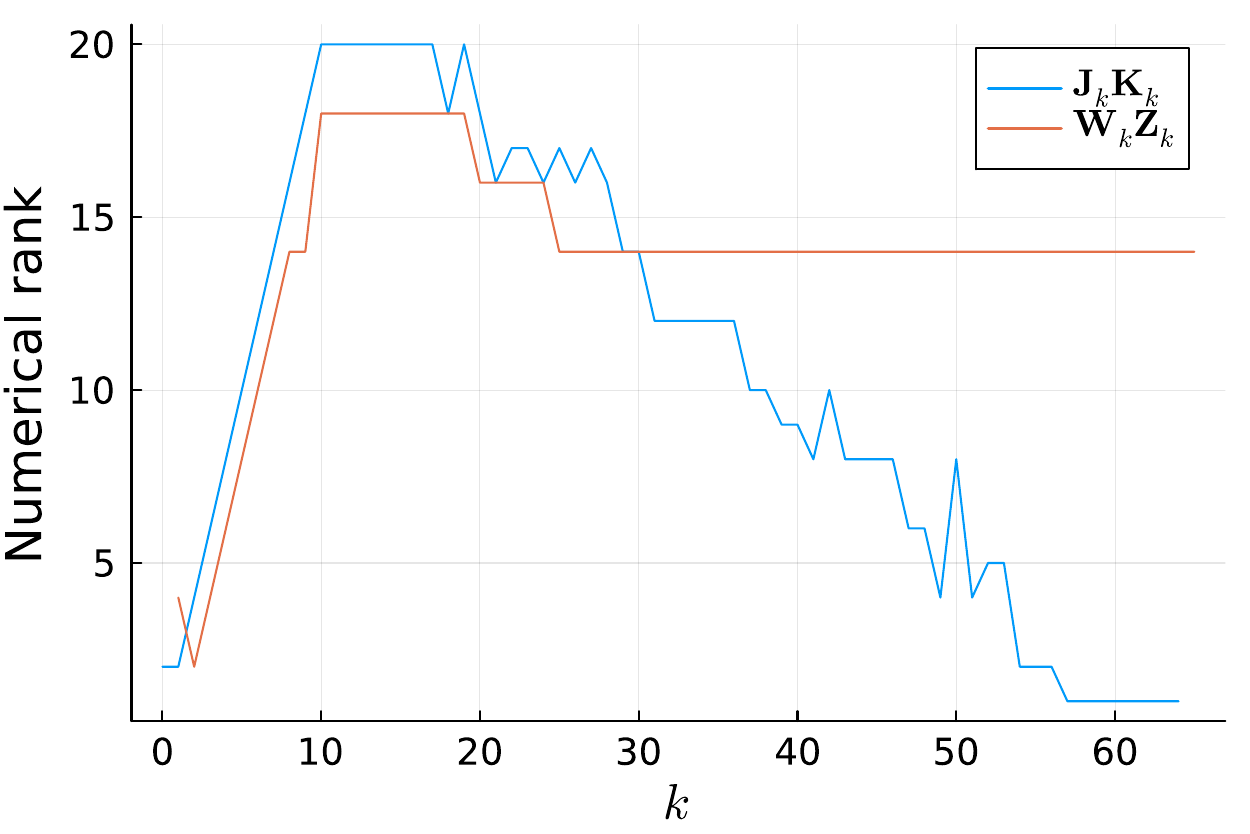}
	\end{subfigure}
	\begin{subfigure}{0.495\linewidth}
		\centering
		\includegraphics[width=\linewidth]{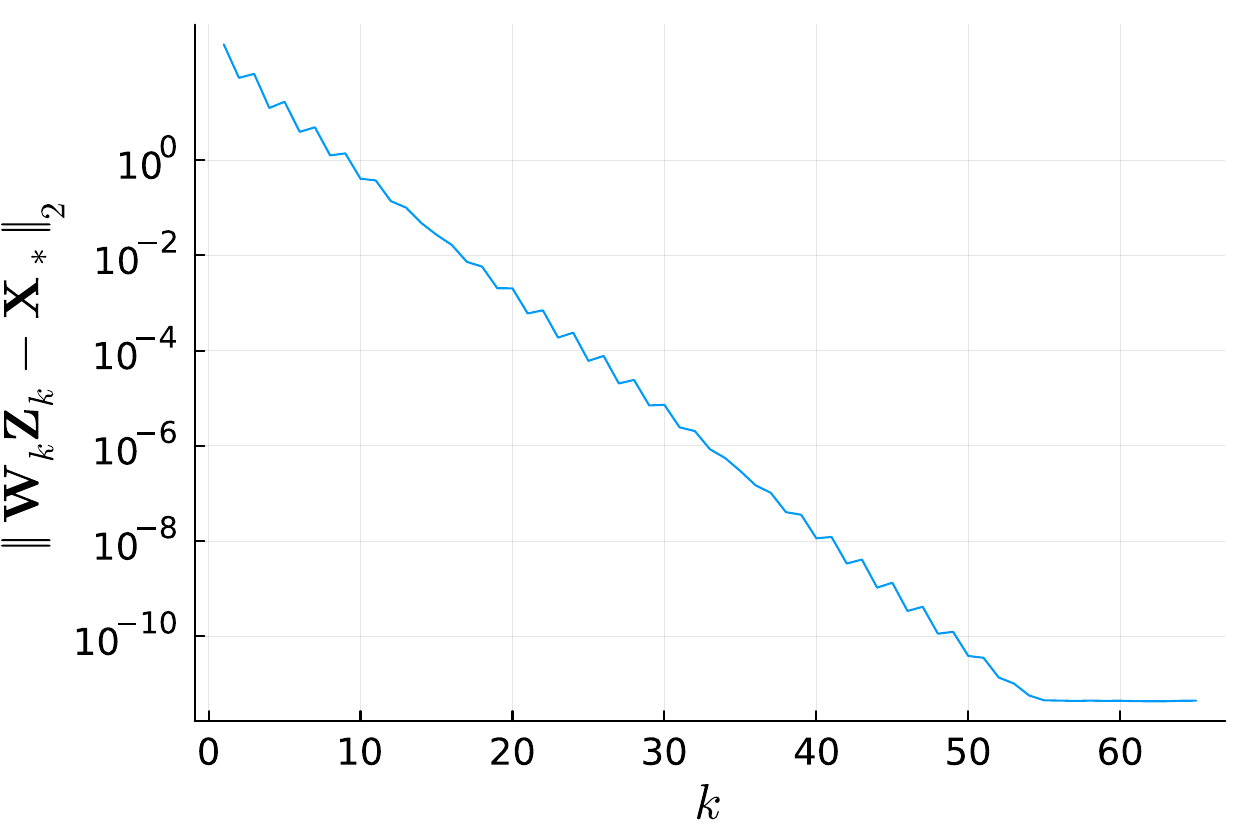}
	\end{subfigure}
\caption{Numerical rank of the iterates of Algorithm~\ref{alg:sylv_low_rank} with weighted compression (left) and error of the approximation at each iteration against the true solution $\bX_*$ (right). The problem is the same as that of Figure~\ref{fig:unweight_rank}, and only the tolerance for the compression of $\bJ_k\bK_k$ was changed.}
\label{fig:weight_rank}
\end{figure}

\subsection{Parameter tuning and heuristics}\label{sect:param}
We use $\mathfrak g$ to denote the Green’s function with pole at infinity, exterior to $\Sigma$; see \cite[Appendix A]{AkhIter}, for example, for its explicit construction. This function satisfies $\re\mathfrak{g}(z)=0$ for $z\in\Sigma$, and $\ex^{\mathfrak{g}(z)}$ is a conformal map from $\compl\setminus\Sigma$ to the exterior of the unit disk. A method to compute $\mathfrak g$ for a collection of disjoint intervals $\Sigma$ is described in \cite[Section 4.2]{RHISM}. For the relevant two interval case, \cite[Section 4.1.1]{AkhIter} gives an explicit formula that may be evaluated directly. Both methods are implemented in \cite{RecurrenceCoefficients.jl}. We plot the level curves of this function in Figure~\ref{fig:bernstein}. As discussed in Section~\ref{sect:intro}, these level curves are a higher genus analog of Bernstein ellipses. Define
    \begin{align}\label{eq:Bvarrho}
    B_\varrho = \{z \in \mathbb C ~:~ \ex^{\re \mathfrak g(z)} < \varrho\}, \quad \varrho > 1.
    \end{align}
    For all $\varrho > 1$, $B_\varrho$ is an open set that contains $\Sigma$.  We have the following lemma from \cite{AkhIter} that will inform compression thresholds in our algorithm:
\begin{lemma}\label{lem:bernstein}
    Let $\Sigma=\bigcup_{j=1}^{g+1}[\beta_j,\gamma_j]\subset\real$ be a union of disjoint intervals and assume that $f(z)$ is analytic for $z \in B_\varrho$, $\varrho > 1$, and $|f(z)|\leq M$ for $z \in B_\varrho$. Let $w$ be a weight function of the form~\eqref{eq:gen_weight} and denote the corresponding orthonormal polynomials by $\left(p_j(x)\right)_{j=0}^\infty$. Then, there exists a constant $C>0$, independent of $f$, such that
    \begin{equation*}
        \left|\langle f,p_\ell\rangle_{L^2_w(\Sigma)}\right|\leq CM\varrho^{-\ell}, \quad \ell\in\mathbb{N}.
    \end{equation*}
\end{lemma}

To determine the convergence rate $\varrho^{-1}$ when $f$ is the sign function, consider $\Sigma=[\beta_1,\gamma_1]\cup[\beta_2,\gamma_2]$, $\gamma_1<\beta_2$. The largest value of $\varrho$ such that $f(z) = \mathrm{sign}(z)$ can be defined so that it is analytic in $B_\varrho$ will correspond to the value of $\varrho$ where the disjoint level curves around $[\beta_1,\gamma_1]$ and $[\beta_2,\gamma_2]$ first intersect, as $\varrho$ increases. This intersection point is precisely the unique local maximum of $\re\mathfrak g(z)$ for $z\in[\gamma_1,\beta_2]$. The location of this intersection follows from \cite[Appendix A]{AkhIter} and is given by\footnote{A formula for $z^*$ may also be obtained by differentiating the formulae in \cite[Section 3.1 and 4.1.1]{AkhIter}; however, given the ease of computing $\mathfrak g$ via the formula in \cite[Section 4.1.1]{AkhIter}, we find that it is fastest and easiest to find the intersection point $z^*$ via a numerical method such as the golden-section search.}
    \begin{equation*}
    z^*=\frac{\int_{\gamma_1}^{\beta_2}\frac{z\df z}{\sqrt{z-\beta_1}\sqrt{z-\gamma_1}\sqrt{z-\beta_2}\sqrt{z-\gamma_2}}}{\int_{\gamma_1}^{\beta_2}\frac{\df z}{\sqrt{z-\beta_1}\sqrt{z-\gamma_1}\sqrt{z-\beta_2}\sqrt{z-\gamma_2}}}.
    \end{equation*}    

    Given the intersection point $z^*$ of the level curves, we have the following bound on the coefficients of a $p_j$-series expansion of the sign function:
    \begin{equation*}
    |\alpha_j|=\left|\langle \sign,p_j\rangle_{L^2_w(\Sigma)}\right|\leq C\varrho^{-j},\quad \varrho=\ex^{\re\mathfrak g(z^*)},
    \end{equation*}
    for some constant $C>0$ that cannot in general be computed. 
    \begin{remark}
        In the symmetric case $\Sigma=[-\gamma,-\beta]\cup[\beta,\gamma]$, $0<\beta<\gamma$, we have that $z^*=0$ and \cite{Eremenko2011}
        \[
        \varrho=\sqrt{\frac{\frac{\gamma}{\beta}-1}{\frac{\gamma}{\beta}+1}}.
        \]
    \end{remark}  
    We use this bound to inform a heuristic for the weighted compression described in Section~\ref{sect:low-rank}. We find that the weighted compression step tends to behave well when the tolerance is simply multiplied by $\frac{\varrho^k}{c}$ at iteration $k$ where $\varrho$ is computed a priori and $c$ is a chosen constant taken in place of $C$. Since the coefficients $\alpha_j$ are computed a priori, one may simply select $c$ such that $|\alpha_j|\leq c\varrho^{-j}$ by plotting or other means; however, we find that it typically suffices to take $c=5$. One can always plot the series coefficients if uncertain. See Figure~\ref{fig:coeff_rate} for plots of the coefficients against the heuristic with $c=5$.
    \begin{figure}
        \centering
    	\begin{subfigure}{0.495\linewidth}
    		\centering
    		\includegraphics[width=\linewidth]{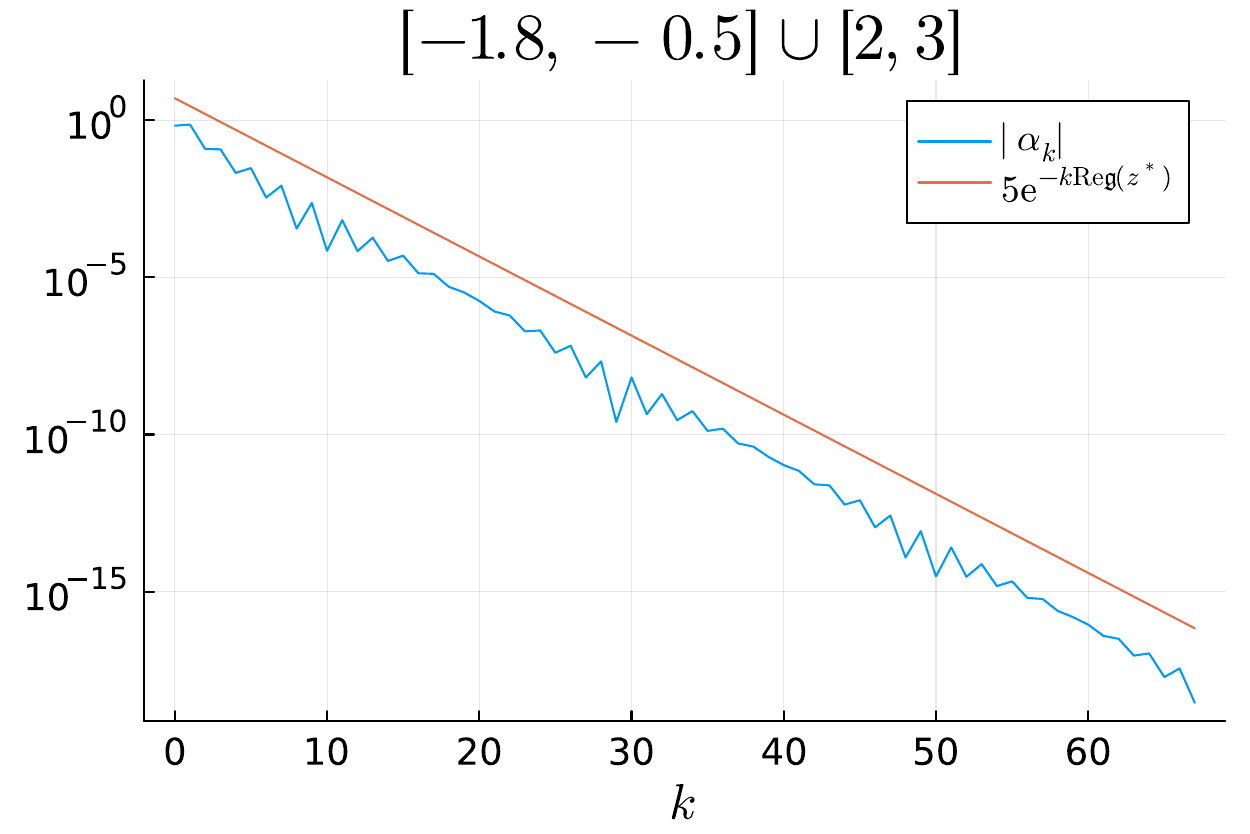}
    	\end{subfigure}
    	\begin{subfigure}{0.495\linewidth}
    		\centering
    		\includegraphics[width=\linewidth]{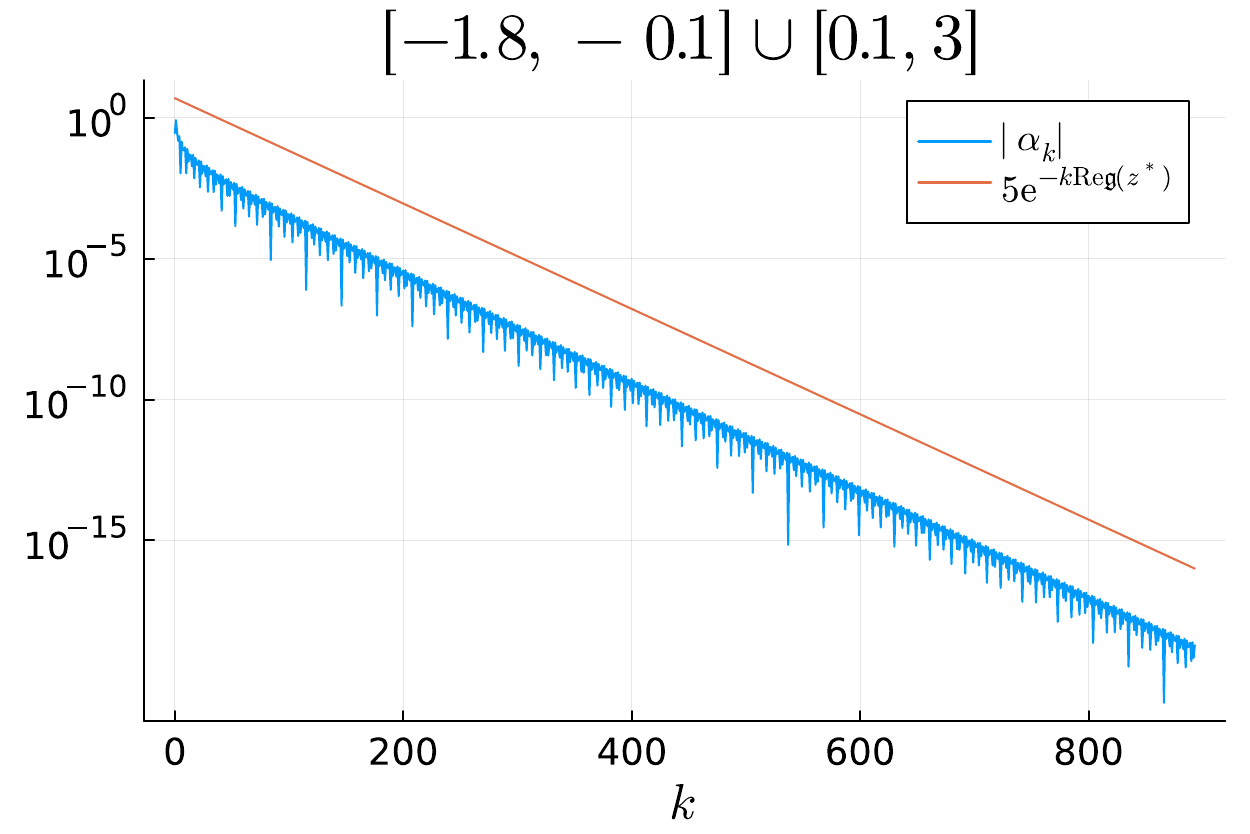}
    	\end{subfigure}
        \caption{Modulus of the coefficients of an Akhiezer polynomial series approximation to the sign function on $[-1.8,-0.5]\cup[2,3]$ (left) and $[-1.8,-0.1]\cup[0.1,3]$ (right) plotted against the convergence rate heuristic with $c=5$ until the heuristic is smaller than $10^{-16}$.}
        \label{fig:coeff_rate}
    \end{figure}
 
\subsection{Convergence and stopping criteria}\label{sect:conv}
In this subsection, we present a convergence analysis that is largely a simplified version of that of \cite[Section 4]{AkhIter} under the additional assumption that the eigenvalues of the matrix lie in the interior of $\Sigma$. Some results are reproduced directly without proof. We begin with the following definition to simplify the analysis:
\begin{definition}\label{def:generic}
We say that a matrix is generic if it is diagonalizable and none of its eigenvalues lie at endpoints of $\Sigma$.
\end{definition}
See \cite[Section 4.1.1]{AkhIter} for remarks on how the analysis can be extended to nongeneric matrices.
The following lemma describes the relevant behavior of the polynomials:
\begin{lemma}\label{lem:poly_asymp}
	Let $\Sigma=\bigcup_{j=1}^{g+1}[\beta_j,\gamma_j]\subset\real$ be a union of disjoint intervals and let $w$ be a weight function of the form \eqref{eq:gen_weight}. Denote the corresponding orthonormal polynomials by $\left(p_j(z)\right)_{j=0}^\infty$, fix $\epsilon>0$, and set $V=\{z\in\compl~:~|z-\beta_j|\geq\epsilon,~|z-\gamma_j|\geq\epsilon,~j=1,\ldots,g+1\}$. Then,
	\begin{equation*}
		\begin{aligned}
		&p_n(z)=\hat\delta_n(z)\ex^{n\mathfrak{g}(z)},&\quad&2\pi\im\mathcal{C}_\Sigma\left[p_nw\right](z)=\delta_n(z)\ex^{-n\mathfrak{g}(z)},\quad z\in V,\\
		\end{aligned}
	\end{equation*}
where $\delta_n(z)$ and $\hat\delta_n(z)$ are uniformly bounded in both $n$ and $z\in V$.
\end{lemma}
In the following, $\|\diamond\|_2$ will denote the Euclidean 2-norm and the induced matrix norm.
\begin{theorem}\label{thm:conv_rate}
Under the assumptions of Lemma \ref{lem:bernstein}, suppose that $\bM$ is generic and $\sigma(\bM) \subset \Sigma$.
% \begin{align*}
%     \sum_{j=0}^{\infty}\varrho^{-j}\|p_j(\bM)\|_2 < \infty.
% \end{align*}
Then, there exists a constant $C' >0$, independent of $f$, such that the approximation~\eqref{eq:series_trunc} satisfies
\begin{equation*}
\left\|f(\bM)-\bF_k\right\|_2\leq C'M\|\bV\|_2\left\|\bV^{-1}\right\|_2\frac{\varrho^{-k}}{1-{\varrho}^{-1}},
\end{equation*}
where $\bM$ is diagonalized as $\bM=\bV\bLambda\bV^{-1}$.
\end{theorem}
\begin{proof}
By the triangle inequality,
\begin{equation*}
\begin{aligned}
\left\|f(\bM)-\bF_k\right\|_2&=\left\|\sum_{j=k}^\infty\alpha_jp_j(\bM)\right\|_2\leq\sum_{j=k}^\infty CM\varrho^{-j}\|\bV\|_2\left\|p_j(\bLambda)\right\|_2\left\|\bV^{-1}\right\|_2\\&=
CM\|\bV\|_2\left\|\bV^{-1}\right\|_2\sum_{j=k}^\infty\varrho^{-j}\max_{\lambda\in\sigma(\bM)}\left|p_j(\lambda)\right|.
\end{aligned}
\end{equation*}
Since $\lambda$ lies in the interior of $\Sigma$ for all $\lambda\in\sigma(\bM)$, by Lemma \ref{lem:poly_asymp}, there exists some constant $c$ such that
\begin{equation*}
\left|p_j(\lambda)\right|=\left|\hat\delta_j(\lambda)\right|\ex^{j\re\mathfrak g(\lambda)}\leq c,
\end{equation*}
for all $j$ and $\lambda\in\sigma(\bM)$. The theorem follows from using $C' = c C$ and
\begin{equation}\label{eq:geo_series}
\sum_{j=k}^{\infty}\varrho^{-j}= \frac{\varrho^{-k}}{1-\varrho^{-1}}.
\end{equation}
\end{proof}
Theorem~\ref{thm:conv_rate} implies that the Akhiezer iteration converges, at worst, at a geometric rate given by $\varrho^{-1}$. If $f$ is the sign function, then $M=1$. This theorem then implies that there exists some constant $D_{\bM}>0$, dependent on the eigenvectors of $\bM$, such that
\begin{equation}\label{eq:big_conv_rate}
\left\|\sign(\bM)-\bF_k\right\|_2\leq D_{\bM}\frac{\varrho^{-k}}{1-\varrho^{-1}}.
\end{equation}
In particular, for a given tolerance $\epsilon> 0$, we require
\begin{equation}\label{eq:iter_count}
k = \left\lceil -\log_\varrho\frac{\epsilon(1-\varrho^{-1})}{D_\bM} \right\rceil,
\end{equation}
iterations to ensure that 
\begin{equation*}
\left\|\sign(\bM)-\bF_k\right\|_2<\epsilon.
\end{equation*}
Since it depends on eigenvectors and orthogonal polynomial asymptotics, $D_\bM$ cannot be cheaply computed in general. Like with the orthogonal polynomial series coefficients, we hypothesize an upper bound $D_\bM\leq10\ell$ for generic $\bM\in\compl^{\ell\times \ell}$ that are not highly nonnormal. One can always run more iterations if this hypothesized bound appears insufficient.

We note that in finite-precision, the iteration will saturate when the coefficients $\alpha_j$ reach machine precision $\varepsilon_{\mathrm{mach}}$. We use
\begin{equation*}
    %k>-\log_\varrho\left(\max\left\{\frac{\epsilon(1-\varrho^{-1})}{10\ell},\frac{\varepsilon_{\mathrm{mach}}}{5}\right\}\right),
    k = \left\lceil \min \left\{ - \log_\varrho\left(\frac{\epsilon(1-\varrho^{-1})}{10\ell}\right), - \log_\varrho\left(\frac{\varepsilon_{\mathrm{mach}}}{5}\right)    \right\} \right\rceil,
\end{equation*}
iterations in place of \eqref{eq:iter_count}.

Now, consider $\bH$ from~\eqref{eq:block_sylv_eqn} and let $\vec X_*$ denote the true solution of the Sylvester equation \eqref{eq:Genmat}.  
Of course, the error of interest is $\bX_k - \vec X_*$, not $\mathrm{sign}(\vec H)-\vec F_k $. Naively, we have
\begin{equation}\label{eq:conv_heur}
\|\bX_k-\bX_*\|_2\leq\frac{1}{2}\left\|\sign(\bH)-\bF_k\right\|_2\leq \frac{D_{\bH}}{2}\frac{\varrho^{-k}}{1-\varrho^{-1}}.
\end{equation}
Empirically, we find that the error in this iteration does not decay faster than the geometric rate given by $\varrho^{-1}$, so we do not seek to improve it further --- we use the same hypothesized upper bound $D_\bH\leq10(n+m)$ to determine the iteration count \eqref{eq:iter_count}. In Figure~\ref{fig:err_heur}, we plot the errors $\|\bX_k-\bX_*\|_2$ at each iteration against this heuristic and observe that it does indeed capture the convergence rate and provide an upper bound until errors saturate.
\begin{figure}
	\centering
	\begin{subfigure}{0.495\linewidth}
		\centering
		\includegraphics[width=\linewidth]{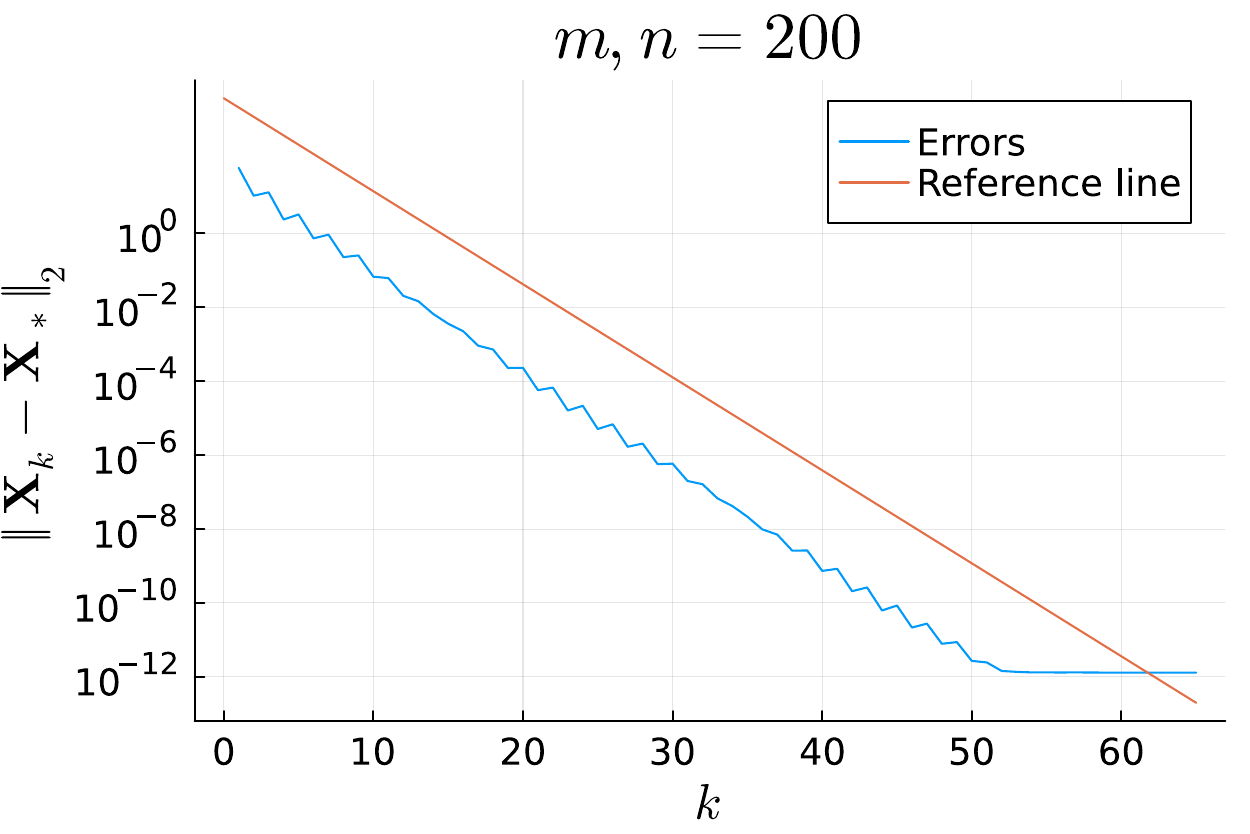}
	\end{subfigure}
	\begin{subfigure}{0.495\linewidth}
		\centering
		\includegraphics[width=\linewidth]{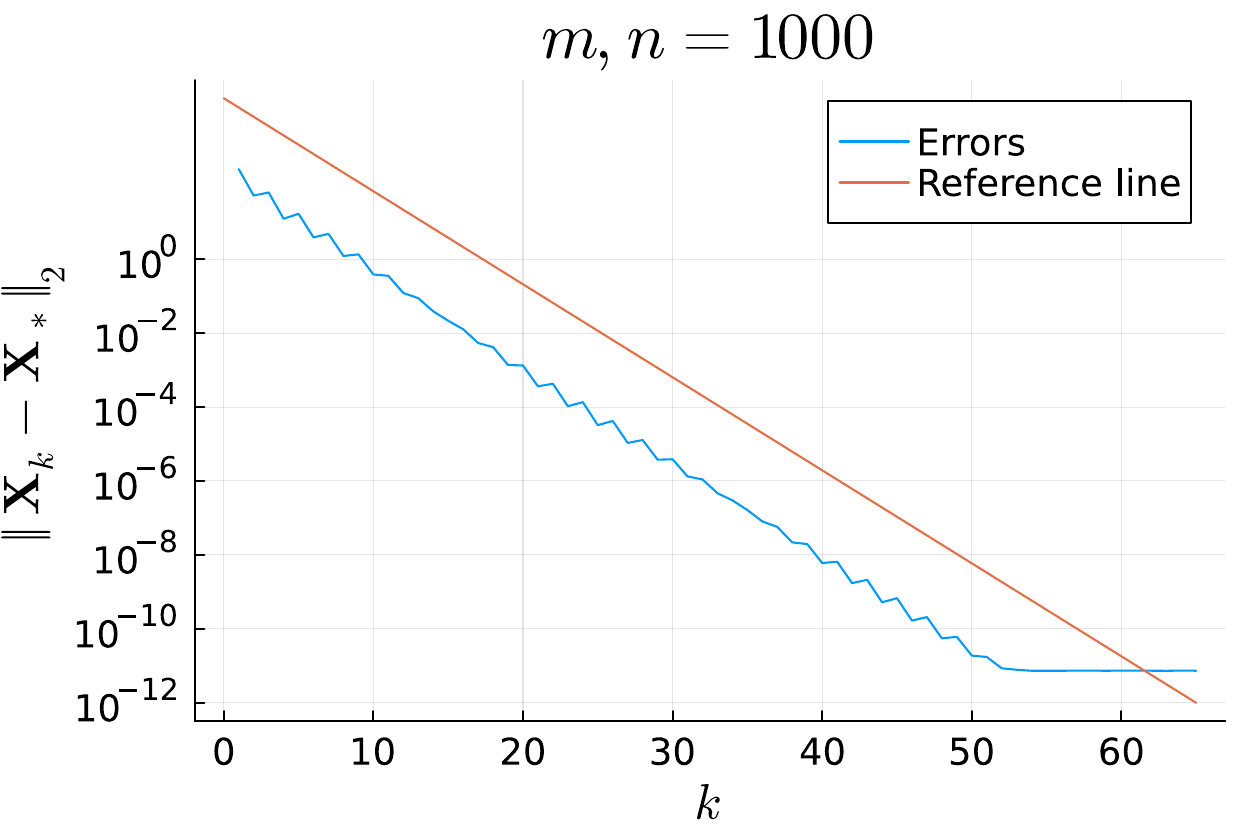}
	\end{subfigure}
\caption{2-norm error for $\bX_k$ against the true solution $\bX_*$ to~\eqref{eq:Genmat} (computed via $\texttt{sylvester}$ in \Julia) at each iteration where $\bC$ is rank 2, $\bA,\bB\in\real^{200\times200}$ (left), $\bA,\bB\in\real^{1000\times1000}$ (right), $\sigma(\bA)\subset[2,3]$, and $\sigma(\bB)\subset[-1.8,-0.5]$. Here, the reference line denotes the convergence rate heuristic \eqref{eq:conv_heur} with $D_\bH$ replaced by the hypothesized upper bound $10(n+m)$.}
\label{fig:err_heur}
\end{figure}

\subsection{Quadrature error analysis}
In this subsection, we discuss the inexactness of quadrature rules. These results can also be found in \cite{AkhIter}.  One assumption that was made in the preceding convergence analysis was that the coefficients $\alpha_j$ in the $p_j$-series expansion of $f$ were computed exactly. To relax this assumption, let 
\begin{align}\label{eq:rhom}
\rho_m(x):=\int_\Gamma \frac{f(z)}{z-x}\df z-\sum_{j=1}^m \frac{f(z_j)}{z_j-x}w_j,
\end{align}
be a measure of the quadrature rule error. Denote the $k$th iteration of Algorithm~\ref{alg:akh_func}, with the approximate coefficients, by
\begin{equation*}
f_{k,m}(\bM):=-\sum_{\ell=0}^{k-1}\left(\sum_{j=1}^mf(z_j)w_j\mathcal{C}_\Sigma\left[p_\ell w\right](z_j)\right)p_\ell(\bM).
\end{equation*}

\begin{theorem}\label{thm:quad_inac}
Under the assumptions of Lemma \ref{lem:bernstein}, let $\bM$ be a square matrix and suppose that
\begin{align*}
    \sum_{\ell=0}^{\infty}\varrho^{-\ell}\|p_\ell(\bM)\|_2 < \infty.
\end{align*}
Further suppose that $\Gamma,f$ are as in Definition~\ref{def:fan} and $\{z_j\}_{j=1}^m$, $\{w_j\}_{j=1}^m$ are quadrature nodes and weights for $\Gamma$, respectively. Then, there exist $\varepsilon_{m,\ell}$ satisfying $\left|\varepsilon_{m,\ell}\right|\leq\|\rho_m\|_\infty := \max_{z \in \Sigma} |\rho_m(z)|$ and a constant $C>0$, independent of $f$, such that
\begin{equation*}
\left\|f(\bM)-f_{k,m}(\bM)+\sum_{\ell=0}^{k-1}\varepsilon_{m,\ell}p_\ell(\bM)\right\|_2\leq CM\sum_{\ell=k}^{\infty}\varrho^{-\ell}\|p_\ell(\bM)\|_2.
\end{equation*}
\end{theorem}   
In particular, following the proof of Theorem~\ref{thm:conv_rate}, Theorem~\ref{thm:quad_inac} implies that when a generic $\bM$ is diagonalized as $\bM = \bV \bLambda \bV^{-1}$ and satisfies $\sigma(\bM)\subset\Sigma$, there exists some constant $C'>0$ such that
\begin{equation}\label{eq:quad_inac}
 \|f(\bM) - f_{k,m}(\bM)\|_2 \leq \|\bV\|_2 \|\bV^{-1}\|_2 \left(  \|\rho_m\|_\infty \sum_{\ell = 0}^{k-1} \|p_\ell (\bLambda)\|_2 + C' M \frac{\varrho^{-k}}{1-{\varrho}^{-1}} \right).
\end{equation}
Since $\|p_\ell(\bM)\|_2 = \OO(1)$ when $\bM$ is generic and $\sigma(\bM) \subset \Sigma$, the former term in the sum is likely to be small; however, as $k$ increases, this term will eventually become large. See Appendix~\ref{ap:eig} for a discussion of why the errors still tend to remain small in this case. This appendix also contains convergence analysis for generic matrices whose eigenvalues may not lie in $\Sigma$.

\subsection{Complexity and storage}
To analyze the complexity of Algorithm~\ref{alg:sylv_low_rank}, we assume that the numerical rank of $c\varrho^{-k}\bJ_k\bK_k$ is given by $R(k)$ and that $R(k)=0$ for all $k\geq k^*$ and some $k^*\in\mathbb{N}$. That is, we assume that the numerical rank of the weighted compression of $\bJ_k\bK_k$ is bounded and eventually becomes zero. This implies that the numerical rank of $\bW_k\bZ_k$ is also bounded, and we denote this quantity by $\hat R(k)$. Suppose that $T_\bA$ is the cost of a left matrix-vector multiplication by $\bA$ and $T_\bB$ is the cost of a right matrix-vector multiplication by $\bB$.

At the $k$th iteration of Algorithm~\ref{alg:sylv_low_rank}, constructing the matrices $\bV p_k(\bA)$ and $p_k(\bB)\bU$ requires at most
\begin{equation*}
r(T_\bA+T_\bB)+5r(m+n),
\end{equation*}
arithmetic operations. The construction of $\bJ_k$ and $\bK_k$ requires at most
\begin{equation*}
(2R(k-1)+R(k-2)+r)m+R(k-1)T_\bA,
\end{equation*}
arithmetic operations, and the construction of $\bW_{k+1}$ and $\bZ_{k+1}$ requires at most
\begin{equation*}
(R(k)+r)m,
\end{equation*}
arithmetic operations. The compression of $\bJ_k\bK_k$ $\texttt{COMPRESS}\left(\bJ_{k},\bK_{k},\frac{\varrho^k\epsilon}{c}\right)$ requires 
\begin{equation*}
\OO\left((2R(k-1)+R(k-2)+r)^3+(2R(k-1)+R(k-2)+r)^2(m+n)\right),
\end{equation*}
arithmetic operations, while the compression of $\bW_k\bZ_k$ $\texttt{COMPRESS}\left(\bW_{k},\bZ_{k},\epsilon\right)$ requires 
\begin{equation*}
\OO\left((R(k)+\hat R(k)+r)^3+(R(k)+\hat R(k)+r)^2(m+n)\right),
\end{equation*}
arithmetic operations. Finally, the cost of computing the recurrence coefficients and $p_j$-series coefficients is assumed to be $\OO(1)$ --- computing the Akhiezer data requires $\OO(k)$ operations for $k$ iterations.

To bound the complexity of Algorithm~\ref{alg:sylv_low_rank} run for $k$ iterations, denote 
\begin{equation*}
\cR:=\max_{j\in\mathbb{N}}\left\{R(j),\hat R(j)\right\}.
\end{equation*}
Then, to run Algorithm~\ref{alg:sylv_low_rank} for $k$ iterations, we need
\begin{equation*}
% k\left(r(T_\bA+T_\bB+7m+5n)+\cR(T_\bA+4m)\right)+
\OO\left(kr(T_\bA+T_\bB)+k\cR T_\bA+k(\cR+r)^3+k(\cR+r)^2(m+n)\right),
\end{equation*}
arithmetic operations. Note that when the stopping criterion \eqref{eq:iter_count} is applied, $k=\OO\left(\log(m+n)\right)$.

In the general case when $T_\bA=\OO(n^2)$, $T_\bB=\OO(m^2)$, $r,\cR\ll m,n$, and the stopping criterion \eqref{eq:iter_count} is applied, Algorithm~\ref{alg:sylv_low_rank} requires\footnote{Note that 
% \textcolor{red}{Do we need this footnote?}
$\OO\left((m^2+n^2)\log\left(m+n\right)\right)$ is equivalent to $\OO\left(m^2\log m+n^2\log n\right)$.} $\OO\left(m^2\log m+n^2\log n\right)$ arithmetic operations. In contrast, the Bartels--Stewart algorithm requires $\OO\left(m^3+n^3\right)$ arithmetic operations \cite{BartelsStewart}. Furthermore, methods such as Alternating-Directional-Implicit (ADI) iterations that require computing matrix inverses will also necessitate $\OO\left(m^3+n^3\right)$ arithmetic operations in general.

Under these same assumptions, a maximum of $(10\cR+6r)(m+n)$ stored entries are required at any given iteration. Therefore, in the general case $r,\cR\ll m,n$, Algorithm~\ref{alg:sylv_low_rank} requires at most $\OO(m+n)$ stored entries. By contrast, a standard polynomial Krylov method such as \cite[Algorithm 5]{simoncini2016computational} can require $\OO(mn)$ (or worse) stored entries when the convergence rate is slow.

% Let $\bA\in\compl^{n\times n}$ and $\bB\in\compl^{m\times m}$ and let $r$ denote the rank of $\bU\bV$. Supposing that $T_\bA$ is the complexity of a left matrix-vector multiplication by $A$ and $T_\bB$ is the complexity of a right matrix-vector multiplication by $B$, running Algorithm~\ref{alg:sylv_low_rank} for $k$ steps requires less than
% \begin{equation*}
% kr(T_\bA+T_\bB)+5krm+6krn+5kR(r)(n+T_\bA)+\OO\left(2k(3R(r)+r)^3+2k(3R(r)+r)^2(m+n)\right)+\OO(4k),
% \end{equation*}
% arithmetic operations. Here, the $\OO\left(2k(3R(r)+r)^3+2k(3R(r)+r)^2(m+n)\right)$ term denotes the cost of the two compression steps, and the $\OO(4k)$ term denotes the cost of computing the recurrence coefficients and $p_j$-series coefficients.

\section{Method 2: inverting the Sylvester operator}\label{sect:Method2}
For a Sylvester equations~\eqref{eq:Genmat}, define the Sylvester operator $\scrS_{\bA,\bB}:\compl^{m\times n}\to\compl^{m\times n}$ by 
\begin{equation}\label{eq:sylv_op}
\scrS_{\bA,\bB}\bY=\bY\bA-\bB\bY.
\end{equation}
Our second method directly builds a polynomial approximation to $\scrS_{\bA,\bB}^{-1}\bC$, the solution to~\eqref{eq:Genmat}. 
% Because~\eqref{eq:Genmat} is equivalent to the large linear system
% \begin{equation}\label{eq:kron_sys}
% \left(\bA^T\otimes\bI_m-\bI_n\otimes\bB\right)\mathrm{vec}(\bX)=\mathrm{vec}(\bC),
% \end{equation}
% $\scrS_{\bA,\bB}$ corresponds to the linear operator $\bA^T\otimes\bI_m-\bI_n\otimes\bB$. 
Applying the Akhiezer iteration to solve~\eqref{eq:sylv_op} requires only the action of $\scrS_{\bA,\bB}$ on a matrix, so it can be implemented by iteratively applying~\eqref{eq:sylv_op}. In particular, we have the following theorem:
\begin{theorem}\label{thm:kron_iter}
    Assume that $\sigma(\scrS_{\bA,\bB})$ is contained in the interior of $\Sigma=\bigcup_{j=1}^{g+1}[\beta_j,\gamma_j]\subset\real$ where $0\notin\Sigma$. Let $w$ be a weight function of the form~\eqref{eq:gen_weight} and denote the corresponding orthonormal polynomials by $\left(p_j(x)\right)_{j=0}^\infty$ with recurrence coefficients $(a_j,b_j)_{j=0}^\infty$. Then, the solution of the Sylvester equation~\eqref{eq:Genmat} is given by
    \begin{equation*}
        \bX=\sum_{j=0}^\infty\alpha_j\bP_j,\quad \alpha_j=2\pi\im\cC\left[p_jw\right](0),
    \end{equation*}
    where the matrices $\bP_j$ are given by the recurrence
    \begin{equation*}
		\begin{aligned}
			&\bP_0=\bC,\\
			&\bP_1=\frac{1}{b_0}(\bP_0\bA-\bB\bP_0-a_0\bP_0),\\
			&\bP_j=\frac{1}{b_{j-1}}(\bP_{j-1}\bA-\bB\bP_{j-1}-a_{j-1}\bP_{j-1}-b_{j-2}\bP_{j-2}),\quad j\geq2.
		\end{aligned}
	\end{equation*}
\end{theorem}
\begin{proof}
    By Lemma~\ref{lem:poly_asymp}, the function $f(x)=1/x$ is represented as an absolutely convergent $p_j$-series on $\Sigma$ as
    \begin{equation*}
    \frac{1}{x}=\sum_{j=0}^\infty\alpha_jp_j(x),\quad \alpha_j=\int_\Sigma\frac{p_j(x)w(x)}{x}\df x=2\pi\im\cC\left[p_jw\right](0).
    \end{equation*}
    By the assumptions on $\sigma(\scrS_{\bA,\bB})$,
    \begin{equation*}
    \bX = \scrS_{\bA,\bB}^{-1} \bC=\sum_{j=0}^\infty\alpha_jp_j\left(\scrS_{\bA,\bB}\right) \bC.
    \end{equation*}
\end{proof}
As before, the assumption that $\sigma(\scrS_{\bA,\bB})\subset\Sigma$ is unnecessary, see Remark~\ref{rem:eigs}. As with Method 1, truncating the infinite series of Theorem~\ref{thm:kron_iter} directly yields an iterative method for solving the Sylvester equation~\eqref{eq:Genmat}; however, the recurrence and series coefficients can be made more explicit in the typical use case. The eigenvalues of $\scrS_{\bA,\bB}$ are precisely the set $\{\lambda-\mu\}_{\lambda\in\sigma(\bA), \mu\in\sigma(\bB)}=\sigma(\bA)-\sigma(\bB)$, so intervals containing $\sigma(\scrS_{\bA,\bB})$ are easily determined from those containing $\sigma(A)$ and $\sigma(B)$. In particular, if $\sigma(\bB)\subset[\beta_1,\gamma_1]$ and $\sigma(\bA)\subset[\beta_2,\gamma_2]$, $\gamma_1<\beta_2$, then $\sigma(\scrS_{\bA,\bB})\subset\Sigma=[\beta_2-\gamma_1,\gamma_2-\beta_1]$. In this case, the iteration can be constructed with Chebyshev polynomials (shifted and scaled to $\Sigma$) and their particularly simple recurrence relation. An implementation of the iteration in this case is given as Algorithm~\ref{alg:sylv_inv} (cf. \cite[Algorithm 1]{AkhIter}).
\begin{algorithm}
		\caption{$\bX\bA-\bB\bX=\bC$ solver via the Akhiezer iteration for $1/x$}\label{alg:sylv_inv}
		\textbf{Input: }{Matrices $\bA,\bB,\bC$ and parameters $\alpha,c$ such that $\sigma(\bA)-\sigma(\bB)\subset[\alpha-c,\alpha+c]$.}\\
        Set $\varrho^{-1}=-\frac{\alpha}{c}+\sqrt{\frac{\alpha}{c}-1}\sqrt{\frac{\alpha}{c}+1}.$\\
		\For{k=0,1,\ldots}{
			\uIf{k=0}{
				Set $\mathbf P_0=\bC$.\\
                Set $S_0=\frac{1}{\sqrt{\alpha-c}\sqrt{\alpha+c}}$.\\
                Set $\mathbf X_{1}= S_0\bP_0$.
			}
			\uElseIf{k=1}{
				Set $\mathbf{P}_1=\frac{1}{c}(\bP_0\bA-\bP_0\bB-\alpha\bP_0)$.\\
                Set $S_1=S_0\varrho^{-1}$.\\
                Set $\mathbf X_{2}=\mathbf X_1 + 2S_1\bP_1$.
			}
   %          \uElseIf{k=2}{
			% 	Set $\mathbf{P}_2=\frac{2}{c}(\bP_1\bA-\bP_1\bB-\alpha\bP_1)-\sqrt{2}\bP_{0}$.\\
   %              Set $S_2=S_1\varrho^{-1}$.
			% }
			\Else{
				Set $\mathbf{P}_k=\frac{2}{c}(\bP_{k-1}\bA-\bP_{k-1}\bB-\alpha\bP_{k-1})-\bP_{k-2}$.\\
                Set $S_k=S_{k-1}\varrho^{-1}$.\\
                Set $\mathbf X_{k+1}=\mathbf X_k + 2S_k\bP_k$.
			}
			\If{\emph{converged}}{\Return{$\bX_{k+1}$}.}
		}
\end{algorithm}
In the single-interval case, this method is essentially a modification of the classical Chebyshev iteration applied to solve the equivalent linear system $\left(\bA^T\otimes\bI_m-\bI_n\otimes\bB\right)\mathrm{vec}(\bX)=\mathrm{vec}(\bC)$. See \cite[Appendix C]{AkhIter} for a discussion of this.
When $\Sigma$ consists of more than one interval, Algorithm~\ref{alg:sylv_inv} (and Algorithm~\ref{alg:sylv_inv_lr} below) may be generalized by constructing $\bP_k$ and $S_k$ with the Akhiezer data as in Theorem~\ref{thm:kron_iter} and Algorithm~\ref{alg:akh_func}. In Figure~\ref{fig:mult_int_inv}, we include an example where $\bA$ has a single outlier eigenvalue, and we observe that convergence is accelerated when $\Sigma$ is chosen to consist of two intervals instead of one.
\begin{figure}
	\centering
    \includegraphics[width=0.6\linewidth]{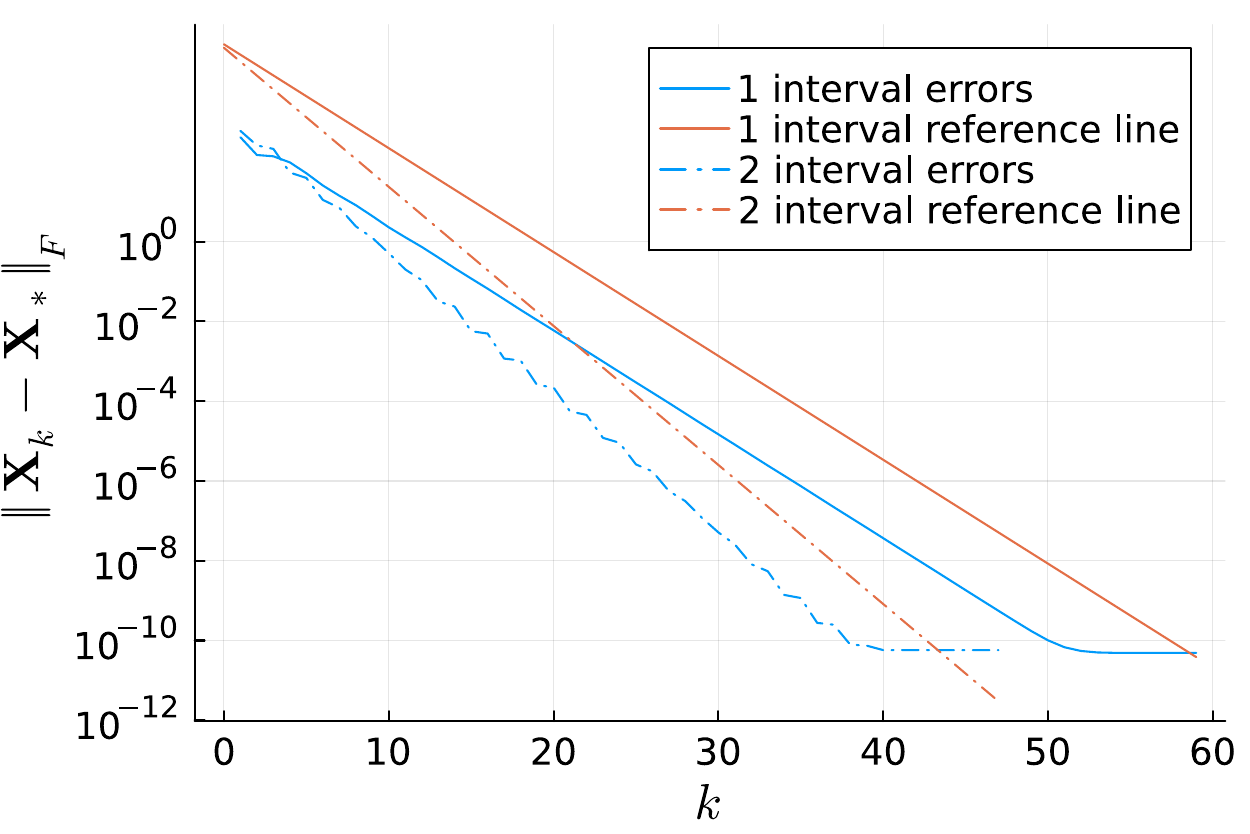}
	% \begin{subfigure}{0.495\linewidth}
	% 	\centering
	% 	\includegraphics[width=\linewidth]{errs1int.pdf}
	% \end{subfigure}
	% \begin{subfigure}{0.495\linewidth}
	% 	\centering
	% 	\includegraphics[width=\linewidth]{errs2int.pdf}
	% \end{subfigure}
\caption{Frobenius norm error for $\bX_k$ against the true solution $\bX_*$ to~\eqref{eq:Genmat} (computed via $\texttt{sylvester}$ in \Julia) at each iteration of Method 2 where $\bC$ is rank 2, $\bA,\bB\in\real^{1000\times1000}$, $\sigma(\bA)\subset[0.5,1]\cup\{10\}$, and $\sigma(\bB)\subset[-1.8,-0.5]$. For the solid lines, Algorithm~\ref{alg:sylv_inv_lr} is used with $\Sigma=[1,11.8]$, and for the dashed lines, Algorithm~\ref{alg:sylv_inv_lr} is modified to use Akhiezer polynomials with $\Sigma=[1,2.8]\cup[10.5,11.8]$. The reference lines denote the convergence rate of Corollary~\ref{cor:inv_conv_rate} with the leading constants replaced by the heuristic $D_{\bA,\bB,\bC}=20(m+n)$.}
\label{fig:mult_int_inv}
\end{figure}
% \begin{remark}
%     Algorithm~\ref{alg:sylv_inv} assumes access to a single interval $\Sigma\supset\sigma(\bA^T\otimes\bI-\bI\otimes\bB)$ and contains explicit formulae for the Akhiezer data in this case. One may generalize $\Sigma$ to multiple intervals by building $P_k$ with the Akhiezer polynomial recurrence coefficients and setting 
%     \[
%     S_k=2\pi\im\cC_\Sigma\left[p_kw\right](0),
%     \]
%     where $(p_k)_{k=0}^\infty$ and $w$ denote the Akhiezer polynomials and their weight function, respectively.
% \end{remark}
\subsection{Low-rank structure and compression}
As with Method 1, we handle low-rank problems of the form~\eqref{eq:lr_sylv_eqn} by writing the recurrence in block low-rank form and compressing matrices to their numerical rank as needed. The procedure for this is essentially the same as for the sign function iteration. In particular, letting $\bP_j=\bJ_j\bK_j$, the recurrence of Theorem~\ref{thm:kron_iter} becomes
\begin{equation*}
    \begin{aligned}
        &\bJ_0\bK_0=\bP_0=\bC,\\
        &\bJ_1\bK_1=\bP_1=\frac{1}{b_0}\begin{pmatrix}\bJ_0 &-\bB\bJ_0 &-a_0\bP_0\end{pmatrix}\begin{pmatrix}
            \bK_0\bA\\ \bK_0\\ \bK_0
        \end{pmatrix},\\
        &\bJ_j\bK_j=\bP_j=\frac{1}{b_{j-1}}\begin{pmatrix}
            \bJ_{j-1} &-\bB\bJ_{j-1} &-a_{j-1}\bJ_{j-1} &b_{j-2}\bJ_{j-2}
        \end{pmatrix}\begin{pmatrix}
            \bK_{j-1}\bA\\ \bK_{j-1}\\ \bK_{j-1}\\ \bK_{j-2}
        \end{pmatrix},\quad j\geq2.
    \end{aligned}
\end{equation*}
Letting $\bX_k=\bW_k\bZ_k$, the update step can be written as
\begin{equation*}
    \bW_{k+1}\bZ_{k+1}=\bX_{k+1}=\begin{pmatrix}
        \bW_k &\alpha_k\bJ_k
    \end{pmatrix}\begin{pmatrix}
        \bZ_k\\ \bK_k
    \end{pmatrix}.
\end{equation*}
With this in hand, Method 2 can be implemented in low-rank form as in Algorithm~\ref{alg:sylv_inv_lr}.
\begin{algorithm}
		\caption{Low-rank $\bX\bA-\bB\bX=\bU\bV$ solver via the Akhiezer iteration for $1/x$}\label{alg:sylv_inv_lr}
		\textbf{Input: }{Matrices $\bA,\bB,\bU,\bV$, parameters $\alpha,c$ such that $\sigma(\bA)-\sigma(\bB)\subset[\alpha-c,\alpha+c]$, and compression tolerance $\epsilon$.}\\
		Initialize $\mathbf W_0$ and $\mathbf Z_0$ to be empty.\\
        Set $\varrho^{-1}=-\frac{\alpha}{c}+\sqrt{\frac{\alpha}{c}-1}\sqrt{\frac{\alpha}{c}+1}.$\\
		\For{k=0,1,\ldots}{
			\uIf{k=0}{
				Set $\mathbf J_0=\bU$.\\
                Set $\mathbf K_0=\bV$.\\
                Set $S_0=\frac{1}{\sqrt{\alpha-c}\sqrt{\alpha+c}}$.
			}
			\uElseIf{k=1}{
				Set $\mathbf{J}_1=\begin{pmatrix}
				    \frac{1}{c}\bJ_0 &-\frac{1}{c}\bB\bJ_0 &-\frac{\alpha}{c}\bJ_0
				\end{pmatrix}$.\\
                Set $\bK_1=\begin{pmatrix}
                    \bK_0\bA\\ \bK_0\\ \bK_0
                \end{pmatrix}$.\\
                Set $S_1=S_0\varrho^{-1}$.
			}
   %          \uElseIf{k=2}{
			% 	Set $\mathbf{J}_2=\begin{pmatrix}
			% 	    \frac{2}{c}\bJ_1 &-\frac{2}{c}\bB\bJ_1 &-\frac{2\alpha}{c}\bJ_1 &-\sqrt{2}\bJ_0
			% 	\end{pmatrix}$.\\
   %              Set $\bK_2=\begin{pmatrix}
   %                  \bK_1\bA\\ \bK_1\\ \bK_1 \\ \bK_0
   %              \end{pmatrix}$.\\
   %              Set $S_2=S_1\varrho^{-1}$.
			% }
			\Else{
				Set $\mathbf{J}_k=\begin{pmatrix}
				    \frac{2}{c}\bJ_{k-1} &-\frac{2}{c}\bB\bJ_{k-1} &-\frac{2\alpha}{c}\bJ_{k-1} &-\bJ_{k-2}
				\end{pmatrix}$.\\
                Set $\bK_k=\begin{pmatrix}
                    \bK_{k-1}\bA\\ \bK_{k-1}\\ \bK_{k-1} \\ \bK_{k-2}
                \end{pmatrix}$.\\
                Set $S_k=S_{k-1}\varrho^{-1}$.
			}
			Set $\bJ_{k},\bK_{k}=\texttt{COMPRESS}\left(\bJ_{k},\bK_{k},\frac{\varrho^k\epsilon}{c}\right)$.\\
            \uIf{k=0}{
			Set $\mathbf W_{k+1}= \begin{pmatrix}
				\mathbf W_{k}& S_k\mathbf{J}_k
			\end{pmatrix}$.\\
            }
            \Else{
            Set $\mathbf W_{k+1}= \begin{pmatrix}
				\mathbf W_{k}& 2S_k\mathbf{J}_k
			\end{pmatrix}$.
            }
			Set $\mathbf Z_{k+1}=\begin{pmatrix}
				\mathbf Z_{k}\\ \mathbf{K}_k
			\end{pmatrix}$.\\
                Set $\bW_{k+1},\bZ_{k+1}=\texttt{COMPRESS}\left(\bW_{k+1},\bZ_{k+1},\epsilon\right)$.\\
			\If{\emph{converged}}{\Return{$\mathbf W_{k+1},\mathbf Z_{k+1}$}.}
            }
\end{algorithm}

We now motivate the choice of compression tolerance. As before, Lemma~\ref{lem:poly_asymp} guarantees that the coefficients $\alpha_j$ satisfy
\begin{equation}\label{eq:inv_coeff_bound}
\left|\alpha_j\right|\leq C\varrho^{-j},\quad \varrho=\ex^{\re\mathfrak g(0)},
\end{equation}
for some constant $C>0$ that cannot in general be computed. Heuristically, we again find that the weighted compression step behaves well when the tolerance is multiplied by $\varrho^k/c$ with $c=5$. In the single-interval case presented in Algorithm~\ref{alg:sylv_inv_lr}, the coefficients $S_k$ are obtained by iteratively dividing by $\varrho$, meaning that~\eqref{eq:inv_coeff_bound} holds for $C=2S_0$. Thus, one may instead weight the compression directly by $S_k$. Empirically, some precision is lost when this is done, so we continue to use the heuristic in this case.

To demonstrate our heuristics, we repeat the experiment of Figures~\ref{fig:unweight_rank}~and~\ref{fig:weight_rank}. In Figure~\ref{fig:weight_rank_inv}, we plot the numerical rank of $\bJ_k\bK_k$ and $\bW_k\bZ_k$ with and without weighting the compression step. We observe that the weighting bounds the numerical rank of $\bJ_k\bK_k$ while the error appears unchanged.
\begin{figure}
	\centering
	\begin{subfigure}{0.495\linewidth}
		\centering
		\includegraphics[width=\linewidth]{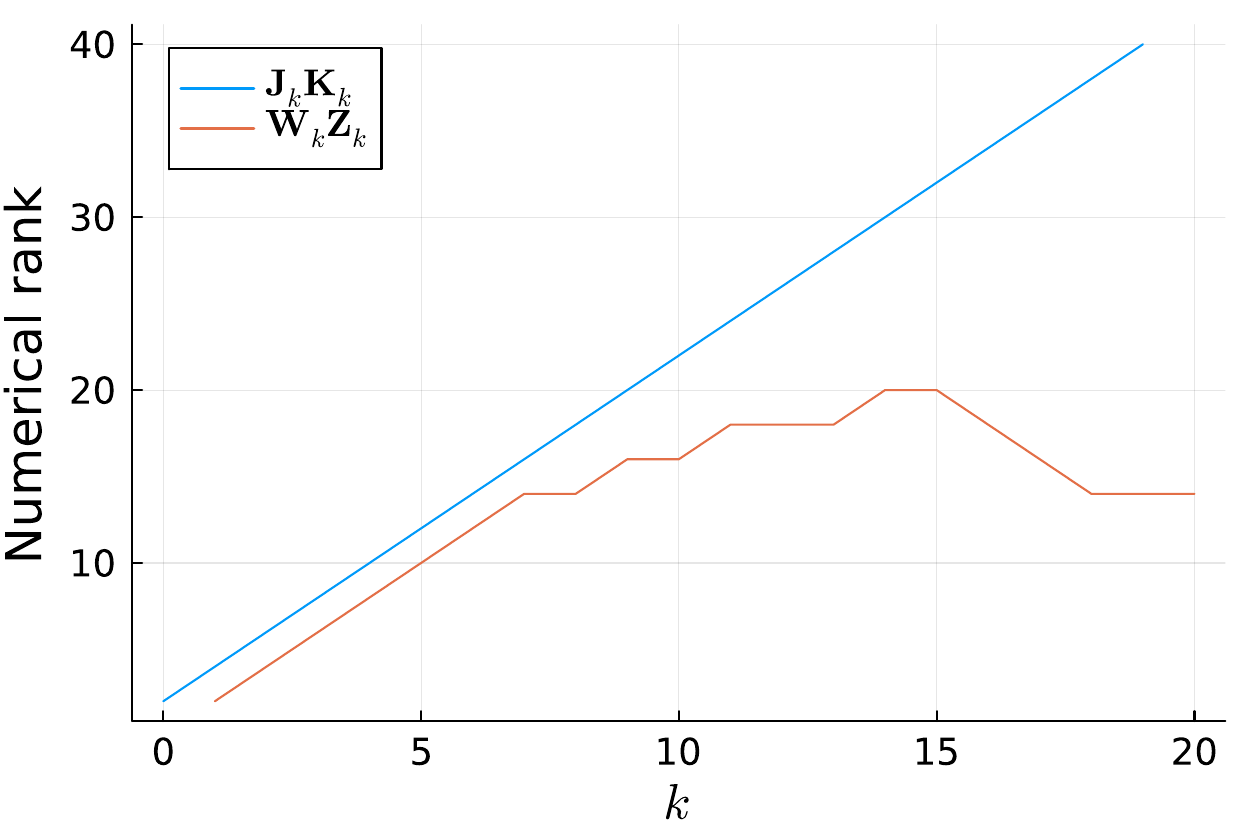}
	\end{subfigure}
	\begin{subfigure}{0.495\linewidth}
		\centering
		\includegraphics[width=\linewidth]{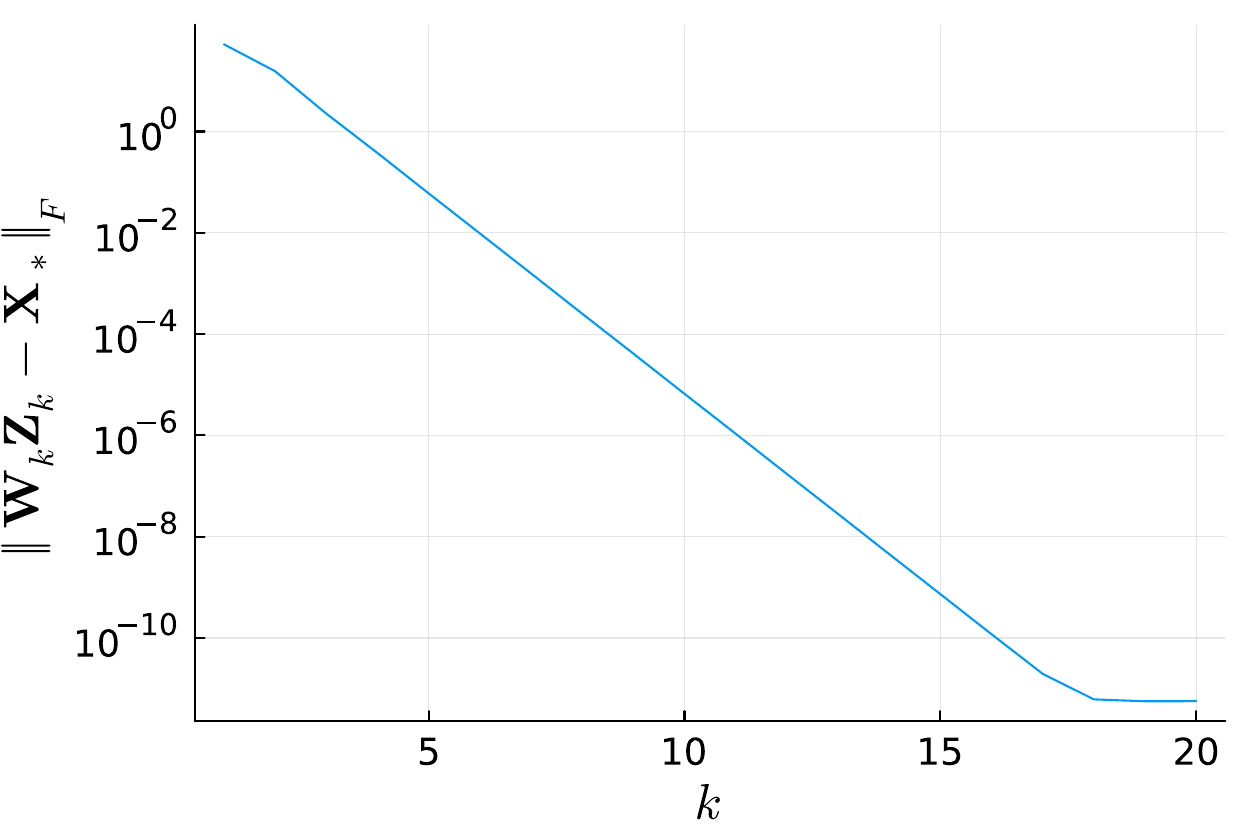}
	\end{subfigure}
    \begin{subfigure}{0.495\linewidth}
		\centering
		\includegraphics[width=\linewidth]{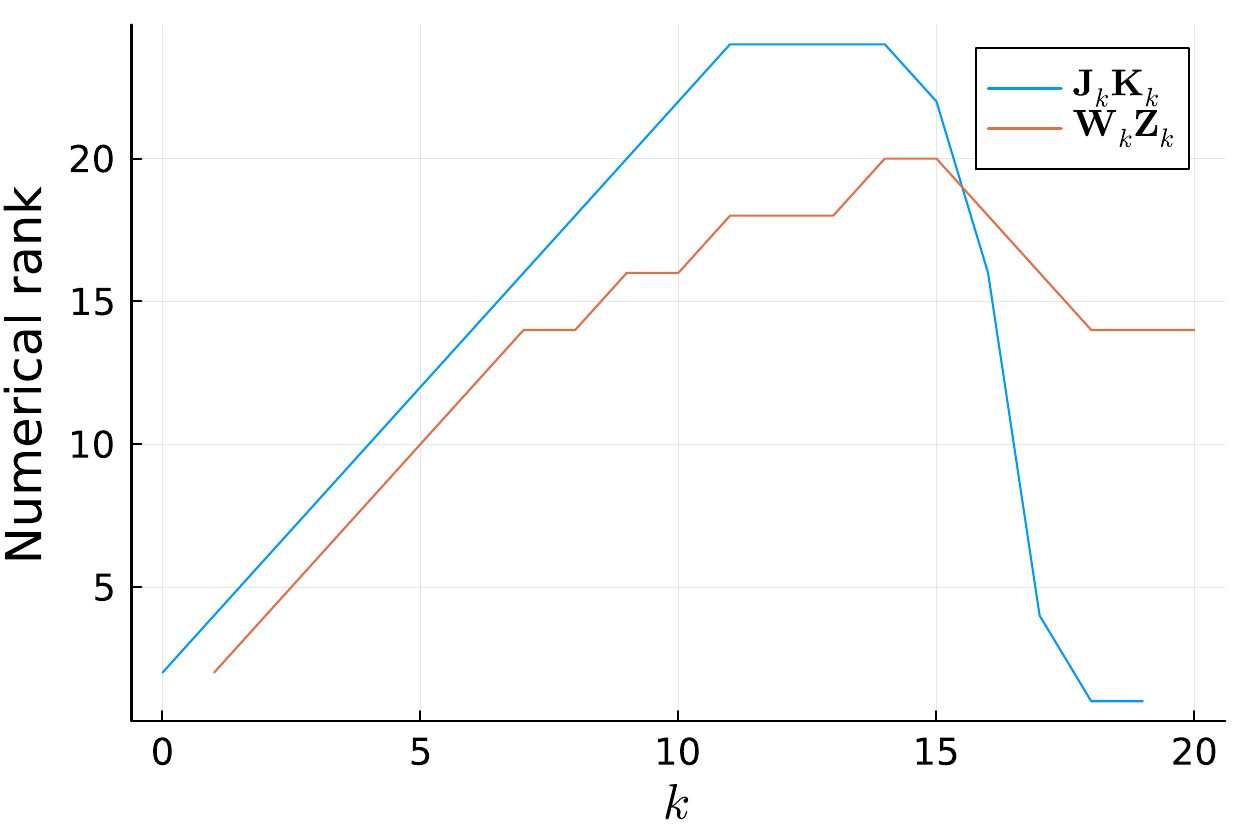}
	\end{subfigure}
	\begin{subfigure}{0.495\linewidth}
		\centering
		\includegraphics[width=\linewidth]{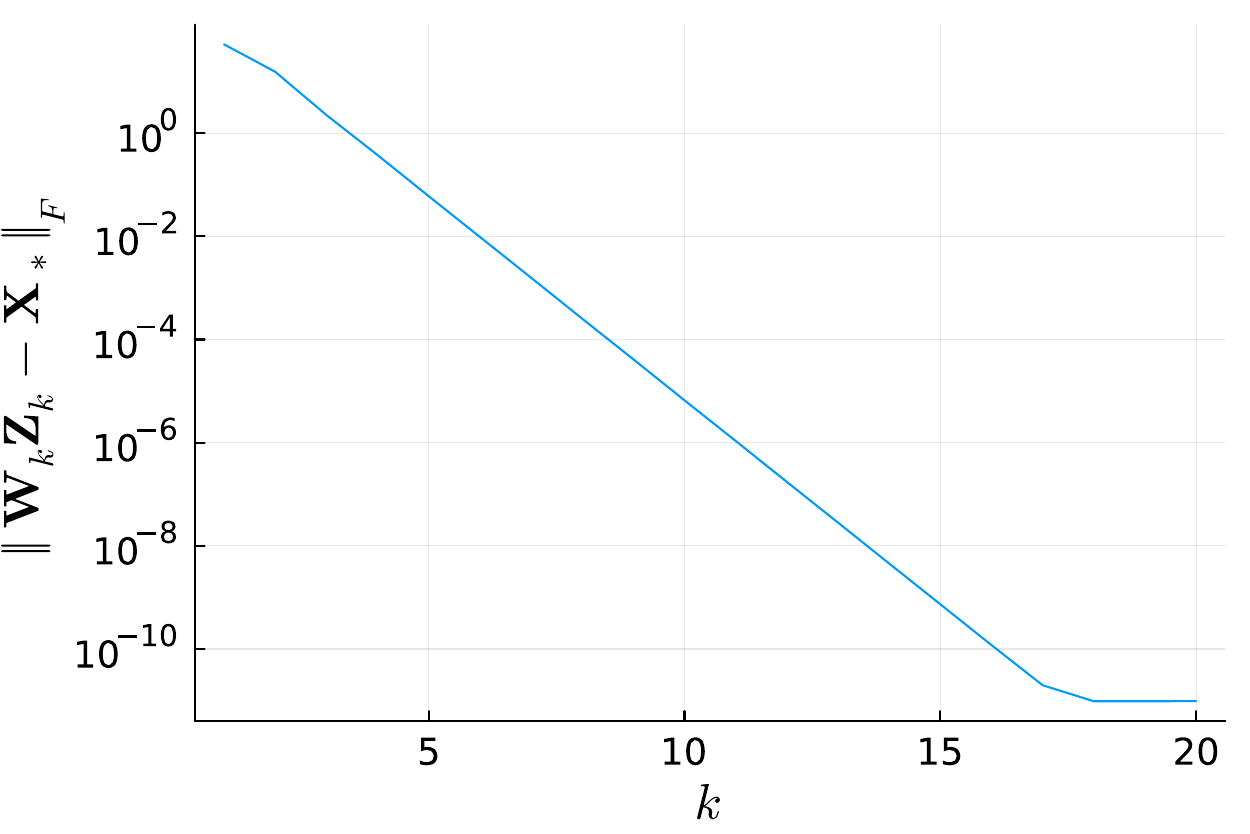}
	\end{subfigure}
\caption{Numerical rank and error of the iterates of Algorithm~\ref{alg:sylv_inv_lr} with unweighted compression (top) and weighted compression (bottom). The matrices $\bA,\bB,\bU,\bV$ are the same as in Figure~\ref{fig:unweight_rank}.}
\label{fig:weight_rank_inv}
\end{figure}
% In the single interval case of Algorithm~\ref{alg:sylv_inv_lr} however, this herustic is somewhat superfluous since the coefficients $S_k$ are obtained by iteratively dividing by $\varrho$. We instead weight the compression directly by $S_k$. Heuristically, some precision seems to be lost when the coefficients are use directly in this manner, hence the addition of the ``fudge factor'' of $d>0$. We find that $d=40$ tends to be a good choice.

\subsection{Convergence and stopping criteria}
We now present a convergence analysis of Method 2 that essentially follows from that in Section~\ref{sect:conv} and is again largely a simplified version of that of \cite[Section 4]{AkhIter}. The following theorem results from \cite[Theorem 4.5]{AkhIter} and describes the convergence of the Akhiezer iteration applied to solve a linear system.
\begin{theorem}
Let $\Sigma=\bigcup_{j=1}^{g+1}[\beta_j,\gamma_j]\subset\real$ be a union of disjoint intervals such that $0\notin\Sigma$, and let a generic matrix $\bM\in\compl^{\ell\times\ell}$ be diagonalized as $\bM=\bV\bLambda\bV^{-1}$ and satisfy $\sigma(\bM)\subset\Sigma$. Let $w$ be a weight function of the form~\eqref{eq:gen_weight} and denote the corresponding orthonormal polynomials by $\left(p_j(x)\right)_{j=0}^\infty$. Then, there exists a constant $d>0$, depending only on the eigenvalues of $\bM$, such that
\begin{equation*}
\left\|\sum_{j=0}^{k-1}2\pi\im\cC\left[p_jw\right](0)p_j(\bM)\bb-\bM^{-1}\bb\right\|_2\leq d\|\bV\|_2\left\|\bV^{-1}\right\|_2\|\bb\|_2\frac{\varrho^{-k}}{1-\varrho^{-1}},\quad \varrho=\ex^{\re\mathfrak g(0)}.
\end{equation*}
\end{theorem}
% \begin{proof}
% We have that
% \begin{equation*}
% \begin{aligned}
% &\left\|
% \sum_{j=0}^{k-1}2\pi\im\cC\left[p_jw\right](0)p_j(\bM)\bb-\bM^{-1}\bb\right\|_2=\left\|\sum_{j=k}^{\infty}2\pi\im\mathcal{C}_\Sigma\left[p_jw\right](0)\bV p_j(\bLambda)\bV^{-1}\bb\right\|_2
% %\leq\|\bV\|_2\sum_{j=k}^{\infty}\|\bC_j(z)\bV^{-1}\bb\|_2
% \\ &\leq
% \|\bV\|_2\left\|\bV^{-1}\right\|_2\|\bb\|_2\sum_{j=k}^{\infty}\max_{\lambda\in\sigma(\bM)}\left|2\pi\im\mathcal{C}_\Sigma\left[p_jw\right](0)p_j(\lambda)\right|.
% \end{aligned}
% \end{equation*}
% Since $\lambda$ lies in the interior of $\Sigma$ for all $\lambda\in\sigma(\bM)$, by Lemma \ref{lem:poly_asymp}, there exists some constant $d>0$ such that
% \begin{equation*}
% \left|2\pi\im\mathcal{C}_\Sigma\left[p_jw\right](0)p_j(\lambda)\right|=\left|\delta_j(0)\right|\left|\hat\delta_j(\lambda)\right|\ex^{j\re\left(\mathfrak g(\lambda)-\mathfrak g(0)\right)}\leq d\varrho^{-j},
% \end{equation*}
% for all $j$ and $\lambda\in\sigma(\bM)$. The theorem follows from~\eqref{eq:geo_series}.
% \end{proof}
The following corollary describes the convergence rate of Method 2.
\begin{corollary}\label{cor:inv_conv_rate}
Under the assumptions of Theorem~\ref{thm:kron_iter} on $\bM=\scrS_{\bA,\bB}$, denote the solution of~\eqref{eq:Genmat} by $\bX$, and let $\bA$ and $\bB$ be diagonalizable and $\sigma(\bA)-\sigma(\bB)$ be contained in the interior of $\Sigma$. Then, there exists $d>0$, depending only on the eigenvalues of $\bM$, such that the iterate $\bX_k=\sum_{j=0}^{k-1}\alpha_j\bP_j$ satisfies
\begin{equation*}
\left\|\bX_k-\bX\right\|_F\leq d\|\bV_\bA\|_2\left\|\bV_\bA^{-1}\right\|_2|\bV_\bB\|_2\left\|\bV_\bB^{-1}\right\|_2\|\bC\|_F\frac{\varrho^{-k}}{1-\varrho^{-1}},\quad \varrho=\ex^{\re\mathfrak g(0)},
\end{equation*}
where $\bA$ and $\bB$ are diagonalized as $\bA=\bV_\bA\bLambda_\bA\bV_\bA^{-1}$ and $\bB=\bV_\bB\bLambda_\bB\bV_\bB^{-1}$, respectively.
\end{corollary}
\begin{proof}
The eigenvector matrix of $\bA^T\otimes\bI_m-\bI_n\otimes\bB$ is given by $\bV_\bA^T\otimes\bV_\bB$. The result follows from the identity $\|\bM\otimes\bN\|_2=\|\bM\|_2\|\bN\|_2$.
\end{proof}
Thus, Method 2 converges at a geometric rate $\varrho^{-1}=\ex^{-\re\mathfrak g(0)}$. In the single-interval case $\Sigma=[\beta,\gamma]$,
\begin{equation*}
\varrho^{-1}=-\frac{\gamma+\beta}{\gamma-\beta}+\sqrt{\frac{\gamma+\beta}{\gamma-\beta}-1}\sqrt{\frac{\gamma+\beta}{\gamma-\beta}+1}.
\end{equation*}
Denote $D_{\bA,\bB,\bC}=d\|\bV_\bA\|_2\left\|\bV_\bA^{-1}\right\|_2|\bV_\bB\|_2\left\|\bV_\bB^{-1}\right\|_2\|\bC\|_F$. Then, for a given tolerance $\epsilon> 0$, we require
\begin{equation*}
k = \left\lceil -\log_\varrho\frac{\epsilon(1-\varrho^{-1})}{D_{\bA,\bB,\bC}} \right\rceil,
\end{equation*}
iterations to ensure that $\left\|\bX_k-\bX\right\|_F<\epsilon$. For generic $\bA\in\compl^{n\times n}$ and $\bB\in\compl^{m\times m}$ that are not highly nonnormal, we hypothesize an upper bound $D_{\bA,\bB,\bC}\leq20(m+n)$. In finite precision, for given a tolerance $\epsilon$, we use
\begin{equation*}
    %k>-\log_\varrho\left(\max\left\{\frac{\epsilon(1-\varrho^{-1})}{10\ell},\frac{\varepsilon_{\mathrm{mach}}}{5}\right\}\right),
    k = \left\lceil \min \left\{ - \log_\varrho\left(\frac{\epsilon(1-\varrho^{-1})}{20(m+n)}\right), - \log_\varrho\left(\frac{\varepsilon_{\mathrm{mach}}}{5}\right)    \right\} \right\rceil,
\end{equation*}
iterations as our stopping criterion.

We demonstrate these heuristics by repeating the experiment of Figure~\ref{fig:err_heur}. In Figure~\ref{fig:err_heur_inv}, the errors $\|\bX_k-\bX_*\|_F$ against the bound of Corollary~\ref{cor:inv_conv_rate} with $D_{\bA,\bB,\bC}$ replaced by $20(m+n)$ and observe that it does indeed bound the error.
\begin{figure}
	\centering
	\begin{subfigure}{0.495\linewidth}
		\centering
		\includegraphics[width=\linewidth]{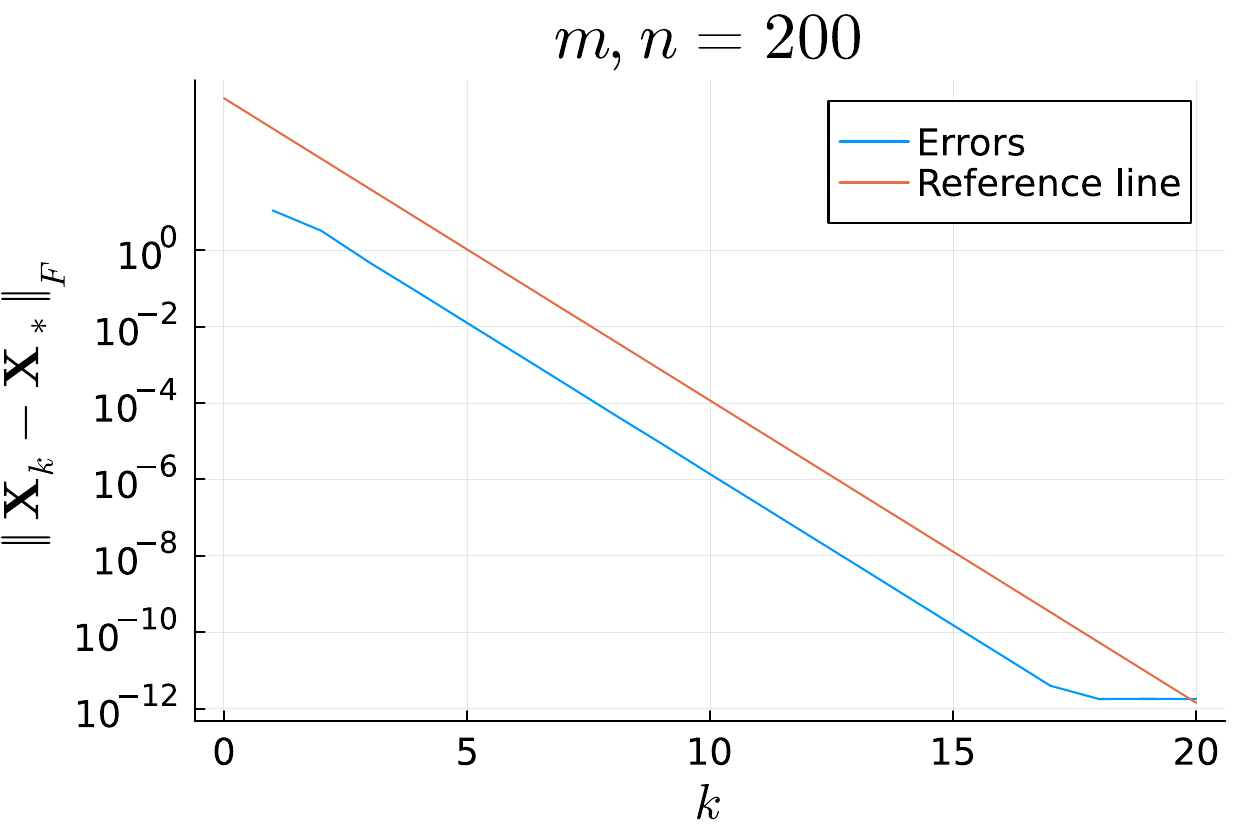}
	\end{subfigure}
	\begin{subfigure}{0.495\linewidth}
		\centering
		\includegraphics[width=\linewidth]{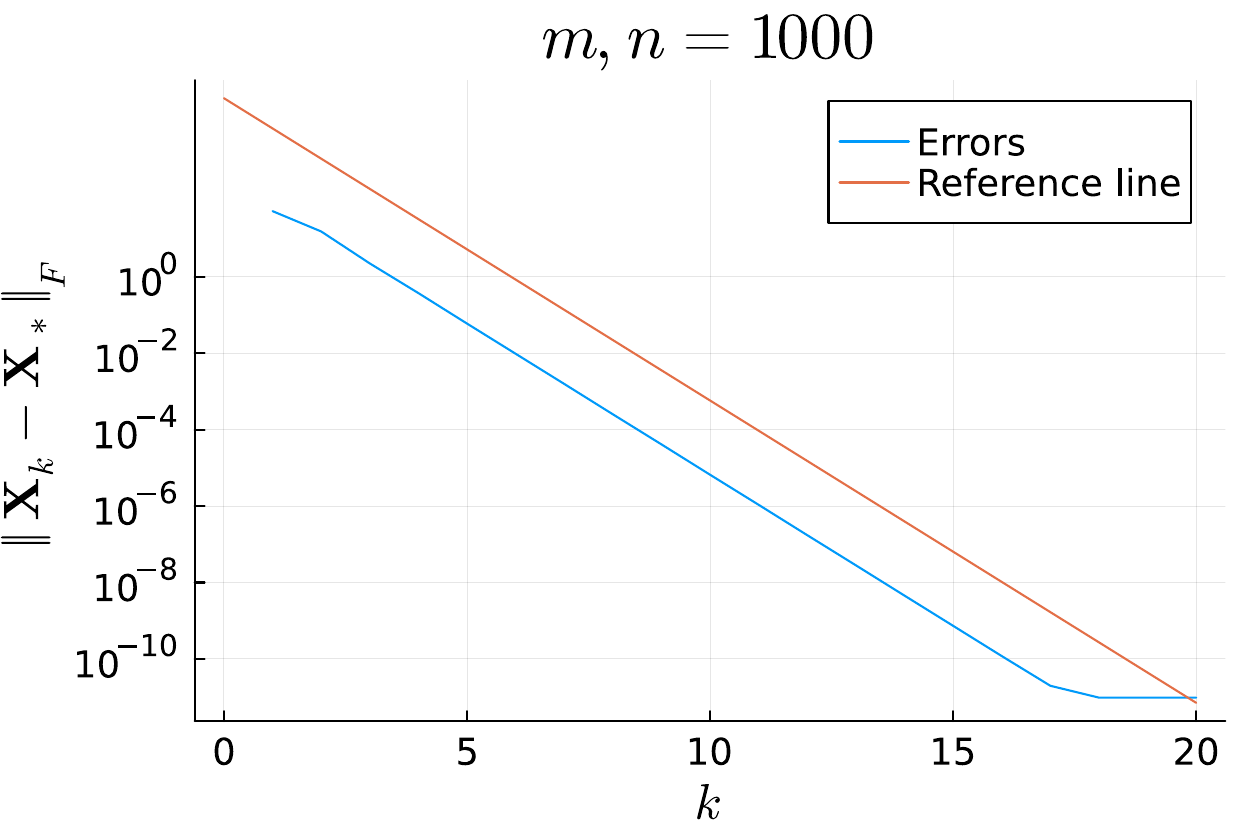}
	\end{subfigure}
\caption{Frobenius norm error for $\bX_k$ against the true solution $\bX_*$ to~\eqref{eq:Genmat} at each iteration where $\bA,\bB,\bC$ are the same as in Figure~\ref{fig:err_heur}. Here, the reference line denotes the convergence rate of Corollary~\ref{cor:inv_conv_rate} with the leading constants replaced by the heuristic $D_{\bA,\bB,\bC}=20(m+n)$.}
\label{fig:err_heur_inv}
\end{figure}

\subsection{Complexity and storage}\label{sect:comp_meth_2}
We analyze the complexity of Algorithm~\ref{alg:sylv_low_rank} under the same assumptions as with Method 1. That is, we assume that the numerical rank of $c\varrho^{-k}\bJ_k\bK_k$ is given by $R(k)$ and that $R(k)=0$ for all $k\geq k^*$ and some $k^*\in\mathbb{N}$. We denote the numerical rank of $\bW_k\bZ_k$ by $\hat R(k)$, the cost of a left matrix-vector multiplication by $\bA$ by $T_\bA$, and the cost of a right matrix-vector multiplication by $\bB$ by $T_\bB$.

At the $k$th iteration of Algorithm~\ref{alg:sylv_inv_lr}, the construction of $\bJ_k$ and $\bK_k$ requires at most 
\[
3mR(k-1)+(T_\bA+T_\bB)R(k-1),
\]
arithmetic operations, and the construction of $\bW_{k+1}$ and $\bZ_{k+1}$ requires at most 
\[
mR(k),
\]
arithmetic operations. The compression of $\bJ_k\bK_k$ $\texttt{COMPRESS}\left(\bJ_{k},\bK_{k},\frac{\varrho^k\epsilon}{c}\right)$ requires 
\begin{equation*}
\OO\left((3R(k-1)+R(k-2))^3+(3R(k-1)+R(k-2))^2(m+n)\right),
\end{equation*}
arithmetic operations, while the compression of $\bW_k\bZ_k$ $\texttt{COMPRESS}\left(\bW_{k},\bZ_{k},\epsilon\right)$ requires 
\begin{equation*}
\OO\left((R(k)+\hat R(k))^3+(R(k)+\hat R(k))^2(m+n)\right),
\end{equation*}
arithmetic operations. Finally, the cost of computing the coefficient $S_k$ is $\OO(1)$. For multiple intervals, the cost of computing the recurrence coefficients and $p_j$-series coefficients is again assumed to be $\OO(1)$, as computing the Akhiezer data requires $\OO(k)$ operations for $k$ iterations.

As before, we bound the complexity of Algorithm~\ref{alg:sylv_inv_lr} run for $k$ iterations by denoting 
\begin{equation*}
\cR:=\max_{j\in\mathbb{N}}\left\{R(j),\hat R(j)\right\}.
\end{equation*}
Then, to run Algorithm~\ref{alg:sylv_inv_lr} for $k$ iterations, we need
\begin{equation*}
\OO\left(k\cR(T_\bA+T_\bB)+k(\cR+r)^3+k(\cR+r)^2(m+n)\right),
\end{equation*}
arithmetic operations. Due to the extra $k\cR T_\bB$ term, this will typically be more expensive than running Method 1 (Algorithm~\ref{alg:sylv_low_rank}) for the same number of iterations; however, as we demonstrate in Figure~{\ref{fig:conv_rates}}, fewer iterations are typically needed to reach a given tolerance. As with Method 1, when $T_\bA=\OO(n^2)$, $T_\bB=\OO(m^2)$, $r,\cR\ll m,n$, and our stopping criterion is applied, Algorithm~\ref{alg:sylv_inv_lr} requires $\OO\left(m^2\log m+n^2\log n\right)$ arithmetic operations. 

Under these same assumptions, a maximum of $10\cR(m+n)$ stored entries are required at any given iteration. Therefore, in the general case $r,\cR\ll m,n$, Algorithm~\ref{alg:sylv_low_rank} requires at most $\OO(m+n)$ stored entries.

\section{Examples and applications}\label{sect:applications}
\subsection{Comparison of methods}
Here, we test the performance of Methods 1 and 2 against various competitors. In Figure~\ref{fig:timings}, we include a timing comparison\footnote{The timing tests in this subsection were done on a Lenovo laptop running Linux Mint 22.1 Cinnamon with 16 cores and 62.2 GB of RAM with an Intel\textregistered{} Core\texttrademark{} Ultra 9 185H processor. All other computations in this paper were performed using a single thread on a Lenovo laptop running Ubuntu version 20.04 with 8 cores and 16 GB of RAM with an Intel\textregistered{} Core\texttrademark{} i7-11800H processor running at 2.30 GHz.} of these methods (Algorithms~\ref{alg:sylv_low_rank}~and~\ref{alg:sylv_inv_lr}) with a standard block polynomial Krylov method~\cite{simoncini2016computational}, factored ADI \cite{BENNER20091035}, and {\tt sylvester}, the built-in \Julia\ implementation of the Bartels--Stewart algorithm \cite{BartelsStewart} for problems of various size\footnote{For all problems in this subsection, $\bA$ and $\bB$ are generated by randomly generating eigenvalues in the designated interval and conjugating by a random orthogonal matrix, and $\bC=\bU\bV$ corresponds to random $\bU,\bV$.}. In the case of the block polynomial Krylov method, we implement the method as described in~\cite[Algorithm 5]{simoncini2016computational} with the method in~\cite[Equation~2.4]{KressnerLMP21} applied to cheaply monitor the residual. We do not implement any restarting strategies.
% To demonstrate the advantageous complexity of Method 1 relative to direct solvers and methods requiring matrix inversion, we include a timing comparison\footnote{All computations in this paper are performed on a Lenovo laptop running Ubuntu version 20.04 with 8 cores and 16 GB of RAM with an Intel\textregistered{} Core\texttrademark{} i7-11800H processor running at 2.30 GHz.} of Algorithm~\ref{alg:sylv_low_rank} with factored ADI \cite{BENNER20091035} and {\tt sylvester}, the built-in \Julia\ implementation of the Bartels--Stewart algorithm \cite{BartelsStewart}, in Figure~\ref{fig:timings}.

\begin{figure}
	\centering
	\begin{subfigure}{0.495\linewidth}
		\centering
		\includegraphics[width=\linewidth,trim=0cm 0cm 0cm 1.5cm, clip=true]{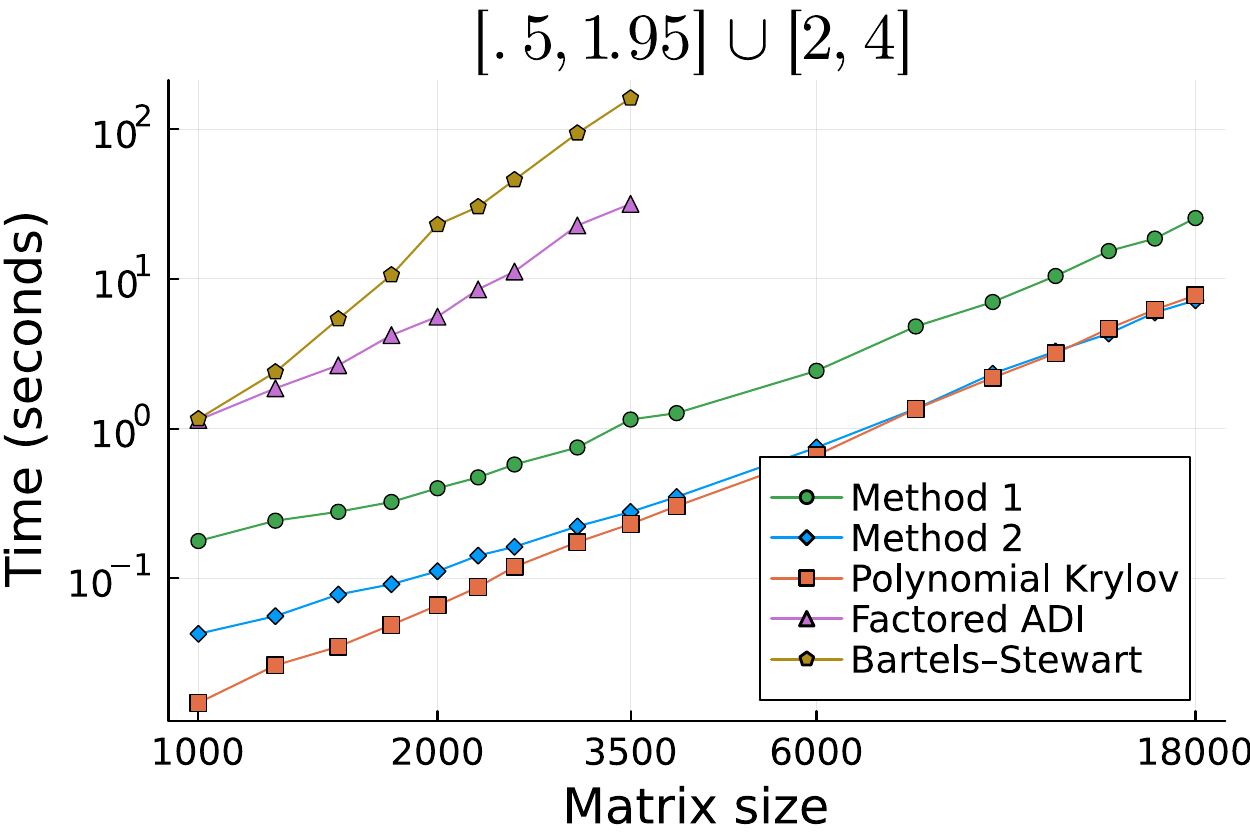}
	\end{subfigure}
	\begin{subfigure}{0.495\linewidth}
		\centering
		\includegraphics[width=\linewidth, trim=0cm 0cm 0cm 1.3cm, clip=true]{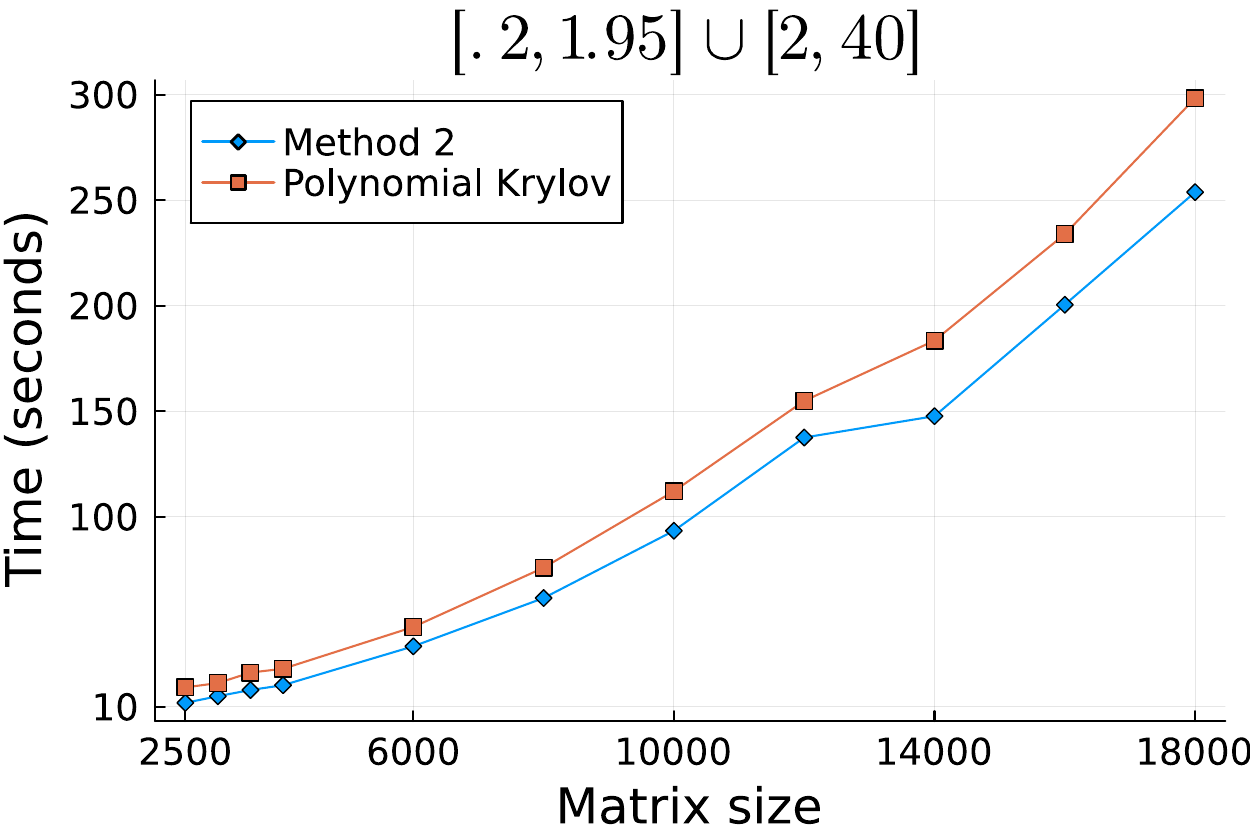}
	\end{subfigure}
\caption{\textbf{Left:} Timing comparisons of Methods 1 and 2 against the factored ADI, Bartels--Stewart, and polynomial Krylov methods for solving~\eqref{eq:Genmat}. Both factored ADI and Methods 1 and 2 are run with a tolerance of $\epsilon=2\times10^{-8}$ (left) while the polynomial Krylov method is run until the residual is less than $2\times10^{-10}$ (left), which requires between $16$ and $30$ iterations. Here, $\bA$ and $\bB$ are both dense square matrices of varying size, with $\sigma(\bA)\subset[0.5,1.95]$, $\sigma(\bB)\subset[-4,-2]$, and $\bC$ is rank 1. Since the timing comparisons between Method 2 and the polynomial Krylov method are hard to parse in this figure, we include a table of timings for the three inverse-free methods with $n \geq 4000$ in Table \ref{tab:timing-table}. \textbf{Right:} Timings for Method 2 and the polynomial Krylov method, which are both substantially faster than the other considered methods, for a more challenging problem. Here, tolerance parameters are kept the same, but $\bC$ is a rank $10$ matrix, $\sigma(\bA)\subset[.5,1.95]$, and $\sigma(\bB)\subset[-40, -2]$. Convergence is slower for both methods (between $62$ and $80$ iterations per run for both methods), and the cost per iteration is greater due to the higher-rank right-hand side.}
\label{fig:timings}
\end{figure}

\begin{center}
\begin{table}[h!]
\begin{tabular}{|c|c|c|c|}
\hline
\textbf{Matrix size ($n \times n$)} & \textbf{Method 1} & \textbf{Method 2} & \textbf{Polynomial Krylov} \\
\hline
%Row 1 & .177036 & .042538 & .01467 \\
%\hline
%Row 2 & .242143 & .0557403 & .0261367 \\
%\hline
%Row 3 & .277608 & .0777326 & .0348405 \\
%\hline
%Row 4 & .32303 & .0911083 & .0487528 \\
%\hline
%Row 5 & .398957 & .111218 & .0657314 \\
%\hline
%Row 6 & .471708 & .141429 & .0871951 \\
%\hline
%Row 7 & .574973 & .162045& .119125 \\
%\hline
%Row 8 & .747359 & .221822 & .173993 \\
%\hline
%Row 9 & 1.1493 & .277165 & .230118 \\
%\hline
$n = 4000$ & $1.26813$ & $0.349146$ & $0.302879$ \\
\hline
$6000$ & $2.43285$ & $0.74685$ & $0.661923$ \\
\hline
$8000$ & $4.81385$ & $1.35561$ & $1.34982$ \\
\hline
$10000$ & $7.01396$ & $2.32737$ & $2.18676$ \\
\hline
$12000$ & $10.4664$ & $3.27932$ & $3.18396$ \\
\hline
$14000$ & $15.3646$ & $4.3167$ & $4.65676$ \\
\hline
$16000$ & $18.6218$ & $5.96538$ & $6.22975$ \\
\hline
$18000$ & $25.5412$ & $7.20563$ & $7.77245$ \\
\hline
\end{tabular}
\caption{Runtimes (in seconds) for the inverse-free methods as shown in Figure \ref{fig:timings} (left). Despite requiring more iterations, Method 2 eventually matches and slightly outperforms the polynomial Krylov method. Importantly, it requires far less memory to run. Unlike the polynomial Krylov method, Methods 1 and 2 have memory costs that remain bounded as the number of iterations grows.}
\label{tab:timing-table}
\end{table}
\end{center}

It is unsurprising that even when $\sigma(\bA)$ and $\sigma(\bB)$ are close, all of the inverse-free methods eventually outperform factored ADI and Bartels--Stewart as the size of $\bX$ increases due to the dense matrices involved. We note that while the required number of iterations of the inverse-free methods grows as $\sigma(\bA)$ and $\sigma(\bB)$, if $\Sigma$ remains constant as the matrix size increases, there will always be a threshold past which inverse-free methods are advantageous.

As shown in Figure \ref{fig:timings} and Table \ref{tab:timing-table}, for large problems where convergence is relatively rapid (e.g., $< 50$ iterations), Method 2 matches or slightly outperforms the polynomial Krylov method. For problems involving higher-rank right-hand sides and eigenvalue distributions that slow the convergence of iterative methods, Method 2 clearly performs better than all other methods. However, the major advantage of both Methods~1~and~2 in either setting is that they requires less memory relative to the polynomial Krylov method, with no specialized restarting strategies needed to achieve these savings. Moreover, there is no need to monitor the residual since we can precompute the number of iterations required using our heuristics and bounds on the solution error norm. We also observe that while Method 1 appears to require the least memory (see Figure~\ref{fig:conv_rates}), it requires a substantially higher runtime than the other inverse-free methods. 

In Figure~\ref{fig:timings-new}, we compare timings for the two fastest inverse-free methods on a fixed problem size ($1500 \times 1500$) with varying eigenvalue distributions. We observe that when one of the intervals containing eigenvalues of $\bA$ and $\bB$ is much larger than the other, Method 2 (Algorithm~\ref{alg:sylv_inv_lr}) outperforms the polynomial Krylov method. In this setting, the number of iterations required by the polynomial Krylov method is only slightly less than the iterations required for Method 2 (e.g, $161$ vs. $169$ when $\beta = 0.1$). Since the polynomial Krylov method has a cost that grows as the number of iterations grows, it can be slower than Method 2 when many iterations are required. If instead the two intervals are symmetric and small, the polynomial Krylov method converges more rapidly than Method 2 (e.g., $36$ vs. $49$ iterations when $\beta = 0.1$) and is also more efficient in terms of runtime.
\begin{figure}
	\centering
	\begin{subfigure}{0.495\linewidth}
		\centering
		\includegraphics[width=\linewidth]{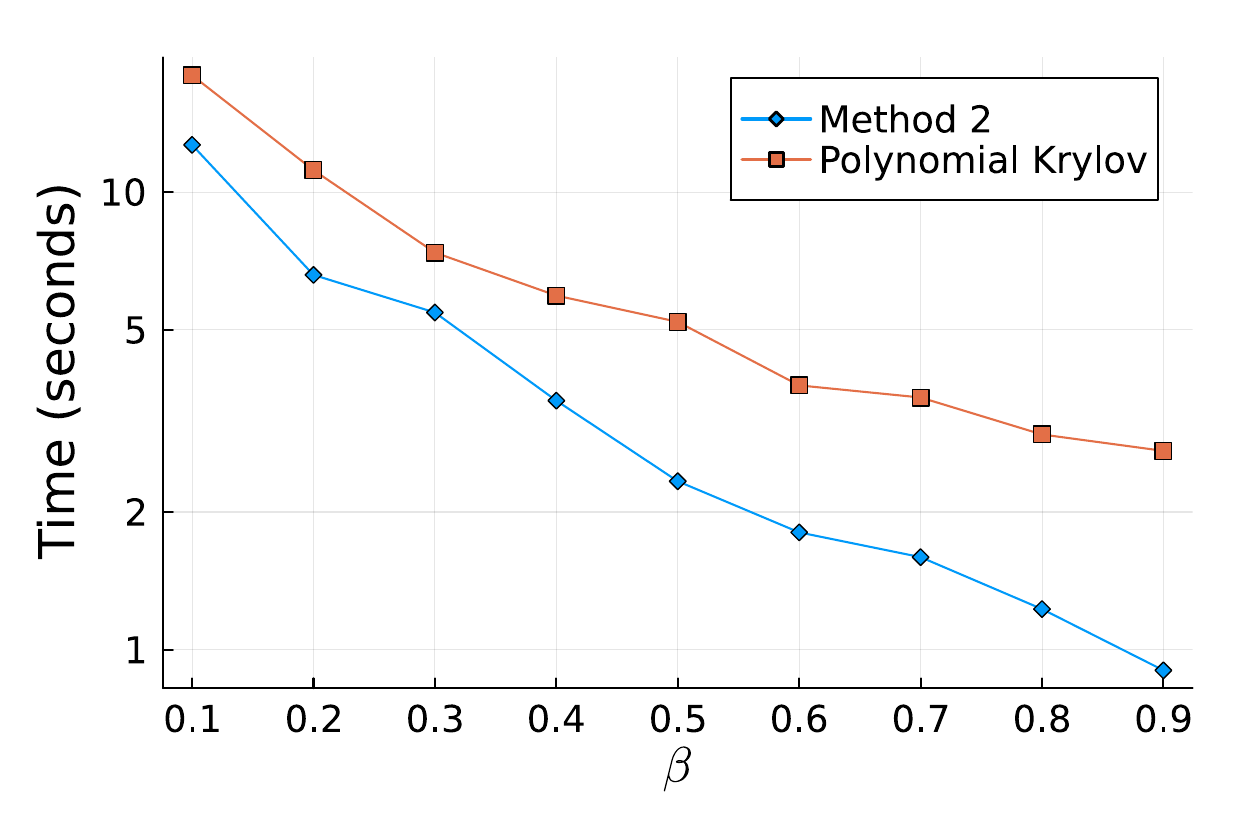}
	\end{subfigure}
	\begin{subfigure}{0.495\linewidth}
		\centering
		\includegraphics[width=\linewidth]{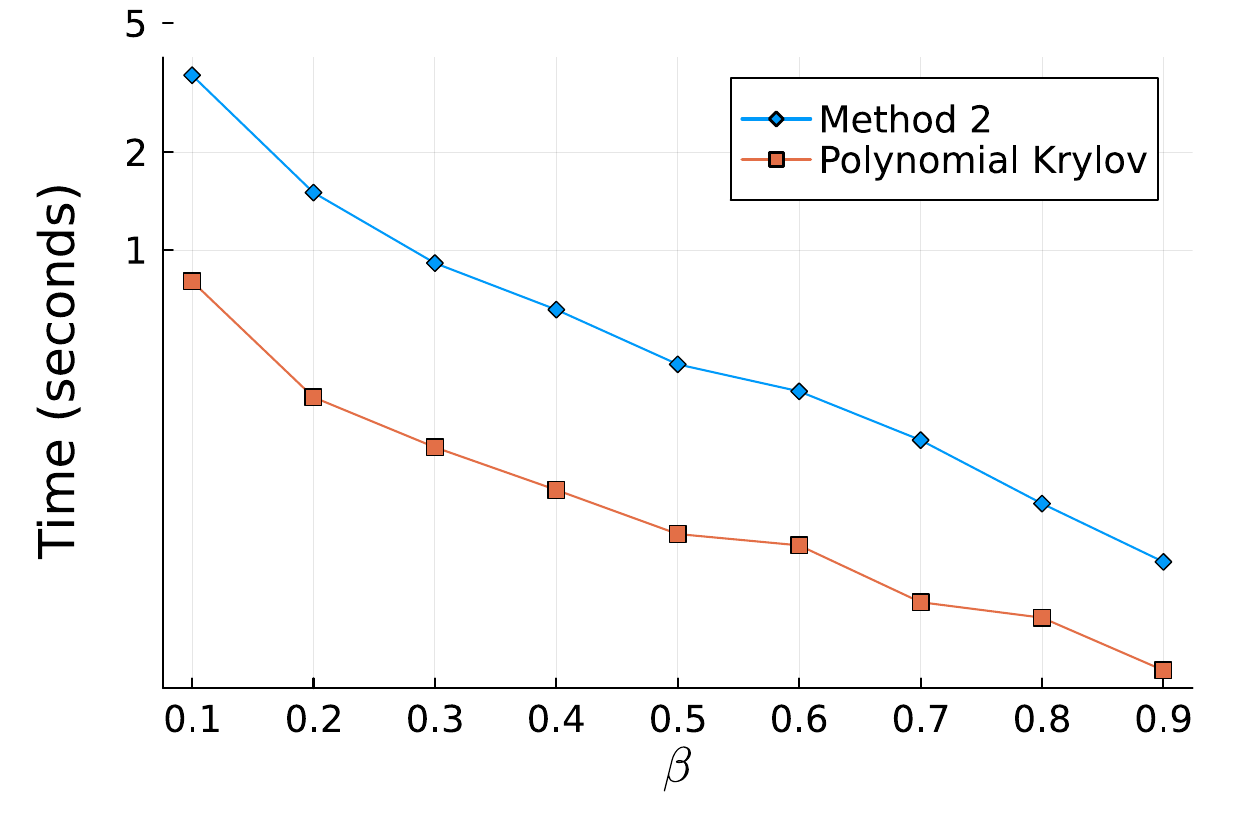}
	\end{subfigure}
\caption{Timing comparisons of Method 2 against a block polynomial Krylov method \cite[Algorithm 5]{simoncini2016computational} for solving~\eqref{eq:Genmat} where $\bA,\bB\in\real^{1500\times1500}$ and $\bC$ is rank $5$. Here, Method 2 (Algorithm~\ref{alg:sylv_inv_lr}) is run with a tolerance of $\epsilon=2\times10^{-9}$, while the polynomial Krylov method is run until the residual is less than $2\times10^{-10}$. The parameter $\beta$ corresponds to the distribution of the eigenvalues of $\bA,\bB$, with $\sigma(\bA)\subset[\beta,1]$, and $\sigma(\bB) \subset[-20,-\beta]$ (left) and $\sigma( \bB) \subset [-1, -\beta]$ (right). The latter case enables much faster convergence for both methods.}
\label{fig:timings-new}
\end{figure}

While there are some regimes where Method 2 is faster than a basic polynomial Krylov method, these experiments can be summarized as follows: Method 2 can essentially match the performance of a basic polynomial Krylov method while requiring far less memory (see Figure~\ref{fig:conv_rates}) and no special restarting strategies. As we show in Section~\ref{sect:Poisson} below, both Methods 1 and 2 can be run for thousands of iterations without issue. They remain stable, even on problems where the convergence of all polynomial-based methods is painfully slow. This is not typically possible for the polynomial Krylov method, especially if $\mathrm{rank}(\bC)$ is not very small, since the matrix factors involved in computing the solution can rapidly grow in their number of columns while also losing critical orthogonality properties. 
\subsection{Integral equations}
Given H\"older continuous, symmetric kernels $K_1,K_2:[-1,1]^2\to\real$ and continuous functions $f_1,\ldots,f_r,g_1,\ldots,g_r:[-1,1]\to\real$, we consider the integral equation
\begin{equation}\label{eq:int_eqn}
2u(x,y)+\int_{-1}^1K_1(x,x')u(x',y)\df x'+\int_{-1}^1K_2(y,y')u(x,y')\df y'=\sum_{\ell=1}^rf_\ell(x)g_\ell(y),
\end{equation}
for the unknown $u:[-1,1]^2\to\real$. We assume that the $L^2\left([-1,1]\right)$ eigenvalues of the operators defined by $K_1, K_2$ lie in the interval $(-1 + \delta,\infty)$ for some $\delta > 0$.   Denote the Gauss--Legendre quadrature nodes and weights for $n$ sample points by $\{x_j\}_{j=1}^n$ and $\{w_j\}_{j=1}^n$, respectively. Applying this quadrature rule to each integral in \eqref{eq:int_eqn} and collocating at the points $\{x_j\}_{j=1}^n$ in both $x$ and $y$ yields the Sylvester equation
\begin{equation}\label{eq:int_eqn_mat}
(\bI+\bK_1)\bU+\bU(\bI+\bK_2)=\sum_{\ell=1}^r\bff_\ell\bg_\ell^T,
\end{equation}
where $\bK_1,\bK_2,\bU\in\real^{n\times n}$ and $\bff_\ell,\bg_\ell\in\real^n$ have entries
\begin{equation}\label{eq:int_sys_def}
\begin{aligned}
&\left(\bK_{1,2}\right)_{j,k}=\sqrt{w_jw_k}K_{1,2}(x_j,x_k),\quad 
\bU_{j,k}\approx \sqrt{w_jw_k}u(x_j,x_k),\\
&\left(\bff_\ell\right)_j=\sqrt{w_j}f_\ell(x_j),\quad \left(\bg_\ell\right)_j=\sqrt{w_j}g_\ell(x_j),\quad \ell=1,\ldots,r.
\end{aligned}
\end{equation}
The eigenvalues of $\bI+\bK_{1,2}$ are larger than $\delta$ and less than\footnote{An upper bound can be more cheaply computed by instead using the Frobenius norm.} $1+\|\bK_{1,2}\|_2$, which remains bounded as $n$ increases. Because $\sum_{\ell=1}^r\bff_\ell\bg_\ell^T$ has rank at most $r$, Method 2 can be applied as in Algorithm~\ref{alg:sylv_inv_lr}. The efficiency of this method allows for such integral equations to be solved on fine grids. In Figure~\ref{fig:int_eqn}, we plot both an approximate solution to \eqref{eq:int_eqn} with $K_1(x,y)=K_2(x,y)=\exp(-2|x-y|)$, $r=1$, $f_1(x)=\frac{\cos(4x)}{1.04-x^2}$, and $g_1(x)=\sin(20x)$ at $n=2000$ gridpoints and the error in \eqref{eq:int_eqn_mat} at each iteration of Method 2.
\begin{figure}
	\centering
	\begin{subfigure}{0.495\linewidth}
		\centering
		\includegraphics[width=\linewidth]{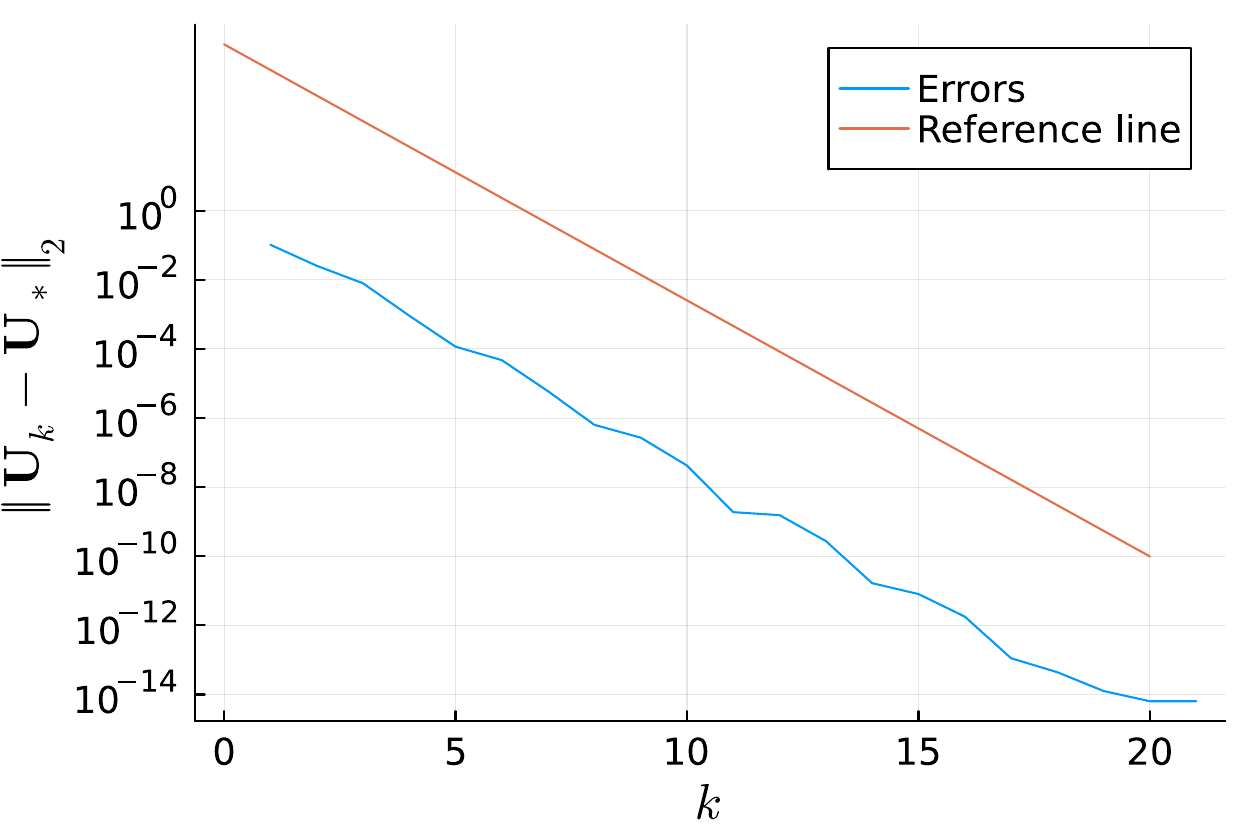}
	\end{subfigure}
	\begin{subfigure}{0.495\linewidth}
		\centering
		\includegraphics[width=\linewidth]{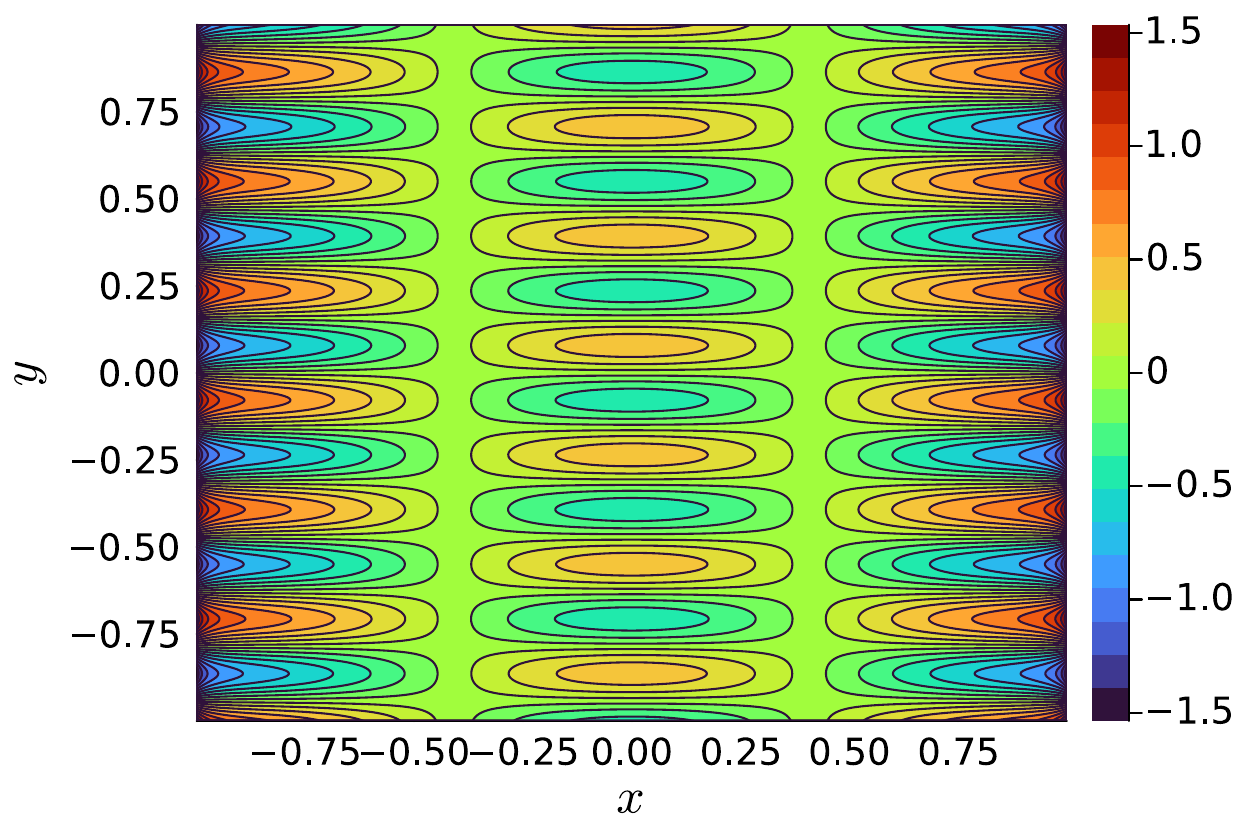}
	\end{subfigure}
\caption{2-norm error for $\bU_k$ against the true solution $\bU_*$ (computed via $\texttt{sylvester}$ in \Julia) at each iteration of Method 2 (Algorithm~\ref{alg:sylv_inv_lr}) in solving \eqref{eq:int_eqn_mat} (left) and corresponding approximate solution of \eqref{eq:int_eqn} at $n=2000$ gridpoints with $K_1(x,y)=K_2(x,y)=\exp(-2|x-y|)$, $f_1(x)=\frac{\cos(4x)}{1.04-x^2}$, and $g_1(x)=\sin(20x)$ (right). Here, the reference line denotes the convergence rate of Corollary~\ref{cor:inv_conv_rate} with $D_{\bA,\bB,\bC}$ replaced by the hypothesized upper bound $20(n+m)$.}
\label{fig:int_eqn}
\end{figure}

As an extension, given additional rank 1 kernel functions\footnote{This is a convenient but unnecessary assumption. To apply Method 2 in the manner described, at least one of the resulting matrices $\bK_3$ and $\bK_4$ must be representable in low-rank form. In this case, the number of required GMRES iterations is one plus the rank of $\bK_4\otimes\bK_3$.}  $K_3,K_4:[-1,1]^2\to\real$, consider the integral equation
\begin{equation}\label{eq:int_eqn_mod}
\begin{aligned}
2u(x,y)+\int_{-1}^1K_1(x,x')u(x',y)\df x'+\int_{-1}^1K_2(y,y')u(x,y')\df y'\\+\int_{-1}^1\int_{-1}^1K_3(x,x')u(x',y')K_4(y,y')\df x'\df y'=\sum_{\ell=1}^rf_\ell(x)g_\ell(y).
\end{aligned}
\end{equation}
Applying Gauss--Legendre quadrature to \eqref{eq:int_eqn_mod} in the same manner yields the generalized Sylvester equation
\begin{equation}\label{eq:int_eqn_mat_mod}
(\bI+\bK_1)\bU+\bU(\bI+\bK_2)+\bK_3\bU\bK_4^T=\sum_{\ell=1}^r\bff_\ell\bg_\ell^T,
\end{equation}
where $\bK_3,\bK_4\in\real^{n\times n}$ have entries
\begin{equation*}
\left(\bK_{3,4}\right)_{j,k}=\sqrt{w_jw_k}K_{3,4}(x_j,x_k),
\end{equation*}
and the other matrices are defined as in~\eqref{eq:int_sys_def}. To see how Method 2 can be used to solve \eqref{eq:int_eqn_mat_mod}, denote the solution to the Sylvester equation~\eqref{eq:Genmat} by the linear map
\begin{equation*}
\scrS_{\bA, \bB}^{-1} \bC = \scrT_{\bA,\bB}\bC:=\bX.
\end{equation*}
Then, \eqref{eq:int_eqn_mat_mod} is equivalent to the linear system
\begin{equation}\label{eq:mat_free}
\left(\bI+\scrT_{(\bI+\bK_1),(\bI+\bK_2)}\right) \bK_3\bU\bK_4^T =\scrT_{(\bI+\bK_1),(\bI+\bK_2)} \sum_{\ell=1}^r\bff_\ell\bg_\ell^T.
\end{equation}
Using Method 2 to compute the action of $\scrT_{(\bI+\bK_1),(\bI+\bK_2)}$, matrix-free GMRES or another iterative method can be used to solve~\eqref{eq:mat_free}.
% \begin{equation*}
% \left(\bI\otimes(\bI+\bK_1)+(\bI+\bK_2)\otimes\bI+\bK_4\otimes\bK_3\right)\mathrm{vec}(\bU)=\mathrm{vec}\left(\sum_{\ell=1}^r\bff_\ell\bg_\ell^T\right).
% \end{equation*} 
% In particular, since a Sylvester equation \eqref{eq:Genmat} is equivalent to the linear system
% \begin{equation*}
% (\bA^T\otimes\bI-\bI\otimes\bB)\mathrm{vec}(\bX)=\mathrm{vec}(\bC),
% \end{equation*}
% the action of $\left(\bI\otimes(\bI+\bK_1)+(\bI+\bK_2)\otimes\bI\right)^{-1}$ on $\mathrm{vec}(\bY)$ can be computed by solving
% \begin{equation}\label{eq:inner_gmres}
% (\bI+\bK_1)\bX+\bX(\bI+\bK_2)=\bY,
% \end{equation}
% for $\bX$ via the Akhiezer iteration. Note that none of these Kronecker-factored systems need ever be formed. 
% Since $K_3$ and $K_4$ are assumed to be rank 1, the $\bY$ in \eqref{eq:inner_gmres} will be rank 1, and Algorithm~\ref{alg:sylv_low_rank} can be applied directly. 
Furthermore, since this system is identity plus rank 1, GMRES will converge in at most two iterations. We include two solutions computed in this manner in Figure~\ref{fig:int_eqn_mod}.
\begin{figure}
	\centering
	\begin{subfigure}{0.495\linewidth}
		\centering
		\includegraphics[width=\linewidth]{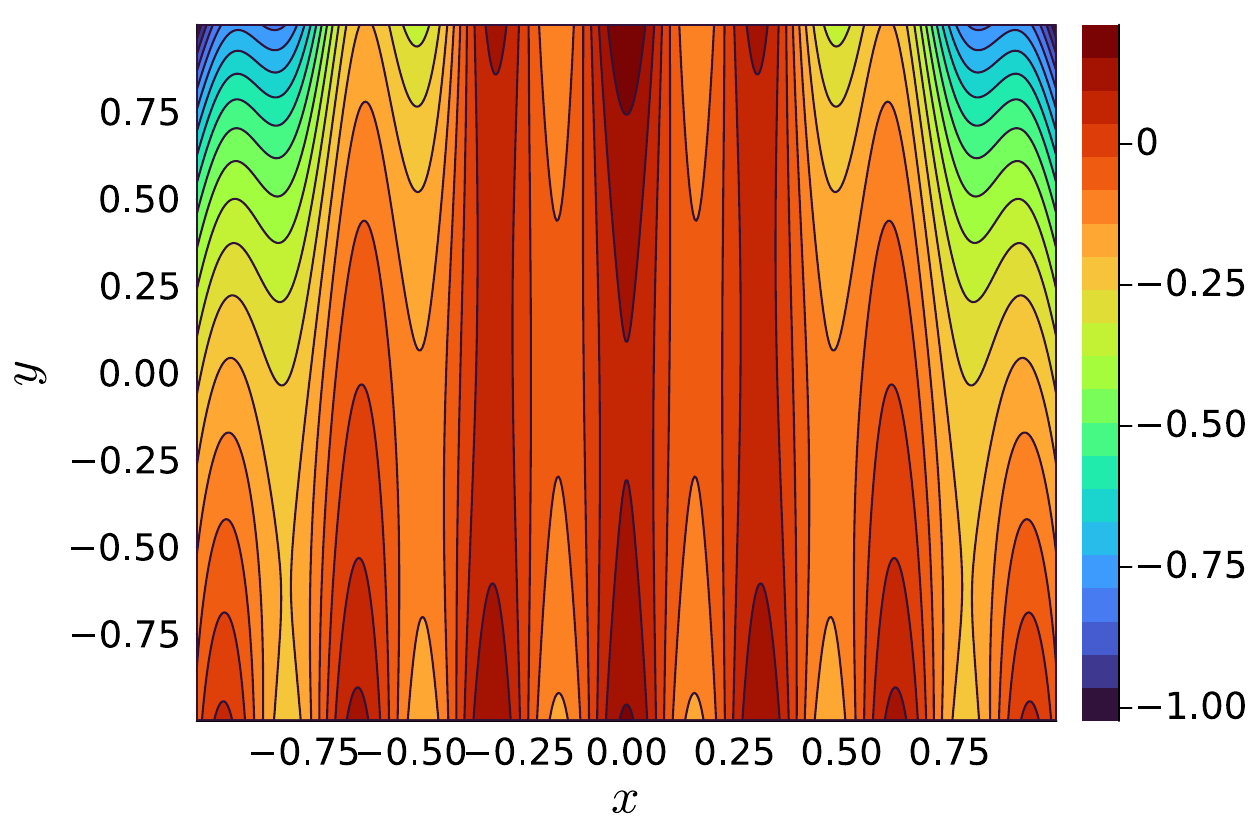}
	\end{subfigure}
	\begin{subfigure}{0.495\linewidth}
		\centering
		\includegraphics[width=\linewidth]{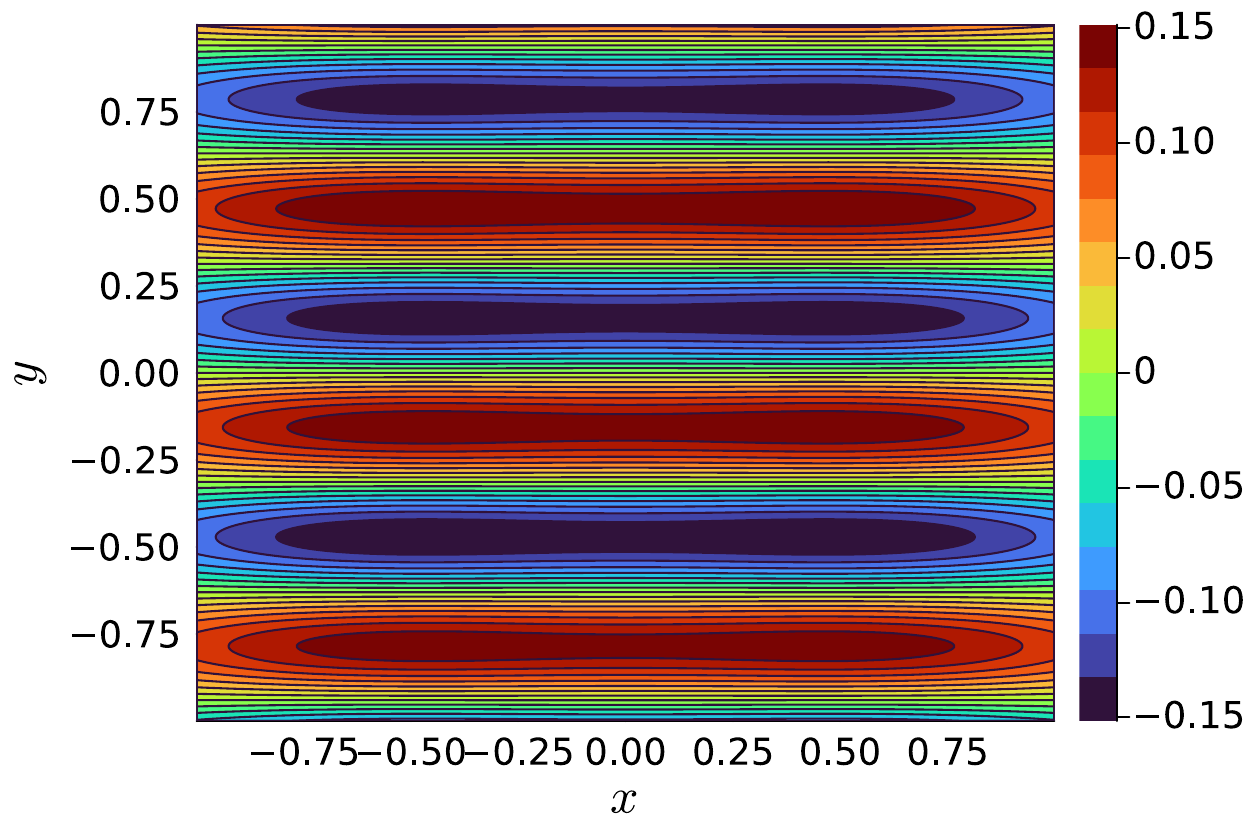}
	\end{subfigure}
\caption{Approximate solution of \eqref{eq:int_eqn_mod} at $n=2000$ gridpoints with $K_1(x,y)=K_2(x,y)=\exp(-2|x-y|)$, $K_3(x,y)=\cos(20x)\exp(y)$, $K_4(x,y)=\cosh(x)\sinh(y)$, $f_1(x)=x^2$, and $g_1(x)=-\exp(x)$ (left) and  $K_1(x,y)=K_2(x,y)=\exp(-(x-y)^2)$, $K_3(x,y)=y\sech^2(x)$, $K_4(x,y)=\exp(x-y)$, $f_1(x)=\frac{1}{x^4+2}$, and $g_1(x)=-\sin(10x)$ (right). In both cases, preconditioned GMRES applied to \eqref{eq:int_eqn_mat_mod} converges in at most two iterations.}
\label{fig:int_eqn_mod}
\end{figure}

\subsection{Collocation for the ``good'' Helmholtz equation}\label{sect:Poisson}
Consider solving the inhomogeneous ``good'' Helmholtz equation
% \textcolor{red}{I think we should maybe refer to this as the Yukawa equation, which is where k is either zero or purely imaginary. If k is real will the intervals of the Sylvester coefficient matrices overlap?}
on the square,
\begin{align*}
     u_{xx}(x,y) &+ u_{yy}(x,y) - k^2u(x,y)= f(x,y), \quad (x,y) \in (-1,1)^2, \quad k \in \mathbb{R},\\
    u(\pm 1, y) &= 0, \quad y \in [-1,1],\\
    u(x, \pm 1) &= 0, \quad x \in [-1,1].
\end{align*}
We solve this problem not because the method developed will compete with the state-of-the-art solvers but to demonstrate three things: first, we show how a discretization can be produced that leads to a Sylvester matrix equation involving non-symmetric dense matrices with real eigenvalues. Second, we use this as an opportunity to highlight the kinds of challenges that arise in this method when discretizing operators with unbounded spectra. As we show, more work is needed to make the Akhiezer iteration practical in this setting\footnote{We use Method 1 here to highlight an alternative approach to compute the series coefficients for unbounded operators. The convergence could be improved upon by instead using Method 2, but the slow convergence of a polynomial method relative to the growth of the eigenvalues still makes this approach impractical at scale.}. In particular, an effective preconditioner that preserves the real spectrum of operators is desirable and would potentially make the method very effective.  Finally, we use the slow convergence rate of the algorithm in this setting to observe that Method 1 exhibits remarkable stability, even when applied over thousands of iterations. Using an inverse-free, unpreconditioned method tests the limits of our approach. We note that in the case where $k =0$ (Poisson's equation), several effective solvers already exist, including an optimal complexity spectral method based on ultraspherical polynomials and the ADI method~\cite{ADIPoisson}. 
%We note that sophisticated higher-order methods, such as the hierarchical Poincare-Steklov method~\cite{gillman2014direct}, can already handle this problem well. Because we use an inverse-free, unpreconditioned method, we want to test the limits of the inverse-free Sylvester solver.  We have yet to discover a preconditioner that preserves the real spectrum, but if one was developed, we believe that this method could be quite effective in practice.   

% Barring such a development, one should really just use a finite difference scheme or, if a spectral method is desired, the method of Fortunato and Townsend \cite{ADIPoisson} is similar to our approach, but the banded nature of their construction enables allows ADI to be efficiently applied. \textcolor{red}{I don't think the Fortunato and Townsend method is applicable here because the method is specially tuned to the Poisson equation. In fact ADI breaks down for the straightforward discretization of the Helmholtz equation. I would omit this paragraph.}

To understand how to discretize this problem using collocation so that an appropriate Sylvester matrix equation is obtained, we consider the collocation discretization of a one-dimensional eigenvalue problem:
\begin{align}\label{odebvp}
\begin{split}
    u''(x) &= z u(x), \quad x \in (-1,1),\quad u(\pm1) = 0.
    \end{split}
\end{align}
Let $\left(p_j^{(\lambda)}(x)\right)_{j=0}^\infty$ denote orthonormal ultraspherical polynomials with parameter $\lambda$, i.e., the orthonormal polynomials on $[-1,1]$ with respect to the normalized weight function $Z_\lambda^{-1} (1 - x^2)^{\lambda - 1/2}$.
To keep notation a bit more classical, let $\left(\check T_j(x)\right)_{j=0}^\infty$ denote orthonormal Chebyshev polynomials of the first kind (ultraspherical, $\lambda = 0$) with respect to the normalized weight function $\pi^{-1} (1 - x^2)^{-1/2}$ on $[-1,1]$.  We follow the approach in \cite{Trogdon2024a}.  From \cite{Olver2013a}, for example, there exists a strictly upper-triangular, sparse $N\times N$ matrix $\mathbf D_2$ such that if
\begin{align*}
     v(x) &= \sum_{j=0}^{N-1} v_j \check T_j(x), \quad \vec v = (v_0,\ldots,v_{N-1})^T,\\
     \vec w &= (w_0, \ldots w_{N-1})^T, \quad  \mathbf w = \mathbf D_2 \mathbf v,
\end{align*}
then
\begin{align*}
    v''(x) = \sum_{j=0}^{N-1} w_j p_j^{(2)}(x).
\end{align*}
 Now, for a grid of $M$ points $-1 < x_1 < x_2 < \cdots < x_M < 1$,  $x_j = x_j(M)$, define the $M \times N$ matrix,
\begin{align*}
    \vec E_M(\lambda) = ( p_j^{(\lambda)}(x_i))_{\substack{1 \leq i \leq M \\ 0 \leq j \leq N-1}},
\end{align*}
and define the boundary operator
\begin{align*}
    \vec B = \begin{pmatrix}
        \check T_0(1) & \check T_1(1) & \cdots & \check T_{N-1}(1)\\
        \check T_0(-1) & \check T_1(-1) & \cdots & \check T_{N-1}(-1)
    \end{pmatrix}.
\end{align*}
The discretization of the problem \eqref{odebvp} is then given by the generalized eigenvalue problem
\begin{align*}
    \begin{pmatrix}
        \vec B \\
        \vec E_{N-2}(2) \vec D_2
    \end{pmatrix} \vec u = z \begin{pmatrix} \vec 0 \\
    \vec E_{N-2}(0)
    \end{pmatrix} \vec u.
\end{align*}
To write this as a conventional eigenvalue problem, we let\footnote{We note that it is possible to use a different choice for $\vec T$ that is sparse, but we do not need this here.}
% \textcolor{red}{I don't understand the notation for your subselection of columns of Q}
\begin{align*}
    \vec Q, \vec R  = \mathrm{qr}(\vec B^T), \quad \vec T = \vec Q_{:,3:N},
\end{align*}
so that the columns of $\vec T$ form a basis for the nullspace of $\vec B$.  Then, in writing $\vec u = \vec T \vec v$, we find the eigenvalue problem\footnote{The invertibility of $\vec E_{N-2}(0) \vec T$ follows from the fact that the only polynomial of degree $N$ that interpolates zero at the union of the nodes and $\{1,-1\}$ is the trivial one.}
\begin{align*}
    \vec E_{N-2}(2) \vec D_2
    \vec T\vec v = z  \vec E_{N-2}(0)\vec T \vec v \quad \Rightarrow \quad ( \vec E_{N-2}(0)\vec T )^{-1} \vec E_{N-2}(2) \vec D_2
    \vec T\vec v = z \vec v.
\end{align*}
Then, define
\begin{align*}
    \vec A_N = ( \vec E_{N-2}(0)\vec T )^{-1} \vec E_{N-2}(2) \vec D_2
    \vec T.
\end{align*}
Empirically, and surprisingly, we see that $\vec A_N$ has negative eigenvalues.

We approximate
\begin{align*}
    u(x,y) \approx u_N(x,y) = \sum_{j=0}^{N-1} \sum_{i=0}^{N-1}  u_{ij} \check T_j(x) \check T_i(y),
\end{align*}
and see that the discretization of the ``good'' Helmholtz equation is then given by
\begin{align}\label{eq:pois_sylv}
    \vec A_N \vec Y &+ \vec Y \vec A_N^T - k^2 \vec Y = \vec G, \quad \vec X = \vec T \vec Y \vec T^T, \quad \vec X = (u_{ij})_{\substack{1 \leq i \leq N\\ 1 \leq j \leq N}},\\
    \vec G &= ( \vec E_{N-2}(0)\vec T )^{-1} \vec F ( \vec E_{N-2}(0)\vec T )^{-T}, \quad \vec F = (f(x_i,x_j))_{\substack{1 \leq i \leq N-2\\ 1 \leq j \leq N-2}}. \notag
\end{align}

The success of the above approach is limited due to the fact that the Laplacian is unbounded and the eigenvalues of $\bA_N$ grow rapidly as $N$ increases. As a consequence of this, the approximation of the series coefficients $\alpha_j$ via circular contours around each interval of $\Sigma = \Sigma_N$ requires an increasing number of quadrature points as $N$ increases. The latter issue can be partially addressed using the principal value integral
\begin{equation*}
-\frac{1}{\pi\im}\pvint_{-\im\infty}^{\im\infty}\frac{\df z}{z-x}=\sign(x),\quad x\in\compl\setminus\im\real,
\end{equation*}
% instead considering infinite counterclockwise-oriented semi-circular contours in the left and right half-planes. The integrals over the arcs cancel, and we are left with the principal value integral
so that the series coefficients are given by
\begin{equation*}
\alpha_j=\frac{1}{2\pi\im}\int_\Sigma\left(2\pvint_{\im\infty}^{-\im\infty}\frac{\df z}{z-x}\right)p_j(x)w(x)\df x.
\end{equation*}
Following \cite[Section 7.1]{OperatorFEAST}, the change of variables $y=\frac{2}{\pi}\arctan(\im z)$ yields
\begin{equation*}
\alpha_j=\frac{1}{2\pi\im}\int_\Sigma\left(\im\pi\int_{-1}^{1}\frac{\sec^2\left(\frac{\pi}{2}y\right)}{x+\im\tan\left(\frac{\pi}{2}y\right)}\df y\right)p_j(x)w(x)\df x.
\end{equation*}
Denote the Gauss-Legendre quadrature nodes and weights by $\{y_\ell\}_{\ell=1}^m$ and $\{w_\ell\}_{\ell=1}^m$, respectively. Then, 
\begin{equation*}
\alpha_j\approx\im\pi\sum_{\ell=1}^m\left(z_\ell^2+1\right)w_\ell\cC_\Sigma\left[p_jw\right](-\im z_\ell),\quad z_\ell=\tan\left(\frac{\pi y_\ell}{2}\right).
\end{equation*}
Empirically, we find that this requires fewer quadrature points to obtain an accurate approximation to the series coefficients $\alpha_j$ when the eigenvalues of $\bA_N$ are large\footnote{In fact, this approach seems to require fewer quadrature points in most cases, even when the intervals of $\Sigma$ are small and well separated. When $0$ does not lie in the gap between intervals, these formulae will need to be shifted appropriately.}. In Figures~\ref{fig:poisson}~and~\ref{fig:helmholtz}, we plot approximate solutions obtained in this manner as well as the error at each iteration in solving \eqref{eq:pois_sylv}. As expected, Method 1 requires many iterations even when $N$ is small, but we are able to generate approximate solutions nonetheless.
\begin{figure}
	\centering
	\begin{subfigure}{0.495\linewidth}
		\centering
		\includegraphics[width=\linewidth]{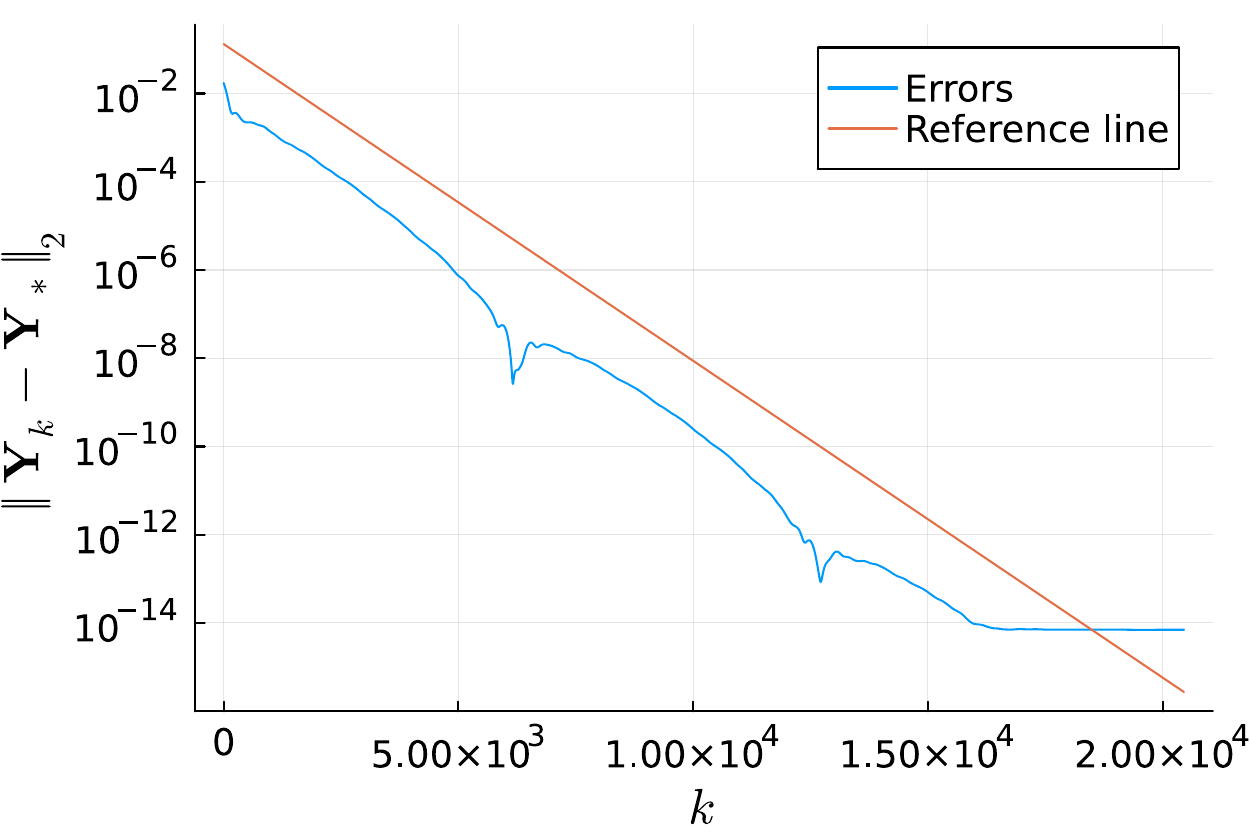}
	\end{subfigure}
	\begin{subfigure}{0.495\linewidth}
		\centering
		\includegraphics[width=\linewidth]{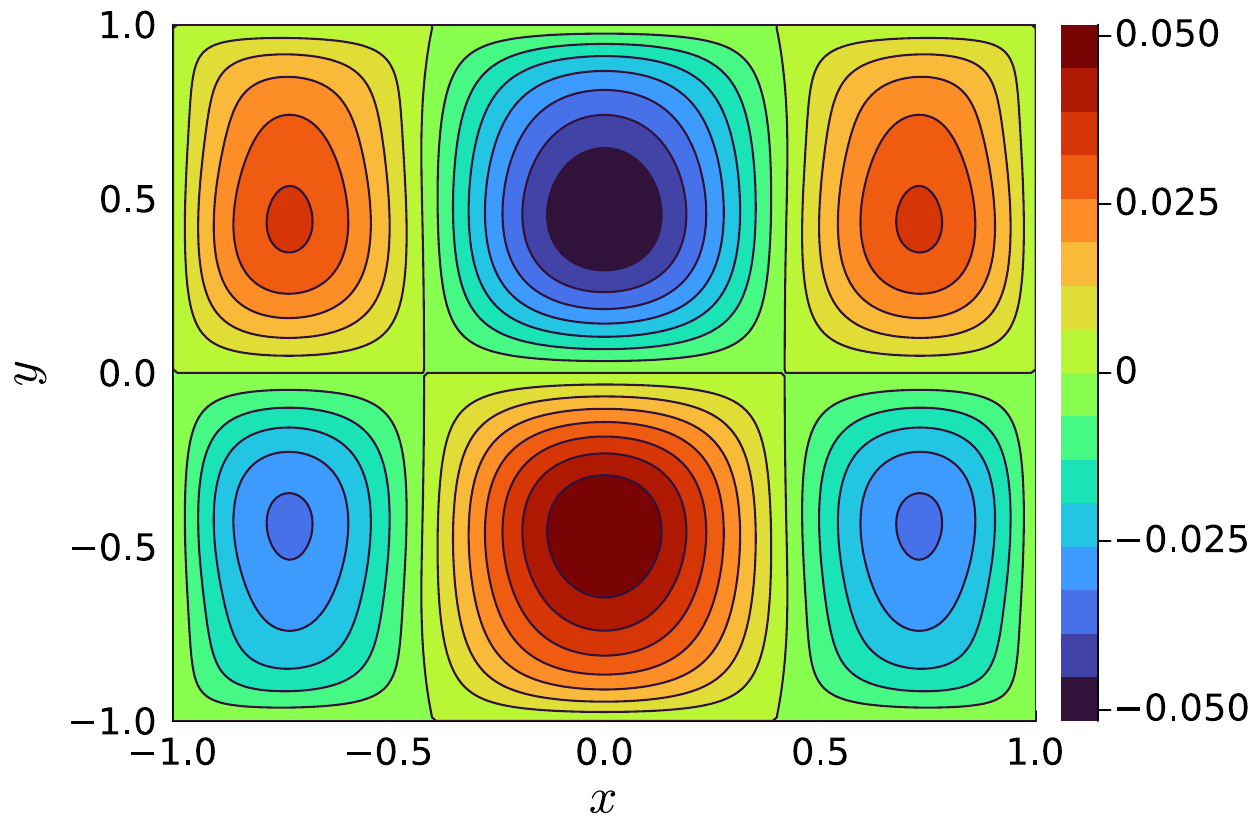}
	\end{subfigure}
\caption{2-norm error for $\bY_k$ against the true solution $\bY_*$ (computed via $\texttt{sylvester}$ in \Julia) at each iteration of Method 1 (Algorithm~\ref{alg:sylv_low_rank}) in solving \eqref{eq:pois_sylv} (left) and corresponding $N=10$ approximate solution of Poisson's equation ($k=0$) with $f(x,y)=\cos(4x)\sign(y)$ (right). Here, the reference line denotes the convergence rate heuristic \eqref{eq:conv_heur} with $D_\bH$ replaced by the hypothesized upper bound $10(n+m)$.}
\label{fig:poisson}
\end{figure}
\begin{figure}
	\centering
	\begin{subfigure}{0.495\linewidth}
		\centering
		\includegraphics[width=\linewidth]{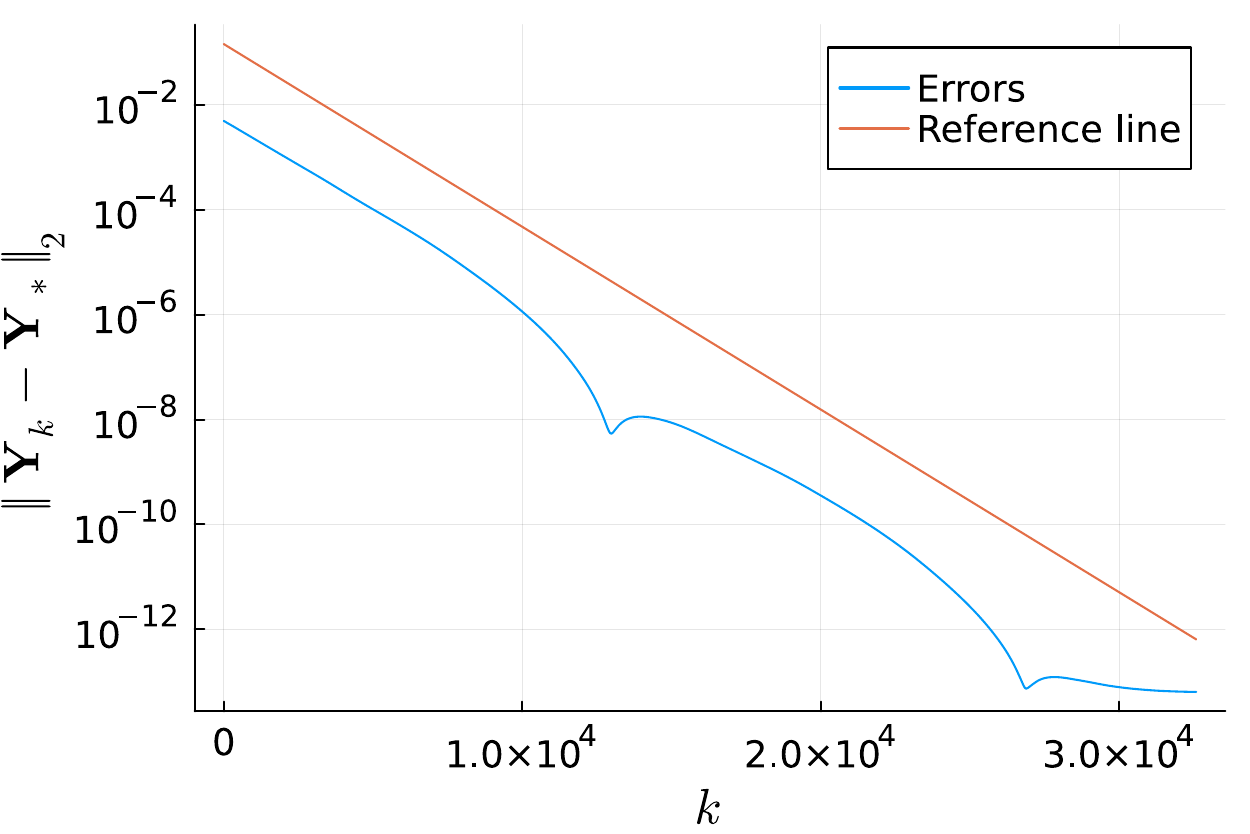}
	\end{subfigure}
	\begin{subfigure}{0.495\linewidth}
		\centering
		\includegraphics[width=\linewidth]{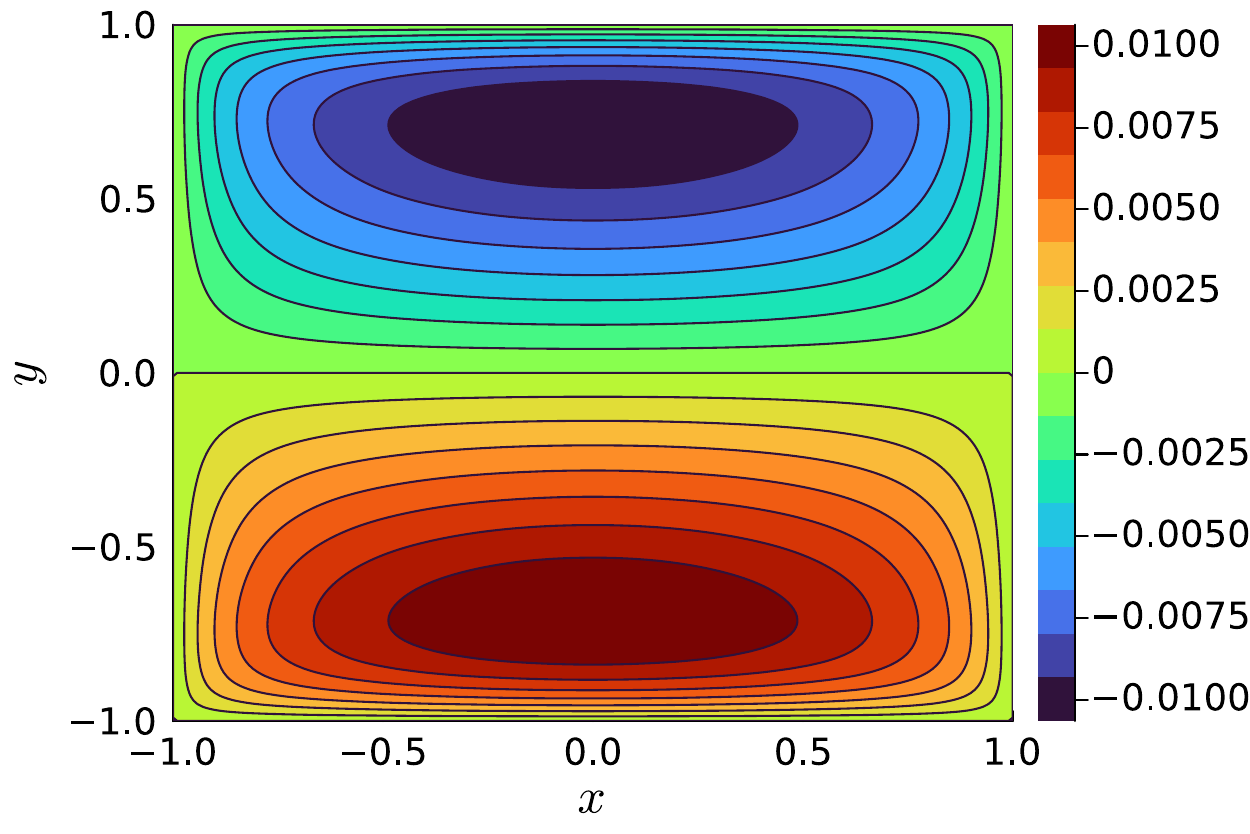}
	\end{subfigure}
\caption{2-norm error for $\bY_k$ against the true solution $\bY_*$ (computed via $\texttt{sylvester}$ in \Julia) at each iteration of Method 1 (Algorithm~\ref{alg:sylv_low_rank}) in solving \eqref{eq:pois_sylv} (left) and corresponding $N=20$ approximate solution of the ``good'' Helmholtz equation with $f(x,y)=\cos(x)\sin(y)$ and $k=7$ (right). Here, the reference line denotes the convergence rate heuristic \eqref{eq:conv_heur} with $D_\bH$ replaced by the hypothesized upper bound $10(n+m)$.}
\label{fig:helmholtz}
\end{figure}

While poor conditioning and the small problem size mean that Method 1 is uncompetitive with direct solvers such as the Bartels--Stewart algorithm for this task, this application demonstrates the behavior of our Sylvester equation solver applied to ill-conditioned problems. Even in these settings where the eigenvalues are large, we observe that Method 1 appears stable and converges at the rate given by our analysis.
% ...?? \textcolor{red}{perhaps add something about how this has helped us with insight about handling large intervals/the method is stable and converges even in such a setting?}
%other iterative methods will likely exhibit similarly slow convergence, as $\bA_N$ will have increasingly large condition number as $N$ increases.

% A similar Sylvester equation-based spectral method for Poisson's equation on the square appears in \cite{ADIPoisson}. While the matrix the authors construct also has a growing condition number, as its eigenvalues tend to zero, the banded nature of their construction allows ADI to be efficiently applied. As we observe in Figure~\ref{fig:scal_approx}, the difference between polynomial and rational approximations to the sign function becomes more pronounced as the domains approach the discontinuities, so their method will likely outperform ours in this case.

% \color{red}TODO: Add more examples/applications, but this should be sufficient if we can't think of others.\color{black}

\section{Other applications}\label{sect:other_apps}

\subsection{Algebraic Riccati equations}
	% \color{red}TODO: This is probably still worth mentioning, but the case that we can solve is actually pretty trivial. That is, we're restricted to the problem
 %    \begin{equation*}
 %    \bA\bX+\bX\bB-\bX\bC\bX=0,
 %    \end{equation*}
 %    which of course has a zero solution. If the solution $\bX$ is invertible, this is equivalent to
 %    \begin{equation*}
 %    \bX^{-1}\bA+\bB\bX^{-1}-\bC=0,
 %    \end{equation*}
 %    a Sylvester equation for $\bX^{-1}$. More generally, we can basically do the same thing as with the Sylvester equation and just solve a least squares problem at the end. This may require that all involved matrices are square which is something to investigate further.

 %    I think the idea is this: the factorization in \cite[Section 2.4]{higham2008functions} only works when the inolved matrices are square. Otherwise, the factorization in the Riccati equation book just yields the trivial solution.
 % \color{black} 

An algebraic Ricatti equation is a generalization of the Sylvester equation of the form
\begin{equation}\label{eq:ricatti}
\bA\bX-\bX\bB-\bX\bD\bX=\bC.
\end{equation}
Under certain conditions on the data matrices, a solution to \eqref{eq:ricatti} can be obtained by computing the sign of the matrix \cite{Bini2011}
\begin{equation*}
\bH=\begin{pmatrix}\bA&\bD \\ \bC&\bB\end{pmatrix}.
\end{equation*}
If the eigenvalues of $\bH$ are known to lie on or near a collection of disjoint intervals $\Sigma$, solving the Ricatti equation \eqref{eq:ricatti} is a direct application of the Akhiezer iteration, albeit without the decoupled and low-rank structure of the Sylvester equation algorithms. However, if only the eigenvalues of the individual data matrices are known, little can be said in general about the eigenvalues of $\bH$. Particular cases when the eigenvalues of $\bH$ are easily determined are when $\bC=\bzero$ and when $\bD=\bzero$. The latter case reduces precisely to the Sylvester equation \eqref{eq:Genmat}. The former case has the trivial solution $\bX=\bzero$. If an invertible solution $\bX$ exists, \eqref{eq:ricatti} is equivalent to 
\begin{equation*}
\bX^{-1}\bA-\bB\bX^{-1}-\bD=0,
\end{equation*}
which is simply a Sylvester equation for $\bX^{-1}$. Thus, while Method 1 can, in principle, be generalized to algebraic Ricatti equations, knowledge of the eigenvalues of the full block matrix $\bH$ is required to compute anything nontrivial. If the eigenvalues of $\bH$ are suspected to lie on a collection of disjoint intervals, the adapative procedure described in \cite[Section 5]{AkhIter} could be employed to approximate these intervals, then the Akhiezer iteration could be applied directly to compute $\sign(\bH)$ and solve~\eqref{eq:ricatti}.

\subsection{Fr\'echet derivatives}
A natural application of this work is to the computation of Fr\'echet derivatives. Given a matrix $\bA\in\compl^{n\times n}$ and matrix function $f:\compl^{n\times n}\to\compl^{n\times n}$, the Fr\'echet derivative of $f$ at $\bA$ is a linear map $L_f(\bA,\diamond)$ satisfying 
\begin{equation*}
f(\bA+\bE)-f(\bA)-L_f(\bA,\bE)=\oo\left(\|\bE\|\right),
\end{equation*}
for all $\bE\in\compl^{n\times n}$. Fr\'echet derivatives need not exist, but must be unique if they do \cite[Section 3.1]{higham2008functions}. By modifying \cite[Theorem 2.1]{MathiasFrechet} and its proof to use lower triangular matrices, if $p$ denotes the size of the largest Jordan block of $\bA$ and $f$ is $2p-1$ times continuously differentiable on an open set containing the spectrum of $\bA$, then for any $\bE\in\compl^{n\times n}$,
\begin{equation*}
	f\begin{pmatrix}\bA&\bzero \\ \bE&\bA\end{pmatrix}=\begin{pmatrix}f(\bA)&\bzero \\ L_f(\bA,\bE)&f(\bA)\end{pmatrix}.
\end{equation*}
Since computing the Fr\'echet derivative requires only the the lower left block of this matrix function, the recurrences utilized in Method 1 can be applied nearly directly with $f$ in place of the sign function; only the series coefficients $\alpha_j$ need be modified appropriately. In Figure~\ref{fig:frechet}, we plot the error at each iteration of computing the Fr\'echet derivative for the matrix sign function and matrix exponential. We observe that the method converges quickly with the same convergence guarantees as when it is applied to Sylvester equations. 
\begin{figure}
	\centering
	\begin{subfigure}{0.495\linewidth}
		\centering
		\includegraphics[width=\linewidth]{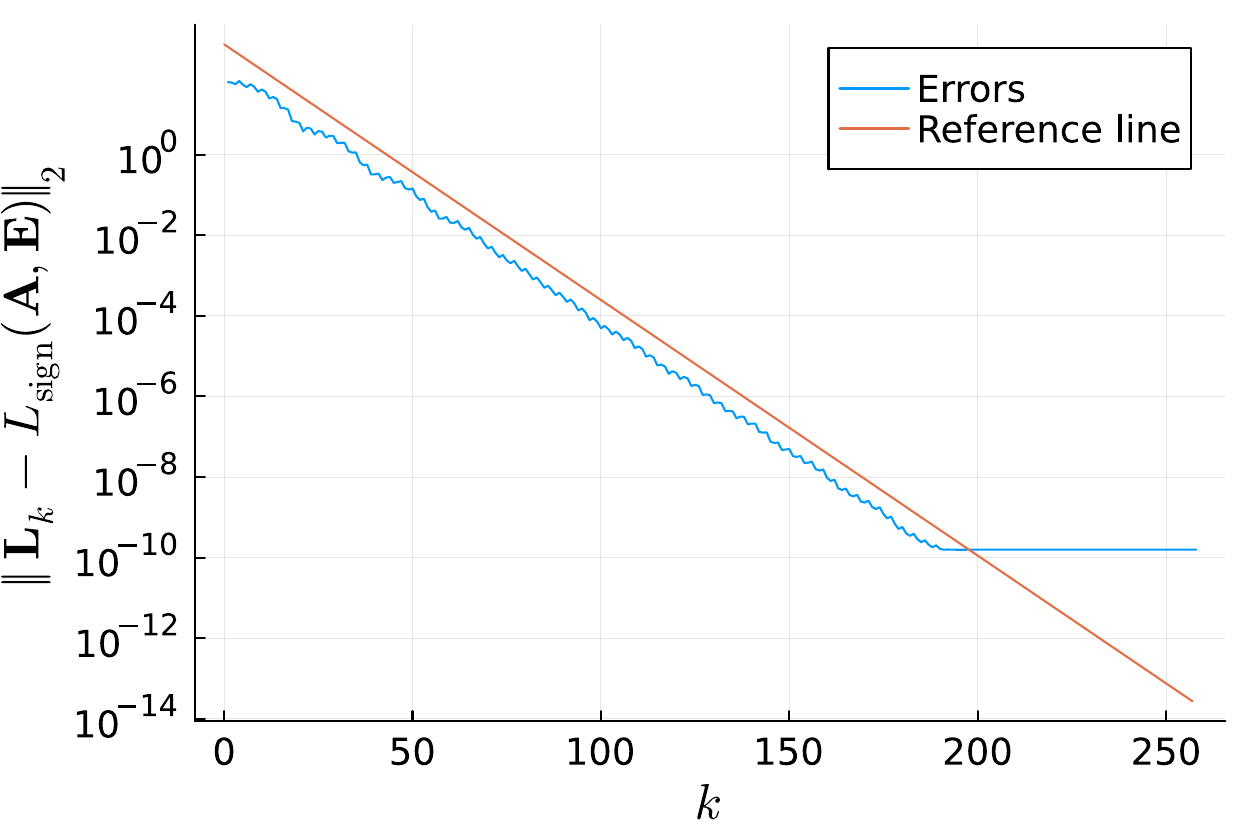}
	\end{subfigure}
	\begin{subfigure}{0.495\linewidth}
		\centering
		\includegraphics[width=\linewidth]{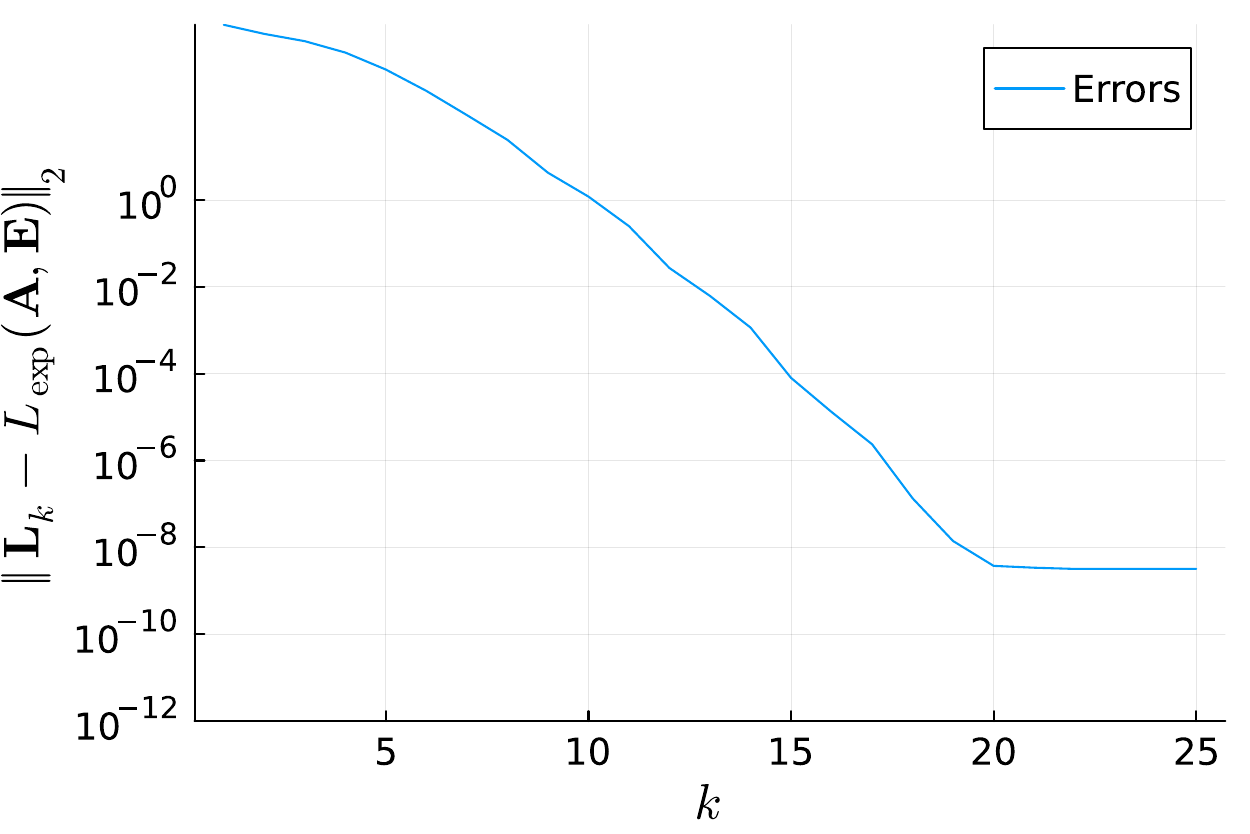}
	\end{subfigure}
\caption{2-norm error for approximation $\bL_k$ to $L_f(\bA,\bE)$ at each iteration of Algorithm~\ref{alg:sylv_low_rank} modified to compute Fr\'echet derivatives for $f(x)=\sign(x)$ (left) and $f(x)=\exp(x)$ (right). Here, $\bA\in\real^{200\times200}$ has eigenvalues in $[-2,-0.5]\cup[0.5,6]$ and $\bE$ is rank 4. Because the exponential function is entire, our algorithm achieves superexponential convergence before errors saturate. The true solution is computed via the Newton iteration \cite[Section 5]{AlMohy2009} (left) and the built-in \Julia\ function $\texttt{exp}$ applied to $\begin{pmatrix}\bA&\bzero \\ \bE&\bA\end{pmatrix}$ (right). Here, the reference line denotes the convergence rate bound \eqref{eq:big_conv_rate} with $D_\bM$ replaced by the hypothesized upper bound $10\ell$ for generic $\bM\in\compl^{\ell\times\ell}$.}
\label{fig:frechet}
\end{figure}
The Fr\'echet derivative of the matrix sign function is useful for analyzing the stability of algorithms to compute the matrix sign, matrix square root, and other related matrix functions \cite{Kenney1991}.

\subsection{Matrix square roots and the polar decomposition} \label{sect:sqrt}
The matrix sign function can also be used to compute the matrix square root and inverse square root. In particular, if $\bA\in\compl^{n\times n}$ has no eigenvalues on $(-\infty,0]$, then \cite[Equation 6.32]{higham2008functions}
\begin{equation*}
\sign\begin{pmatrix}0 & \bA \\\bI & 0
\end{pmatrix}=\begin{pmatrix}
0 & \bA^{1 / 2} \\
\bA^{-1 / 2} & 0
\end{pmatrix}.
\end{equation*}
If $\bA$ has eigenvalues contained in or near some interval $[a,b]$ on the positive real axis, the eigenvalues of $\begin{pmatrix}0 & \bA \\\bI & 0\end{pmatrix}$ are in or near $[-\sqrt{b},-\sqrt{a}]\cup[\sqrt{a},\sqrt{b}]$, and the Akhiezer iteration can be applied to compute $\bA^{1 / 2}$ and $\bA^{-1 / 2}$ in parallel. The off-diagonal structure of the relevant block matrix prevents the direct application of Lemma~\ref{lem:block_recurr}, and an analogous result for off-diagonal block $2\times 2$ matrices would need to be worked out for this approach to be made efficient. It is currently unclear whether this would be more efficient than directly applying the Akhiezer iteration to $\bA$ to compute $\bA^{1/2}$ and $\bA^{-1/2}$ separately. Regardless, the Akhiezer iteration should produce an inverse-free method for computing the matrix square root and inverse square root; we leave this for future work.

Similarly, the matrix sign function is connected to the polar decomposition, which can be used to solve the orthogonal Procrustes problem \cite[Theorem 8.6a]{higham2008functions}. If the polar decomposition of a matrix $\bA\in\compl^{m\times n}$ $(m\geq n)$ is given by $\bA=\bU\bH$ where $\bU\in\compl^{m\times n}$ has orthonormal columns and $\bH\in\compl^{n\times n}$ is Hermitian and positive semidefinite, then \cite[Equation 8.6]{higham2008functions}
\begin{equation*}
\sign\begin{pmatrix}0 & \bA \\\bA^* & 0
\end{pmatrix}=\begin{pmatrix}
0 & \bU \\
\bU^* & 0
\end{pmatrix}.
\end{equation*}
Thus, the Akhiezer iteration can be applied to compute the polar factor $\vec U$ of the polar decomposition, in particular when the eigenvalues of $\bA^*\bA$ are located in or near a known interval on the real axis. In the special case where $\bA$ is symmetric, an algorithm is immediate as ${\rm sign}(\bA) = \bU$. Otherwise, an off-diagonal analog to Lemma~\ref{lem:block_recurr} would likely be necessary for this approach to be made efficient. Another option is to compute the polar factor by applying the Akhiezer iteration to the inverse square root with the formula $\bU=\bA(\bA^*\bA)^{-1/2}$ \cite[Chapter 8]{HighamFrechet}. 

The Fr\'echet derivative of the polar decomposition can be obtained via the identity \cite{GawlikPolar}
\begin{equation*}
\sign\left(\begin{array}{cccc}
0 & \bA & 0 & \bE \\
\bA^* & 0 & \bE^* & 0 \\
0 & 0 & 0 & \bA \\
0 & 0 & \bA^* & 0
\end{array}\right)=\left(\begin{array}{cccc}
0 & \mathcal{P}(\bA) & 0 & L_{\mathcal{P}}(\bA, \bE) \\
\mathcal{P}(\bA)^* & 0 & L_{\mathcal{P}}(\bA, \bE)^* & 0 \\
0 & 0 & 0 & \mathcal{P}(\bA) \\
0 & 0 & \mathcal{P}(\bA)^* & 0
\end{array}\right).
\end{equation*}
Here, $\cP$ denotes a function that maps a matrix to its polar factor and $\bA,\bE\in\compl^{n\times n}$. The Akhiezer iteration can again be applied to compute $L_{\mathcal{P}}(\bA, \bE)$, in particular when the eigenvalues $\bA^*\bA$ are located in or near a known interval on the real axis since the eigenvalues of the relevant block matrix are given by both square roots of these eigenvalues. Once again, an analog of Lemma~\ref{lem:block_recurr} for $4\times4$ block matrices with the relevant sparsity pattern is likely needed for efficiency, but this would yield an inverse-free method that, in contrast with that of \cite{GawlikPolar}, would not require that all singular values of $\bA$ be small.

\section{Discussion and future work}\label{sect:discussion}

Perhaps the largest question that we have not addressed is that of adaptivity. Throughout this work, we have assumed access to intervals on the real axis that (approximately) contain the spectra of the coefficient matrices $\bA$ and $\bB$. If these intervals are instead unknown, the approach outlined in \cite[Section 5]{AkhIter} can be applied to the block matrix $\bH$ from~\eqref{eq:block_sylv_eqn} to approximate the necessary intervals. Alternatively, since the problem of finding the eigenvalues of $\bH$ decouples, the same method with a single interval and shifted and scaled Chebyshev polynomials could be applied to each coefficient matrix individually. In the low-rank case of Method 1 (Algorithm~\ref{alg:sylv_low_rank}), one could also consider performing this analysis with the norms of the iterates $\bJ_k$, $\bK_k$ as empirically, these appear to grow at the same rate as the norm of $\bH$. 

A more difficult generalization would be allowing the coefficient matrices to have nonreal eigenvalues. For instance, $\bA$ and $\bB$ could have eigenvalues contained in disjoint open sets in the complex plane. If these sets are well-separated, there potentially exists a pair of intervals on the real axis, or a pair of lines in the complex plane, that enable Method 1 to converge in a reasonable number of iterations. Such a result exists for the related Chebyshev iteration and relies on the notion of an ``optimal ellipse'' \cite{Manteuffel1977,manteuffel_adaptive_1978}. Extending this notion to higher genus cut domains is a challenging open problem. Solving it likely requires a complete understanding of the generalized Bernstein ellipses of Figure~\ref{fig:bernstein}. Such a generalization could, for instance, enable an inverse-free method for the matrix-sign approach to eigenvalue computation when used in place of the Newton iteration in \cite[Section 4]{Banks2022}.

As we demonstrate in Figure~\ref{fig:conv_rates}, both methods introduced in this work achieve slower convergence rates than the ``worst case'' error bound of a polynomial Krylov method. This is because a polynomial Krylov method constructs a bivariate approximation, while the current framework of the Akhiezer iteration only allows for univariate approximation. Future work will aim to construct a method that achieves this faster convergence rate by extending the Akhiezer iteration to multivariate matrix functions.

For symmetric matrices, ${\rm sign}(\bM)$ is equal to the polar factor for $\bM$. The Akhiezer polynomials could be applied in the context of divide-and-conquer eigensolvers based on spectral projections attained via the polar factor~\cite{nakatsukasa2016computing}. An inverse-free polar decomposition method based on Halley's iteration was developed in~\cite{nakatsukasa2010optimizing}, where inverse factors are rewritten using QR decompositions of a block matrix related to the iterate. However, the Akhiezer polynomials should lead to an eigensolver with the polar factorization step only requiring matrix products. Akhiezer polynomials could also potentially provide a cheaper alternative to the minimax polynomials in the recent Polar Express algorithm \cite{PolarExpress}.

We have not attempted to implement methods involving hierarchical numerical linear algebra, but a natural connection to consider is the setting where $\bA$, $\bB$ and $\bC$ are rank-structured. As shown in~\cite{kressner2019low}, it is often the case that $\bX$ is well approximated in this setting by a rank-structured matrix.  Combining the Akhiezer iteration with the divide-and-conquer scheme in~\cite{kressner2019low} would take advantage of fast multipole-based matrix-vector products available for $\bA$ and $\bB$ and potentially supply an inverse-free version of the algorithms.

% I suggest perhaps we omit this application: 
%
%An alternative inverse-free method for computing $\sign(\bH)$ is the Newton--Schulz iteration:
%\begin{equation}\label{eq:newton_schulz}
%\bH_{k+1}=\frac{1}{2}\bH_k\left(3\bI-\bH_k^2\right).
%\end{equation}
%This iteration achieves quadratic convergence but requires that $\|\bI-\bA^2\|<1$ in any subordinate matrix norm \cite[Section 5.3]{higham2008functions}. If one were to compute an bound on this quantity from our convergence analysis and heuristics, they may wish to switch over from the Akhiezer iteration to the Newton--Schulz iteration once this condition is achieved. One complication in implementing such a method is that Algorithm~\ref{alg:sylv_low_rank} does not compute the full diagonal blocks of $\sign(\bH)$, so \eqref{eq:newton_schulz} cannot be directly applied its output.

\section*{Acknowledgments}
We thank Mika\"el Slevinsky and Sheehan Olver for valuable discussions and Nick Trefethen, Alex Townsend, Yuji Nakatsukasa, Bernhard Beckermann, and Daan Huybrechs for their feedback on the initial version of this paper.  This work was supported in part by the National Science Foundation under DMS-2306438 (TT) and DMS-2410045 (HW).

\appendix
\section{Additional considerations}\label{ap:eig}

In this section, we first discuss matrices that have complex eigenvalues and then discuss truncation errors from the approximation of contour integrals. This is done in the context of the general Akhiezer iteration and Method 1, but the analysis of Method 2 follows in the same manner. In particular, its convergence rate is governed by the ``worst offender'' eigenvalue in $\sigma(\bA)-\sigma(\bB)$ rather than $\sigma(\bH)=\sigma(\bA)\cup\sigma(\bB)$.

\subsection{Complex eigenvalues}
A particular class of generic matrices that one may wish to consider applying the methods presented here to are matrices whose eigenvalues lie near, but not exactly on $\Sigma$. All theorems in \cite{AkhIter} apply to this class. We include some of these results here without proof.

The primary difference between eigenvalues on $\Sigma$ and eigenvalues near $\Sigma$ is that the convergence rate of our algorithms now depends on the ``worst offender'' eigenvalue of the generic matrix $\bM$, as measured by $\re \mathfrak g$. Specifically, the convergence rate of Algorithm~\ref{alg:akh_func} is governed by the function
\begin{align*}
    \nu(z;\bM)= \max_{\lambda \in \sigma(\bM)}\re\mathfrak g(\lambda) - \re \mathfrak g(z).
\end{align*}
In this setting, the conclusion of Theorem~\ref{thm:conv_rate} is modified to be
\begin{equation}\label{eq:inexact_eigs}
\left\|f(\bM)-\bF_k\right\|_2\leq C'M\|\bV\|_2\left\|\bV^{-1}\right\|_2\frac{\ex^{k \nu(z_*;\bM)}}{1 - \ex^{\nu(z_*;\bM)}},
\end{equation}
where $z_*$ is any point on the level curve $\Gamma_\varrho = \partial B_\varrho$ \eqref{eq:Bvarrho}, interior to which $f$ is analytic. See Lemma~\ref{lem:bernstein}. In particular, the iteration converges if and only if $\nu(z_*;\bM)<0$.

When $f$ is the sign function, this implies that there exists some constant $D_{\bM}>0$ such that
\begin{equation*}
\left\|\sign(\bM)-\bF_k\right\|_2\leq D_{\bM}\frac{\ex^{k \nu(z_*;\bM)}}{1 - \ex^{\nu(z_*;\bM)}}.
\end{equation*}
Applying the hypothesized bound $D_{\bM}\leq 10\ell$ of Section~\ref{sect:conv} for generic $\bM\in\compl^{\ell\times\ell}$, given a tolerance $\epsilon>0$ and denoting $\varrho_\bM=\ex^{\nu(z_*;\bM)}$, we require
\begin{equation*}
k=\left\lceil-\log_{\varrho_\bH}\frac{\epsilon(1-\varrho_\bH^{-1})}{5(n+m)}\right\rceil,
\end{equation*}
iterations to ensure that the iterate $\bX_k$ of Method 1 (Algorithms~\ref{alg:sylv_decoupled}~or~\ref{alg:sylv_low_rank}) satisfies
\begin{equation*}
\|\bX_k-\bX_*\|_2\leq\epsilon,
\end{equation*}
where $\bX_*$ denotes the solution of the Sylvester equation \eqref{eq:Genmat}. In Figure~\ref{fig:eig}, we plot the convergence of Algorithm~\ref{alg:sylv_low_rank} applied to coefficient matrices $\bA$ and $\bB$ with eigenvalues near $[0.5,6]$ and $[-2,-0.5]$, respectively. The iteration converges to the solution $\bX_*$ at the slower rate $\varrho_\bH^{-1}$ governed by the ``worst offender'' eigenvalue.
\begin{figure}
	\centering
	\begin{subfigure}{0.495\linewidth}
		\centering
		\includegraphics[width=\linewidth]{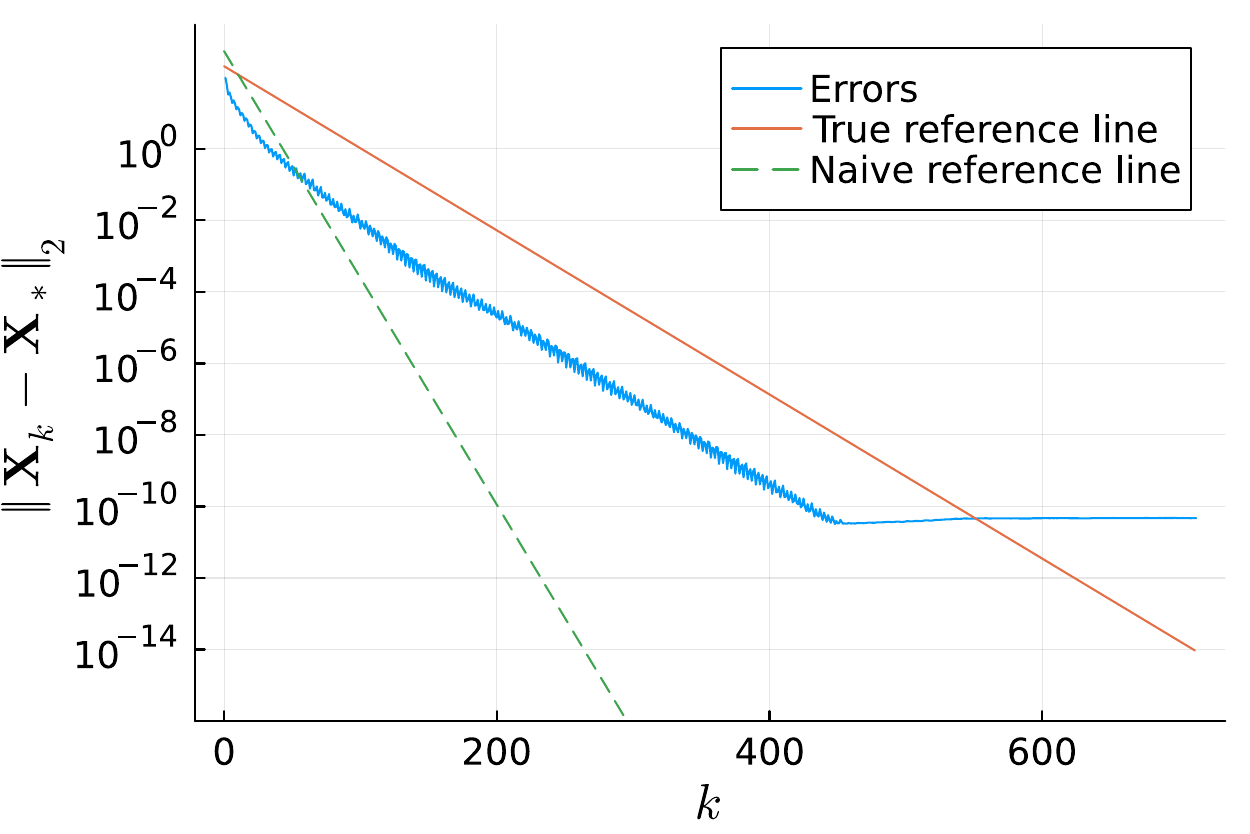}
	\end{subfigure}
	\begin{subfigure}{0.495\linewidth}
		\centering
		\includegraphics[width=\linewidth]{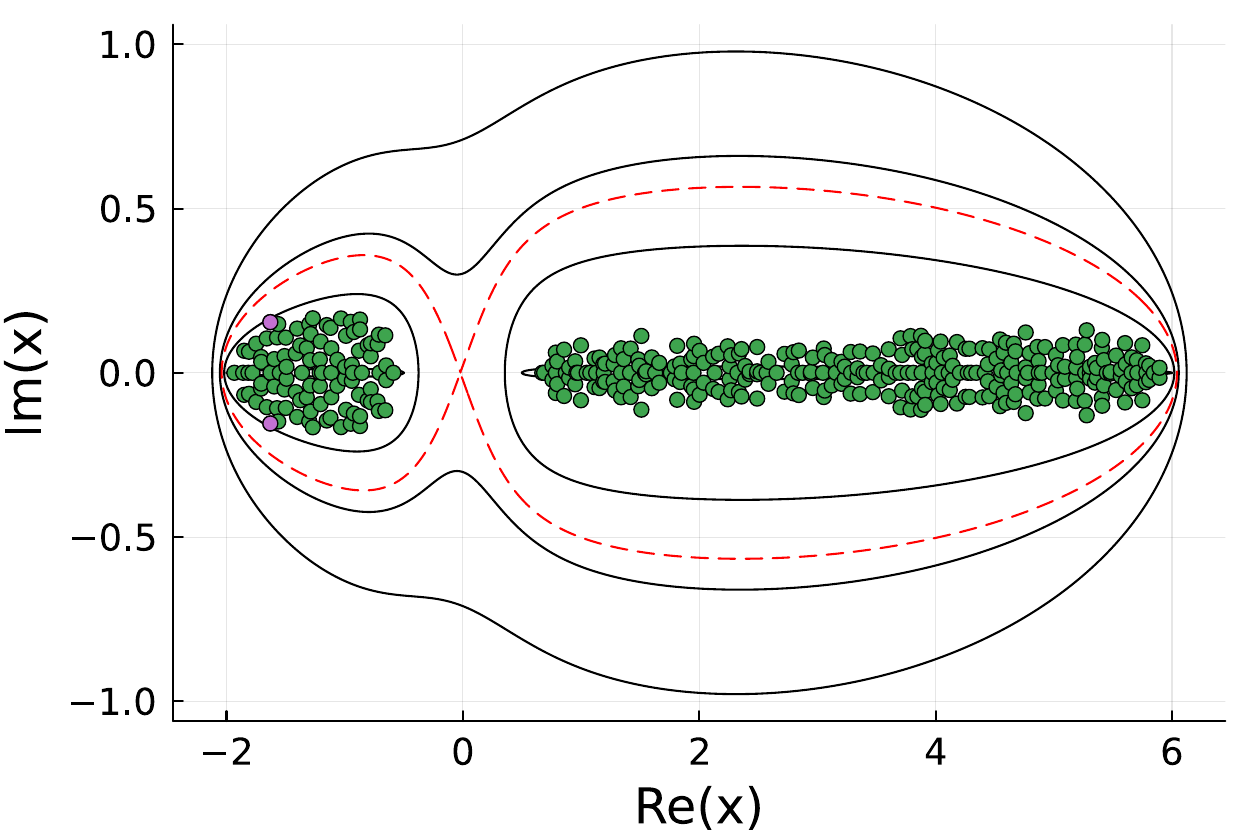}
	\end{subfigure}
\caption{2-norm error in solving~\eqref{eq:Genmat} for $\bX_k$ against the true solution $\bX_*$ (computed via $\texttt{sylvester}$ in \Julia) at each iteration plotted against the convergence rate heuristic \eqref{eq:conv_heur} (with $D_\bH$ replaced by the hypothesized upper bound $10(n+m)$) corresponding to $\varrho$ (dashed green line) and $\varrho_\bH$ (orange line) where $\bA\in\real^{300\times300}$ has eigenvalues near $[0.5,6]$, $\bB\in\real^{100\times100}$ has eigenvalues near $[-2,-0.5]$, and $\bC$ is rank 2 (left) and eigenvalues of $\bH$ superimposed over the generalized Bernstein ellipses of Figure~\ref{fig:bernstein} with the ``worst offender'' colored in purple (right). Since all eigenvalues are contained within the red dashed level curve, the method still converges.}
\label{fig:eig}
\end{figure}

\subsection{Further quadrature error analysis}

To incorporate errors from quadrature rules in the context of Algorithm~\ref{alg:akh_func}, the estimate \eqref{eq:inexact_eigs} can be inserted into \eqref{eq:quad_inac}. For an additional estimate, we recall \eqref{eq:rhom} and restate \cite[Theorem 4.9]{AkhIter}.

\begin{theorem}\label{thm:func_error}
    Let $\Sigma=\bigcup_{j=1}^{g+1}[\beta_j,\gamma_j]\subset\real$ be a union of disjoint intervals and $w$ be a weight function of the form \eqref{eq:gen_weight} and denote the corresponding orthonormal polynomials by $\left(p_j(x)\right)_{j=0}^\infty$. Suppose that $\Gamma,f$ are as in Definition~\ref{def:fan} and $\{z_j\}_{j=1}^m$, $\{w_j\}_{j=1}^m$ are quadrature nodes and weights for $\Gamma$, respectively. Suppose further that $\bM\in\compl^{n\times n}$ is generic and satisfies $\nu(z_j;\bM) < 0$ for $j = 1,2,\ldots,m$, $|f(z)|\leq M$ for $z\in\Gamma$, and $\sum_{j=1}^{m}|w_j|\leq 2\pi L$. Then, there exists $c'>0$ such that
	\begin{equation*}
	\left\|f(\bM)-f_{k,m}(\bM)\right\|_2\leq \|\bV\|_2\left\|\bV^{-1}\right\|_2\left(\frac{1}{2\pi}\max_{\lambda \in \sigma(\bM)}|\rho_m(\lambda)|+c'LM\max_{j\in\{1,\ldots,m\}}\frac{\ex^{k\nu(z_j;\bM)}}{1-\ex^{\nu(z_j;\bM)}}\right),
	\end{equation*}
	where $\bM$ is diagonalized as $\bM=\bV\bLambda\bV^{-1}$.
\end{theorem}
Define
\begin{align*}
    \nu(\bM) = \max_{\lambda \in \sigma(\bM)} \re \mathfrak g(\lambda).
\end{align*}
Supposing that $\Gamma \subset B_\varrho$, combining \eqref{eq:inexact_eigs} with Theorem~\ref{thm:func_error} yields that 
\begin{align*}
    \frac{\|f(\bM) - f_{k,m}(\bM)\|_2}{\|\bV\|_2 \|\bV^{-1}\|_2} &\leq \min\left\{  \|\rho_m\|_\infty \sum_{\ell = 0}^{k-1} \|p_\ell (\bLambda)\|_2 + C' M \frac{\ex^{k \nu(z_*;\bM)}}{1 - \ex^{\nu(z_*;\bM)}},\right. \\
    &\left. \qquad \frac{1}{2\pi}\|\rho_m(\bLambda)\|_2 +c'LM\max_{j\in\{1,\ldots,m\}}\frac{\ex^{k\nu(z_j;\bM)}}{1-\ex^{\nu(z_j;\bM)}}\right\}\\
    &\leq \min\left\{  d \|\rho_m\|_\infty \sum_{\ell=0}^{k-1} \ex^{\ell\nu(\bM)} + C' M \frac{\ex^{k \nu(z_*;\bM)}}{1 - \ex^{\nu(z_*;\bM)}},\right. \\
    &\left. \qquad \frac{1}{2\pi}\|\rho_m(\bLambda)\|_2 +c'LM\max_{j\in\{1,\ldots,m\}}\frac{\ex^{k\nu(z_j;\bM)}}{1-\ex^{\nu(z_j;\bM)}}\right\},
\end{align*}
for some constant $d>0$.

Empirically, we find that the first term in the minimum typically gives a tighter bound; however, the latter term reflects that errors remain small as the iteration proceeds and $\sum_{\ell = 0}^{k-1} \|p_\ell(\bLambda)\|_2$ becomes large. Furthermore, if the eigenvalues of $\bM$ are not contained in $\Sigma$, the latter term will still be small, giving useful bounds.  The first term will also likely be small, but as $\nu(\bM)>0$, the sum grows exponentially with respect to $k$. In this case, the balance, i.e., which term gives a better bound in a given situation, will change.

\bibliographystyle{plain}
\bibliography{refs}

\end{document}